\documentclass[12pt,oneside,a4paper]{book}

\usepackage{amsmath,amssymb,amstext,amsthm,enumerate,fancyvrb,fixltx2e,hyperref,lipsum,mathrsfs,mathtools,setspace,tikz}

\usepackage{color}
\usepackage{graphicx}
\usepackage{setspace}
\usepackage{url}
\usepackage{verbatim}
\usepackage{fancyhdr}
\usepackage[comma,numbers]{natbib}
\usepackage[tiny,md,sc]{titlesec}
\usepackage{titling}
\usepackage[labelfont={bf,sf,footnotesize,singlespacing},
textfont={sf,footnotesize,singlespacing},
justification={justified,RaggedRight},
singlelinecheck=false,
margin=0pt,
figurewithin=chapter,
tablewithin=chapter]{caption}

\usepackage{amsfonts}
\usepackage{amsmath}
\usepackage{amssymb}
\usepackage[a4paper]{geometry}
\usepackage{mathrsfs}
\usepackage[english]{babel}
\usepackage{xcolor}
\usepackage{cleveref}
\usepackage{footnote}
\usepackage[shortlabels]{enumitem}        
\newlist{abbrv}{itemize}{1}

\setlist[abbrv,1]{label=,labelwidth=1.6in,align=parleft,itemsep=0.02\baselineskip,leftmargin=!}
\setlist[enumerate,1]{leftmargin=1.9cm}

\usepackage{bm}
\usepackage{pst-node}
\usepackage{tikz-cd}
\usepackage[utf8]{inputenc}
\usepackage{mathtools}
\usepackage{indentfirst}

\usepackage[]{hyperref} 

\numberwithin{equation}{section}     
\renewcommand*{\theequation}{%
	\ifnum\value{section}=0 %
	\thechapter
	\else
	\thesection
	\fi
	.\arabic{equation}%
}
\theoremstyle{plain}
\newtheorem{theorem}{Theorem}[section]
\renewcommand*{\theequation}{%
	\ifnum\value{section}=0 %
	\thechapter
	\else
	\thesection
	\fi
	.\arabic{equation}%
}
\newtheorem{proposition}[theorem]{Proposition}
\newtheorem{corollary}[theorem]{Corollary}
\newtheorem{lemma}[theorem]{Lemma}
\newtheorem{cor}[theorem]{Corollary}

\newtheorem{conj}[theorem]{Conjecture}

\theoremstyle{definition}
\newtheorem{definition}[theorem]{Definition}
\newtheorem{remark}[theorem]{Remark}
\newtheorem{example}[theorem]{Example}

\definecolor{ImperialColor}{rgb}{0.000, 0.243, 0.455}
\definecolor{chaptergrey}{rgb}{0.000, 0.243, 0.455}

\newcommand{\mytitlepage}{
	\doublespacing
	\title{Wave equations on graded groups and hypoelliptic Gevrey spaces}
	\author{Chiara Alba Taranto}
	\begin{titlepage}
		\begin{center}
			\null\vfill
			{\LARGE\textcolor{ImperialColor}{\textsc{\thetitle}}}
			\\
			\vfill
			{\sc \footnotesize
				{\sc a thesis presented for the degree of}\\
				Doctor of Philosophy of Imperial College London\\
				{\sc and the}\\
				Diploma of Imperial College\\
				{\sc by} \\
			}
			{\large\sc{\theauthor}}
			\vfill\vfill\vfill
			{\sc \footnotesize{
					Department of Mathematics\\
					Imperial College \\
					180 Queen's Gate, London SW7 2BZ}
				\vfill
				{{October 2017}}
			}
		\end{center}
	\end{titlepage}

	\setstretch{1.2}
	
	\addtocounter{page}{1}
	
	\newpage\thispagestyle{empty}~
	\newpage
	\hphantom{1}
	\vfill
	\noindent I certify that this thesis, and the research to which it refers,
	are the product of my own work, and that any ideas or quotations from the
	work of other people, published or otherwise, are fully acknowledged in accordance
	with the standard referencing practices of the discipline.
	\vfill
	
	\chapter*{\sc Copyright}
	
	The copyright of this thesis rests with the author and is made available under a Creative Commons Attribution Non-Commercial No Derivatives licence. Researchers are free to copy, distribute or transmit the thesis on the condition that they attribute it, that they do not use it for commercial purposes and that they do not alter, transform or build upon it. For any reuse or redistribution, researchers must make clear to others the licence terms of this work.
	\pagestyle{plain}
}

\newcommand{\abstractpage}{
	\newpage
	\pagestyle{fancy}
	\renewcommand{\headrulewidth}{0.0pt}
	\vspace*{35pt}
	\begin{center}
		\Large 
		\Huge \textcolor{ImperialColor}{\scshape Abstract} \normalsize \\ \rm
	\end{center}
	The overall goal of this dissertation is to investigate certain classical results from harmonic analysis, replacing the Euclidean setting, the abelian structure and the elliptic Laplace operator with a non-commutative environment and hypoelliptic operators.
	
	More specifically, we consider wave equations for hypoelliptic homogeneous left-invariant operators on graded Lie groups with time-dependent non-negative propagation speeds that are H\"older-regular or even more so. The corresponding Euclidean problem has been extensively studied in the `80s and some additional results have been recently obtained by Garetto and Ruzhansky in the case of a compact Lie group. We establish sharp well-posedness results in the spirit of the classical result by Colombini, De Giorgi and Spagnolo. In this investigation, a structure reminiscent of Gevrey regularity appears, inspiring deeper investigation of certain classes of functions and a comparison with Gevrey classes.
	
	In the latter part of this thesis we discuss such Gevrey spaces associated to the sums of squares of vector fields satisfying the H\"ormander condition on manifolds. This provides a deeper understanding of the Gevrey hypoellipticity of sub-Laplacians. It is known that if $\slp$ is a Laplacian on a closed manifold $M$ then the standard Gevrey space $\gamma^s$ on $M$ defined in local coordinates can be characterised by the condition that $\|e^{D\slp^{\frac{1}{2s}}}\phi\|_{L^2(M)}<\infty$ for some $D>0$. The aim in this part is to discuss the conjecture that a similar characterisation holds true if $\slp$ is H\"ormander's sum of squares of vector fields, with a sub-Laplacian version of the Gevrey spaces involving these vector fields only. We prove this conjecture in one direction, while in the other we show it holds for sub-Laplacians on $SU(2)$ and on the Heisenberg group $\mathbb{H}_n$. 
	\vspace*{\fill}
	\newpage \lhead{} \rhead{}
	\cfoot{\thepage}
}

\newcommand{\dedicationpage}{
	\newpage \thispagestyle{empty} \vspace*{\fill}
	\it \noindent 
	To my grandparents,\\
	Chiara, Ciro, Enza, Salvatore.
	\vspace*{\fill} \newpage \rm
}

\newcommand{\acknowledgments}{
	\chapter*{\sc\textit{Acknowledgements}}
	\noindent
	I would like to express my sincere gratitude to my supervisor Professor Michael Ruzhansky for believing in me and giving me the opportunity to pursue my Ph.D. at Imperial College London. I thank him for his patience and immense knowledge. I am extremely grateful to Dr. Veronique Fischer for all the time she dedicated to me. Without her precious support I would have not appreciated the deep meaning and the beauty of this piece of Mathematics. Her enthusiasm and dedication have been the best example to follow, and her insightful comments and guidance helped me in both researching and writing. I also offer my gratitude to all my lecturers at \textit{Università degli Studi di Bari Aldo Moro}, particularly to my advisor Dr. Sandra Lucente.  
	
	I also want to thank all the people at Imperial College London, in particular our postgraduate administrator Mr. Anderson Santos, always ready to assist all the students, and all the Ph.D. students in the Maths department: clever, committed and interested in so many different activities and aspects of life. They made the campus an extremely inspiring place.
	
	I wish to thank all my friends: the new ones who allowed me to discover and enjoy the beauty of the `Big Smoke', the old ones who have always found a way to be present in my life independently of the physical distance. In particular my childhood friend Irene, my little cousin Rossana and my `Vitamin A' Alfredo. 
	
	I am deeply grateful to Professors Alberto Parmeggiani and Boguslaw  Zegarlinski, my examiners, for their invaluable improvements to my thesis and for inspiring me with further possible directions to develop this work.
	
	Last but not least, I would like to thank my family: my parents for supporting me spiritually throughout my studies and my life in general, for helping me to believe in my own skills and become a stronger and independent person, for all the times they met me at the airport with the biggest smiles I have ever seen.
	My brother, the most talented and lively person that I know. He has always been and always will be an example for me. My beloved boyfriend, Freddy. This last year would have not been the same if he had not been with me. His support and love, his perfectionism, his patience and joy helped me to overtake all the difficulties in this crucial and intense year. And of course, I thank him for all the mathematical discussions, for constantly encouraging me and for proof-reading my thesis.
	\vspace*{\fill} \newpage
}

\usepackage[titletoc]{appendix}
\usepackage[titles]{tocloft}
\hypersetup{colorlinks,citecolor=ImperialColor,filecolor=black,linkcolor=black,urlcolor=ImperialColor}

\DeclareMathOperator*{\esssup}{ess\,sup}
\DeclareMathOperator{\Hom}{Hom}
\DeclareMathOperator{\Tr}{Tr}
\DeclareMathOperator{\singsupp}{sing\,supp}
\DeclareMathOperator{\Id}{Id}
\DeclareMathOperator{\SPAN}{span}
\DeclareMathOperator{\Dom}{Dom}
\DeclareMathOperator{\Op}{Op}
\DeclareMathOperator{\OP}{op}

\usepackage[newcentury]{quotchap}

\newcommand{\C}{{\mathbb C}}
\newcommand{\Z}{{\mathbb Z}}
\newcommand{\N}{{\mathbb N}}
\newcommand{\R}{{\mathbb R}}
\newcommand{\h}{{\mathbb H}}
\newcommand{\Rn}{{\R}^n}
\newcommand{\slp}{{\mathcal L}}
\newcommand{\HS}{{\mathtt{HS}}}
\newcommand{\Gh}{{\widehat{G}}}
\newcommand{\SU}{{\rm SU(2)}}
\newcommand{\Rep}{{\rm Re}}
\newcommand{\Imp}{{\rm Im}}

\begin{document}

\mytitlepage
\setstretch{1.2}
\abstractpage
\dedicationpage
\acknowledgments
\tableofcontents

\mainmatter
\addtocounter{page}{10}
\setcounter{chapter}{0}

\chapter{Introduction}\label{CP:1}

At the beginning of the 20th century a new class of second order partial differential equations with non negative and degenerate characteristic form started to attract the attentions of several mathematicians. These equations initially arose from the mathematical modelling of problems from theoretical physics and diffusion processes\footnote{A diffusion process is a solution to a stochastic differential equation.}. The first derivation is due to the mathematician Andrey Kolmogorov \cite{K1931,K1934}. Indeed, during his studies related to stochastic processes he found the following set of equations, pivotal in the field. They were written in the following form by Lars H\"ormander \cite{H1967}:
\begin{equation}\label{Kolmogorov}
-\frac{\partial u}{\partial x_0}+\sum_{j,k=1}^{n}a_{j,k}\frac{\partial^2 u}{\partial x_j\partial x_k}+\sum_{j,k=1}^nb_{j,k}x_j\frac{\partial u}{\partial x_k}+cu=f.
\end{equation}
Here $a_{j,k},b_{j,k}$ and $c$ are real constants and the matrix $A=(a_{j,k})$ is symmetric and positive semi-definite.
The same equations are also known as Fokker--Plank equations, as they have been independently derived by the physicists Adriaan Fokker and Max Plank to describe the time evolution of the probability density function of the velocity of a particle in Brownian motion. Their version turned out to be a special case of the Kolmogorov equations, explaining diffusion phenomena from a probabilistic point of view.
A peculiar aspect of this class of equations is that, despite the definiteness of $A$, in general they are not elliptic, since they have multiple characteristics, but they are hypoelliptic, defined as follows. 
\begin{definition}[Hypoelliptic operator]
	Given an open set $\Omega\subset\Rn$, a partial differential operator $P$ with coefficients smooth in $\Omega$ is hypoelliptic if $Pu\in\mathcal{C}^{\infty}(S)$ implies $u\in\mathcal{C}^{\infty}(S)$ for every distribution $u\in\mathcal {D}^{'}(\Omega)$ and every open subset $S$ of $\Omega$.
\end{definition}

The general quest of this work is to investigate certain classical results from harmonic analysis in which the Euclidean setting and the elliptic Laplacian are replaced by a non-commutative environment and hypoelliptic operators. As we have begun and will shortly continue to describe, in addition to the inherent desire of the mathematician to look for generalisations of all types, the interest in this direction of research comes also from both theoretical and applied fields, such as calculus of variations or robotics where hypoelliptic operators on non-commutative, or sub-Riemannian settings appear, e.g., \cite{CS2014,M1993}.

Examples of hypoelliptic operators include the heat operator or any elliptic operator. As for elliptic operators, sufficient and necessary conditions were initially determined for operators with constant coefficients, and subsequently extended to certain classes of operators with variable coefficients, see, e.g., \cite{H1990}. However, these classes do not include the Kolmogorov-type operators. 

We can observe from its definition that a sufficient condition to deduce hypoellipticity is the existence of a `smooth fundamental solution' in the sense of \cite[Theorem 11.1.1]{H1990}. In fact, in his work in 1934, Kolmogorov constructed a fundamental solution for \eqref{Kolmogorov} that is a smooth function off the diagonal. From this, one can immediately obtain the hypoellipticity of \eqref{Kolmogorov}. 

Notwithstanding this, it is extremely tricky to find fundamental solutions. Certain properties of the Kolmogorov operators inspired H\"ormander, and he made the important contribution of a collection of sufficient conditions to deduce hypoellipticity of a significant class of operators and link it to geometric properties. Momentarily we state the axial theorem by H\"ormander, in a style similar to the original publication \cite[Theorem 1.1]{H1967}\footnote{In \cite{H1967}, H\"ormander almost completely characterized hypoelliptic second order partial differential operators with $\mathcal{C}^{\infty}$ coefficients, in the sense that it can be seen that the sufficient condition of Theorem \ref{THM:Hormander} for hypoellipticity is also essentially necessary. In addition he extended a necessary condition for hypoellipticity to operators with variable coefficients, since he showed that the principal symbol of a hypoelliptic operator must be semi definite, i.e., all eigenvalues are non-negative.}. The idea underlying the theorem is that if given vector fields do not span the whole tangent space at every point, it is still possible to prove that certain operators are hypoelliptic under an appropriate geometric hypothesis. Such a geometric hypothesis allows one to `generate' the missing directions as commutators of the given ones. 

\begin{theorem}[H\"ormander's theorem]\label{THM:Hormander}
	Let an operator $P$ be written in the form
	\begin{align}\label{EQ:HormOp}
	P=\sum_{j=1}^{r}X_j^2+X_0+c,
	\end{align}
	where $X_0,\dots,X_r$ denote real-valued first order homogeneous differential operators in an open set $\Omega\subset\Rn$ with $\mathcal{C}^\infty$-coefficients and $c\in\mathcal{C}^\infty(\Omega)$. Let us assume that among the operators 
	\begin{align*}
	X_{j_1},[X_{j_1},X_{j_2}],[X_{j_1},[X_{j_2},X_{j_3}]],\dots,[X_{j_1},[X_{j_2},[X_{j_3},\dots]]],\dots,
	\end{align*}
	where $j_i=0,\dots,r$, there exist $n$ which are linearly independent at any given point in $\Omega$. Then it follows that $P$ is hypoelliptic. 
	
	Under these conditions, the differential operators $X_1,\dots,X_r$ are said to be a \textit{H\"ormander system}.
\end{theorem}
This was a milestone result that opened a research field. In particular, the theory of subelliptic second order partial differential equations with semi-definite characteristic form has found much attention. 
The main protagonists in this theory are the sub-Laplacian operators defined as
\begin{align*}
\slp:=-\sum_{j=1}^r X_j^2,
\end{align*}
for a given H\"ormander system $\{X_j\}_{j=1}^r$. It is natural, and even useful in many applications and results \cite{F75,FS-CPAM,S1973}, to extend and investigate these operators and their related properties to a more general setting than the Euclidean one. This is provided by a class of suitable nilpotent non-abelian Lie groups and corresponding Lie algebras admitting a `harmonic analysis' and for which left-invariant, homogeneous H\"ormander operators of degree two exist; in Chapter \ref{CP:2} we offer more details. In this non commutative environment, the sub-Laplacians play a crucial r\^ole, comparable to that of the Laplacian for elliptic equations. Indeed sub-Laplacians satisfy symmetry properties analogous to those of Laplacians, in the sense that
\begin{itemize}
	\item they are left-invariant with respect to the group law;
	\item they have degree $2$ with respect to the dilation automorphisms on the group.
\end{itemize} 

Remarkable contributions to this field were given in the `70s and `80s by Folland, Rothschild and Stein who developed and applied to sub-Laplacians the singular integral theory on nilpotent Lie groups, and over two decades later by Ponge \cite{P2008}. The paper by Folland \cite{F1977} presents an outline of the developments of the theory for the aforementioned hypoelliptic differential operators on non-abelian groups. We mention here a short selection from the long list of works that have been crucial to the development of the theory, referring the interested reader to these and the diverse references therein.

In 1974 Folland and Stein investigated operators from complex analysis in the context of subelliptic, in particular hypoelliptic, operators \cite{FS-CPAM}. More precisely, they considered the \textit{tangential Cauchy--Riemann complex} or $\overline{\partial}_b$ \textit{complex} and the associated `Kohn Laplacian' $\square_b$ on the complex sphere. Unlike the ordinary Laplacian, $\square_b$ is not elliptic, but subelliptic and on the complex sphere it can be proved that it is hypoelliptic. They argued that the case of the complex sphere can be analysed locally approximating the sphere with the Heisenberg group\footnote{On the Heisenberg group they found an integral operator which inverts $\square_b$ exactly.} $\h_n$, (see Section \ref{Heisenberg} for an introduction to $\h_n$), where the operator $\square_b$ is easily computed. In this setting, they could extend previous results about a fundamental solution for the hypoelliptic sub-Laplacian on $\h_n$ from \cite{F1973} and deduce several estimates for $\square_b$ in terms of $L^p$ Sobolev spaces.  

The following year Folland published a thorough study \cite{F75} of sub-Laplacians on \textit{stratified Lie groups}. In fact he constructed a general theory of subelliptic regularity on these groups: for example he extended several basic theorems from the Euclidean theory of differentiability to this non-abelian setting. Among other things, he obtained interpolation properties, boundedness of singular integrals and $L^p$ estimates.  

Subsequently, in 1976  Rothschild and Stein generalised the results by Folland and Stein by proving the so called \textit{Rothschild--Stein lifting theorem} through an inductive argument. Indeed, in \cite{Rothschild-Stein:AM-1976} they developed a theory to approximate more general operators by operators on nilpotent groups and, hence, extend the regularity theorems to more general cases. The works by Folland and Stein, and later Rothschild and Stein, extended the breadth of applicability of the subelliptic theory of differentiability to a bigger class of second order hypoelliptic partial differential equations.

To understand the underlying idea of the Rothschild--Stein lifting theorem, which allows us to apply the subelliptic regularity theory in a more general case, let us consider a `classical' example from both \cite{Rothschild-Stein:AM-1976} and \cite{F1977}. In general, given an operator defined as a sum of squares of vector fields, as in \eqref{EQ:HormOp}, and which satisfies the H\"ormander condition, we can immediately claim that the latter is hypoelliptic by H\"ormander's theorem \ref{THM:Hormander}. Nevertheless the application of H\"ormander's theorem does not yield any interesting or sharp results about its regularity properties. It is here that the Rothschild and Stein theorem intervenes. Suppose that we want to study the Grushin operator, on $\R^2$,
\begin{align*}
H:=\partial_x^2+(x\partial y)^2.
\end{align*}
The vector fields generating $H$, namely $X_1:=\partial_x$ and $X_2:=x\partial_y$, lose their linear independence along the line $x=0$.  Nevertheless at $x=0$ the commutator  
\begin{align*}
[X_1,X_2]=\partial_y+x\partial_x\partial_y-x\partial_y\partial_x=\partial_y
\end{align*}
can be utilised, so that the set $\{X_1,[X_1,X_2]\}$ spans $\R^2$ everywhere. Therefore the H\"ormander theorem applies and allows us to conclude that the Grushin operator is hypoelliptic. However it is impossible to have a two-dimensional non-abelian nilpotent Lie group. The cunning---and perhaps counter-intuitive---idea of Rothschild and Stein was to \textit{add certain extra variables} and therefore `lift' the original operator to a higher dimensional manifold endowed with the `required' geometric structure. Specifically, in the context of the Grushin operator, we can define the vector fields
\begin{align*}
\tilde{X}_1:=\partial_x,\quad\text{and}\quad\tilde{X}_2:=\partial_t+x\partial_y
\end{align*}
in $\R^3$ and then work in the corresponding three-dimensional group. In fact $\tilde{X}_1$ and $\tilde{X}_2$, together with their commutator $[\tilde{X}_1,\tilde{X}_2]$, give rise to the three-dimensional Lie algebra associated to the non-abelian three-dimensional Heisenberg group $\mathbb{H}_1$. Now we can study the hypoelliptic operator $\tilde{H}:=\tilde{X}_1^2+ \tilde{X}_2^2$ by using the non-abelian nilpotent stratified group $\mathbb{H}_1$, and afterwards transfer the results back to the original Grushin operator $H$.

We conclude this partial historical review by mentioning Folland's words from \cite{F1977}:

``[...]From this work\footnote{Folland himself refers to \cite{Rothschild-Stein:AM-1976}.} emerged the following general philosophy concerning the theory of differential operators constructed from vector fields whose commutators play an essential r\^ole. A natural class of models for non commuting vector fields is the left-invariant vector fields on non-Abelian Lie groups. To study a set of non commuting vector fields, then, one should find a group whose Lie algebra mimics the structure of the original vector fields in a suitable sense, and which admits a `harmonic analysis' that will yield results similar to the classical Euclidean case. The nilpotent groups with dilations (or the more special stratified groups) fulfil the second requirement, and they seem to be general enough to fulfil the first in all known cases." 

Another aspect of harmonic analysis on groups that we have not mentioned so far is the \textit{representation theory} of groups, that we will introduce in Chapter \ref{CP:2}. Although representations do not come into play in the aforementioned works by Folland, Stein and Rothschild, several results can be achieved using this `unwieldy'\footnote{According to \cite{F1977}.} tool. For example, Rockland proved a hypoellipticity theorem \cite{Rockland}  in the context of representations of the Heisenberg group $\h_n$. This is an analogue of the fact that a homogeneous differential operator on $\Rn$ with constant coefficients is hypoelliptic if and only if it is elliptic. For our purposes, it is important to observe that representations of Lie groups and infinitesimal representations of the corresponding Lie algebra have been used to develop a non-commutative Fourier analysis on groups, see \cite{Fol,F1994,T1986} or the more recent \cite{FR2016,RT2009}. This aspect will have a relevant r\^ole in the second part of this dissertation, so we postpone to Chapter \ref{CP:2} and \ref{CP:3} a more detailed presentation of the topic.

Since their introduction, these partial differential equations (pdes) with multiple characteristics have frequently appeared in several problems arising from both theoretical and applied fields. These include the geometric theory of several complex variables, calculus of variations and also, with regard to applied sciences, mathematical models in finance, in human visions and in robotics. A more comprehensive historical overview and list of references can be found in the monograph by Bonfiglioli, Lanconelli and Uguzzoni \cite{BLU}. As already observed for the Kolmogorov equation, one of the peculiarities of these `new' pdes is that they might fail to be elliptic at every point, explaining why they are known as \textit{subelliptic equations}. This lack of ellipticity means that the quadratic form associated with such operators is non-coercive, and therefore there may be a loss of control in some directions in the tangent space.

Nevertheless, it is possible to associate to these equations a peculiar underlying geometry, known as \textit{sub-Riemannian geometry} that allows one to recover the missing directions by an appropriate choice of commutators of vector fields. As we will carefully show in the next Chapter, a natural setting to study these problems is given by graded Lie groups\footnote{See Definition \ref{defGraded}.}, which possess an invariant measure and an intrinsic family of dilations making them a natural habitat for extensions of classical results in Euclidean harmonic analysis; see \cite{FS}. 

Several problems and results of Euclidean harmonic analysis are related to the elliptic Laplace operator, its spectral theory and functional calculus. It is often possible to replace the Laplace operator with a more general elliptic operator. However, in various contexts, particularly on non-commutative Lie groups and sub-Riemannian manifolds, the natural substitute for the Laplace operator need not be elliptic, and may be only subelliptic. Here new and interesting phenomena arise and the few known results typically involve a combination of tools from different areas of mathematics: evolution equations, representation theory and differential geometry.

We begin our analysis by explaining the idea of well-posedness for a generalisation of the time-dependent wave equation on groups.

\section{Weakly hyperbolic equations}
The literature on the classification of pdes is not always consistent. For this reason, we will specify here what is meant in this dissertation by a (weakly or strictly) hyperbolic equation. 

According to some authors, such as H\"ormander \cite{H1990}, the hyperbolicity of a given operator can be related to the possibility of admitting a well-posed initial value problem for arbitrary smooth data. The drawback of this interpretation is that it is not easy to check. Nevertheless, it is only the principal part of the operator, i.e. the sum of the terms of the operator involving the highest order of derivation, that affects the well-posedness of the corresponding initial value problem. On this basis, we provide below sufficient conditions on the principal part to guarantee  the hyperbolicity of the operator.

Let us consider a linear partial differential operator of order $m$ with positive time-dependent coefficients of the form
\begin{align*}
Pu(t,x)=\partial_t^m u(t,x)+ \sum_{|\alpha|\leq m}a_\alpha(t)\partial_x^\alpha u(t,x), 
\end{align*}
with $x$ in an $n$-dimensional manifold and $t$ in a finite interval $[0,\tau]$. We introduce an invariant quantity associated with the given differential operator: the principal symbol of the operator $P$ is the homogeneous polynomial of degree $m$ given by
\begin{align*}
p(\omega,\xi)=\omega^m+\sum_{|\alpha|= m }a_\alpha(t)\xi^\alpha.
\end{align*}
If the polynomial $p(\omega,\xi)$ admits $m$ real roots then the operator $P$ is said to be \textit{hyperbolic}. In particular, if 
\begin{itemize}
	\item $p(\omega,\xi)$ admits $m$ \textit{distinct} real roots, then $P$ is \textit{strictly hyperbolic},
	\item the real roots of $p(\omega,\xi)$ have non-trivial multiplicities, then $P$ is \textit{weakly hyperbolic}.
\end{itemize} 
The theory for strictly hyperbolic operators is well developed. In fact there are well-known techniques that allow one to determine the well-posedness of the initial value problems, to construct a representation formula for solutions, or to study the regularity of the operators. For details the interested reader can consult, e.g., the four volumes by H\"ormander \cite{H1990} or the monograph by Duistermaat \cite{D1996} where the Fourier integral theory is developed and applied to strictly hyperbolic operators. 

On the contrary weakly hyperbolic equations have been a challenging problem for a long time and there are no known general techniques to establish necessary and sufficient conditions for the well-posedness. The lack of strict hyperbolicity, due to the fact of dealing with degenerate-elliptic operators, means that Fourier integral operator techniques are no longer available, see \cite{Muller-Stein:Lp-wave-Heis}.  In the first Subsection we consider an important example of a weakly hyperbolic equation that has attracted and continues to attract the interest of many mathematicians. In the following Subsection we outline the ideas behind the definition of Gevrey spaces. In both we present the outline of the contribution made to these topics in this dissertation. In the final Subsection of this Chapter we give an overview of the structure of the material that follows, and offer an index of the notation implemented.

\subsection{The time-dependent wave equation}\label{SEC:intWE}

Let us consider the initial value problem for the wave equation on $\Rn$ with a time-dependent coefficient $a(t)$ , written
\begin{equation}\label{EQ:weRn}
\left\{\begin{array}{rclc}
\partial ^2_{t}u(t,x)+a(t) \Delta u(t,x)	&=&0,			&(t,x)\in [0,\tau]\times\Rn,\\
u(x,0)	&=&\varphi(x),	&x\in\Rn,\\
\partial _t u(x,0)	&=&\psi(x),		&x\in\Rn,
\end{array}\right.
\end{equation}  
where $\Delta=-\sum_{j=1}^n\partial ^2_{x_j}$ is the classical (positive) Laplacian on $\Rn$. In the `70s and `80s this problem has been extensively studied. Oleinik determined a sufficient condition for the coefficient $a(t)$ under which Problem \eqref{EQ:weRn} is well-posed in $\mathcal{C}^\infty$, see \cite{O1970}. If we drop Oleinik's condition, it is possible to construct explicit examples of initial value problems that are not well-posed in $\mathcal{C}^\infty$ and $\mathcal{D}'$, see \cite{CS1982, Colombini-Jannelli-Spagnolo:Annals-low-reg}. From this example, we can already realise that $\mathcal{C}^\infty$ and $\mathcal{D}'$ might not be ideal spaces to look for solutions for Cauchy problems. Indeed, there are certain classes of functions, the so-called Gevrey spaces, that appear naturally when dealing with pdes, in particular with weakly-hyperbolic equations. In fact, in \cite{CDS79} Colombini, De Giorgi and Spagnolo proved sharp well-posedness results for the Cauchy problem \eqref{EQ:weRn} in Gevrey classes.

Of course, the passage to non-Euclidean and non-commutative settings is a natural development in these studies. Several authors analysed similar problems on manifolds and groups, replacing the elliptic Laplacian in the wave equation with more general subelliptic operators. In this way, one faces weakly hyperbolic equations.

To convince ourselves this is the case we can consider a simple case for a positive sub-Laplacian operator $\slp$, that is, the wave equation
\begin{align*}
\partial_t u+a(t)\slp u=0
\end{align*}
on the special unitary group $SU(2)$, isomorphic to the three-dimensional sphere. Its Lie algebra $\mathfrak{su}(2)$ can be endowed with an orthonormal basis of vector fields ${X, Y, T }$, such that $[X, Y ] = T $. Therefore, the sub-Laplacian operator on $SU(2)$ can be defined as follows
\begin{align}\label{EQ:slpSU2}
\slp_{SU(2)} =-\big(X^2 +Y^2\big).
\end{align}
Moreover, the elements of $SU(2)$ can be parametrised using the so-called Euler angles, and the same applies to the vector fields in the Lie algebra $\mathfrak{su}(2)$, see \cite[Definition 11.3.1, Proposition 11.5.9]{RT2009}. In particular, replacing the expression of the vector fields in the Euler angles in \eqref{EQ:slpSU2}, we have a formulation of the sub-Laplacian given by
\begin{align*}
\slp_{SU(2)} =-\Big(\partial_{\theta}^2+\frac{1}{\sin^2\theta}\partial_\phi^2-2\frac{\cos\theta}{\sin^2\theta}\partial_{\theta}\phi\partial_\psi +\big(\frac{1}{\sin^2\theta}-1\big)\partial_{\theta}\psi^2+\frac{\cos\theta}{\sin\theta}\partial_\theta\Big),
\end{align*}
where we consider $(\phi,\theta,\psi)$ in the ranges $0 \leq \phi< 2\pi$, $0 \leq\theta\leq\pi$ and $-2\pi \leq \psi < 2\pi$. If we consider $\eta = (\eta_1,\eta_2,\eta_3)$ to be the dual variable to the Euler angles, the principal symbol of $\slp$ in these coordinates is
\begin{align*}
p(\eta)=-\Big(\eta_2^2+\frac{1}{\sin^2\theta}(\eta_1-\cos\theta \eta_3)^2\Big),
\end{align*}
so that the wave equation is weakly hyperbolic, with non-trivial multiplicities on the set $\eta_1 =\cos\theta\eta_3$ and $\eta_2 =0$.

Of course, several other groups and operators giving rise to weakly hyperbolic equations can be considered. We mention below some relevant works in this spirit.

In 1982, Nachman \cite{Nachman:wave-Heisenberg-CPDE-1982} studied the wave operator
\begin{align}\label{EQ:Nachman}
\partial_t^2-\sum_{j=1}^n \big(X_j^2+Y_j^2\big)+i\alpha T,
\end{align}
on the Heisenberg group $\mathbb{H}_n$, where $\alpha\in \R$ and $\{X_j,Y_j,T\}_{j=1}^n$ is a basis for the left invariant vector fields on $\mathbb{H}_n$. The Heisenberg group is plausibly the simplest example of a non-abelian, non-compact, locally compact, nilpotent, unimodular and stratified Lie group, that we will discuss in Section \ref{Heisenberg}. Nachman considered $|\alpha|<n$ and found, among other things, a formula for the fundamental solution of \eqref{EQ:Nachman} supported in a ``forward light cone"\footnote{This cone is much more complicated than the corresponding cone in the Euclidean case, owing to the underlying sub-Riemannian geometry.}.

In 1984 Helgason in \cite{Helgason:wave-eqns-hom-spaces-1984} considered a wave equation formally identical to \eqref{EQ:weRn} on homogeneous spaces and with $a(t)\equiv 1$, and obtained solution formulae for the Cauchy problem. 

Estimates of $L^p$-type for wave equations on manifolds and Lie groups have been considered, as well. For example, Seeger, Sogge and Stein \cite{SSS1991} and Chen, Fan and Sun \cite{CFS2010}  studied such estimates for strictly hyperbolic wave equations (with the usual Laplacian) using Fourier integral theory. These techniques are unavailable once we replace the elliptic Laplacian with the subelliptic (hypoelliptic) sub-Laplacian. In fact, a different approach has been implemented by M\"uller and Stein in \cite{Muller-Stein:Lp-wave-Heis} to study $L^p$-estimates for the wave equation
\begin{equation*}
\partial ^2_{t}u(t,x)+\slp u(t,x)=0,	\quad		(t,x)\in [0,\tau]\times\mathbb{H}_n,
\end{equation*}  
where $a(t)\equiv 1$ and $\slp=-\sum_{j=1}^n\big(X_j^2+Y_j^2\big)$ is a sub-Laplacian in the sense of H\"ormander. 

Finally, we mention the results by Garetto and Ruzhansky that are particularly relevant to the work presented in this dissertation. Indeed in \cite{GR2015} they studied the well-posedness of the following Cauchy problem for the wave equation for a sum of squares of vector fields on a \textit{compact} Lie group $G$:
\begin{equation}\label{EQ:weG}
\left\{\begin{array}{rclc}
\partial ^2_{t}u(t,x)-a(t)\slp u(t,x)	&=&0,			&(t,x)\in [0,\tau]\times G,\\
u(x,0)	&=&\varphi(x),	&x\in G,\\
\partial _t u(x,0)	&=&\psi(x),		&x\in G,
\end{array}\right.
\end{equation}  
with\footnote{According to their notation the operator $\slp$ is negative.} $\slp=X^2_1+\dots+X_k^2$, where $\{X_1,\dots,X_k\}$ is a family of left-invariant vector fields in the Lie algebra of $G$, not necessarily satisfying the H\"ormander condition. They established well-posedness results for \eqref{EQ:weG} in `Gevrey-type' spaces depending on the time-dependent coefficient $a$ for equations with irregular coefficients and multiple characteristics. Their approach consists in applying the global Fourier analysis on a compact Lie group (developed for example in \cite{T1986,RT2009}) to the Cauchy problem \eqref{EQ:weG}. Hence, the wave equation is transformed into an equation with coefficients depending only on $t$. Using properties of the symbolic calculus of operators and of the global Fourier transform on \textit{compact} Lie groups, the scalar equation in \eqref{EQ:weG} becomes a system of ordinary differential equations that decouples completely. This allows one to study each equation separately, using tools from classical analysis, such as, the reduction to a first-order system and estimates of the energy based either on Gr\"onwall's lemma or on the construction of the so-called (quasi-)symmetriser.\footnote{These arguments will become clearer in Chapter \ref{CP:3} where we will extend their techniques. Moreover we underline here that the work by Garetto and Ruzhansky follows the earlier argument by Kinoshita and Spagnolo \cite{KS2006}.}

Even though in this dissertation we are interested in non-abelian settings and developing global techniques based on group properties (such as the group Fourier transform), it is important to acknowledge that weakly hyperbolic operators with multiple characteristics have been investigated by several other authors using classical Euclidean techniques, wavefront sets and symplectic geometry. For instance, Ivrii and Petkov in \cite{IP1974}, H\"ormander in \cite{H1977} and, more recently, Parenti and Parmeggiani in \cite{PP2004,PP2009} have studied the $\mathcal{C}^\infty$-well-posedness of the Cauchy problems for hyperbolic operators with double characteristics in the Euclidean phase-space perspective, obtaining general results expressed in geometric and Euclidean vocabulary. 

\textbf{Original work.}\\
We consider a twofold generalization of the time-dependent wave equation studied by Garetto and Ruzhansky. Indeed, we study the well-posedness on \textit{graded Lie groups} of the wave equation for hypoelliptic homogeneous left-invariant operators $\mathcal{R}$ with time-dependent H\"older propagation speeds $a(t)$, i.e.
\begin{align*}
\partial ^2_{t}u(t,x)+a(t)\mathcal{R} u(t,x)=0.
\end{align*}
The operators $\mathcal{R}$ are a generalisation of sub-Laplacians first considered by \cite{Rockland} and later called after \textit{Rockland}. As specified in Definition \ref{rockland} they can be of any homogeneous degree $\nu$. Hence, the time-dependent wave equation for the sub-Laplacian on the Heisenberg group (or on a general stratified Lie group), or $p$-evolution equations for higher order operators are examples fitting our setting. Even in these cases our results are new.

The lack of compactness and the presence of more general operators of homogeneous degree $\nu>0$ make this analysis challenging. Also in this case, the group Fourier analysis and the symbolic calculus for Rockland operators, discussed later in Chapter \ref{CP:2}, are the key ingredients allowing us to establish sharp well-posedness results in the spirit of Colombini, De Giorgi and Spagnolo, \cite{CDS79}. In particular, we describe an interesting phenomenon of a loss of regularity, depending on the step of the group and on the order of the operator considered.

In this study, similarly to that of Garetto and Ruzhansky, a condition reminiscent of Gevrey regularity appears, in the sense that we obtain well-posedness results in the spaces $\mathcal{G}_{\mathcal{R}}^s(G)$ for $s\geq 1$, defined as follows
\begin{align}\notag
\mathcal G^s_{\mathcal R}(G):=\{f\in\mathcal{C}^\infty(G)\,|\,\exists A>0\,:\,\|e^{A\mathcal R^{\frac{1}{\nu s}}}f\|_{L^2(G)}<\infty\}.
\end{align}
The dependence on the Rockland operator $\mathcal{R}$ is apparent, though we demonstrate in Section \ref{SEC:weR} that these spaces are independent of the choice of $\mathcal{R}$. Moreover it is natural to associate them to Gevrey spaces, because of the classical characterisation existing in the Euclidean case (see Theorem \ref{Thm:Rodino}). This characterisation is equivalent to the classical definition of Gevrey space. Until now, a description of the spaces $\mathcal G^s_{\mathcal R}(G)$ as a modification of the classical definition of Gevrey spaces is still missing. This inspires the investigation of the relationship between the spaces $\mathcal G^s_{\mathcal R}(G)$ and a certain variant of Gevrey spaces we will introduce, the latter of which is outlined in the next Subsection and defined rigorously in Chapter \ref{CP:4}.

\subsection{Sub-Laplacian Gevrey spaces}

\textit{Gevrey functions}, so called in honour of Maurice Gevrey who introduced them  in 1918 \cite{G1918},  are defined as follows:
\begin{definition}[Gevrey functions of order $s$ in $\Omega$]
	Let $\Omega$ be an open subset of $\Rn$ and let $s\geq 1$. A function $f$ is a Gevrey function of order $s$, i.e. $f\in G^s(\Omega)$, if $f\in\mathcal{C}^\infty(\Omega)$ and for all compact subsets $K$ of $\Omega$ there exist two positive constants $A$ and $C$ such that for all $\alpha\in\N^n_0$ and for all $x\in K$ we have
	\[
	|\partial^\alpha f(x)|\leq A C^{|\alpha|}(\alpha!)^s.
	\]
\end{definition}
From the definition\footnote{For $s=1$ the corresponding Gevrey class of functions coincides with the space of real analytic functions. The Gevrey order $s$ is a way to estimate the divergence of the Taylor expansion of a smooth function: the larger $s$, the `more divergent' the Taylor expansion.} we can deduce that the Gevrey classes of functions are intermediate spaces between the smooth functions and the real-analytic functions\footnote{A function $f\in\mathcal{C}^\infty(\Rn)$ is real-analytic if and only if for every compact set $K\subset \Rn$ there exists a constant $C$ such that for every $x\in K$ and every $\alpha\in\N^n_0$ we have the bound $|\partial^\alpha f(x)|\leq C^{|\alpha|+1}\alpha!.$}. For this reason they are reasonably important in several applications in, e.g., Gevrey micro-local analysis \cite{BT1997,MR1997}, Gevrey solvability \cite{CZ1993,MR1997,DFT2009}, in the study of hyperbolic equations \cite{Bronshtein:TMMO-1980,CDS79,M1985}, dynamical systems \cite{G1995, RS1996} and evolution partial differential equations \cite{F1989,KV2011,LMR2000}. 

A characterisation on the Fourier transform side of the Gevrey spaces effectively  increases their applicability in many problems, most notably allowing one to obtain energy estimates for evolution pdes, see e.g. Rodino's book \cite{R1993} or \cite{DFT2009}. Particularly well known is the equivalence between Gevrey functions and smooth functions $\widehat\phi$ satisfying the inequality 
\begin{equation}\label{EQ:charRn}
|\widehat\phi(\xi)|\leq C e^{-\epsilon |\xi|^{\frac{1}{s}}},\quad\text{for all }\xi\in\Rn,
\end{equation}
for positive constants $C$ and $\epsilon$. This characterisation highlights an important link between the Gevrey functions and the Laplacian operator, by observing that the Fourier transform of the Laplacian $|\xi|^2$ appears in the function $ |\xi|^{\frac{1}{s}}=(|\xi|^2)^{\frac{1}{2s}}$ in \eqref{EQ:charRn}.

In 1969, Seeley \cite{S1969} characterized analytic functions on compact manifolds in terms of their eigenfunction expansions. Recently, Dasgupta and Ruzhansky, starting from Seeley's work, extended the study of the Gevrey regularity to the non-commutative  compact case: firstly, they shown in \cite{DR2014} that the Gevrey spaces defined in local coordinates on a compact Lie group allow for a similar `global' descriptions in terms of the Laplacian on the group. Subsequently, in \cite{DR2016} they further extended  the characterisation from \cite{DR2014} to the setting of an elliptic operator $E$ on a general compact manifold $M$. In particular they proved that a function $f$ belongs to the Gevrey class $G^s(M)$ if and only if 
\begin{equation}\label{EQ:Gs}
|\widehat{f}(j)|\leq C e^{-\epsilon \lambda_j^{\frac{1}{s\nu}}},
\end{equation}
for some $\epsilon, C>0$ and for all $j\in\N_0$, where $\lambda_j$ are the eigenvalues of $E$ and 
$\widehat{f}(j)=(f,e_j)_{L^2(M)}$, with $e_j$ being the eigenfunction of $E$ corresponding to $\lambda_j$. Furthermore, the spectral calculus of $E$ allows one to prove the equivalence between the latter condition \eqref{EQ:Gs} and
\begin{equation}\label{EQ:GsExp}
\exists D>0\textrm{ such that }
\| e^{D E^{\frac{1}{s\nu}}} f\|_{L^2(M)}<\infty. 
\end{equation}

\textbf{Original work.}\\
Our initial aim has been to study the Gevrey hypoellipticity of sub-Laplacians by modifying the Gevrey spaces.
In order to do that, we investigated what happens in the characterization of Gevrey spaces if the operator $E$ is no longer elliptic but \textit{hypoelliptic}. The model case for this situation will be when, for example, we replace the second order elliptic pseudo-differential operator $E$ in \eqref{EQ:GsExp} by a H\"ormander sum of squares, say $\slp$. The main question here is to understand how to define \text{new} Gevrey classes depending on $\slp$, say $\gamma^s_\slp(M)$, in order for such a characterisation still to hold.

Furthermore, it is of key importance to remove the compactness assumption on $M$ while also, for some results, allowing it to have more structure. For $\slp$ a sub-Laplacian, we will look at the cases when $M$ is a stratified Lie group, for example the Heisenberg group.

We will present partial answers to these questions. Indeed, we introduce new classes of functions, obtained adapting the definition of Gevrey spaces to the fact that we consider hypoelliptic operators. However, we will be assisted by the choice of the operator being H\"ormander's sum of squares so that we have a fixed collection of vector fields at our disposal to work with. We prove the inclusion of our new spaces in spaces of functions defined in terms of the boundedness of a heat kernel operator. Unfortunately, the other direction, that would yield the general characterisation, is still open, \textit{even in the compact case}. Notwithstanding this, we show that the well-known non-commutative Fourier analysis on specific groups allows us to obtain the desired characterisation on these groups. This is the case for the special unitary group SU(2) and the Heisenberg group $\h_n$, respectively compact and non-compact. It is important to emphasise also that, contrary to the Dasgupta--Ruzhansky work, the general argument we have developed for this study is independent of the symbolic calculus, and exploits standard tools from analysis such as factorial inequalities, norm properties, functional calculus for sub-Laplacians and Sobolev embeddings. Moreover, if we consider the `degenerate' case when the sub-Laplacian is actually a Laplacian and the manifold is compact, our arguments provide a shorter, less technical proof of \cite[Theorem 2.6]{DR2014} which does not use the symbolic calculus.

\subsubsection{Sub-Laplacian wave equation \& sub-Laplacian Gevrey spaces}

As we observed for the Euclidean case, the characterisation of Gevrey functions on the Fourier transform side is highly relevant for applications, for example in the study of the well-posedness of hyperbolic pdes in $\Rn$ 
\cite{CDS79}. Moving closer to the subject of this thesis, we can consider again the following Cauchy problem for the sub-Laplacian wave equation on a stratified Lie group $G$:
\begin{equation}\notag
\left\{\begin{array}{rclc}
\partial ^2_{t}u(t,x)+a(t)\slp u(t,x)	&=&0,			&(t,x)\in [0,\tau]\times G,\\
u(x,0)	&=&\varphi(x),	&x\in G,\\
\partial _t u(x,0)	&=&\psi(x),		&x\in G,
\end{array}\right.
\end{equation}  
where $\slp$ is a positive sub-Laplacian, clearly a particular case of \eqref{EQ:weG}. We observed that even in local coordinates it is natural to expect the appearance of Gevrey spaces in such problems. Equivalently in this non-commutative and hyperbolic context we can prove sharp well-posedness results in spaces that in literature have been na\"ively called \textit{Gevrey spaces}:
\begin{align}\label{EQ:GsR}
\mathcal G^s_{\slp}(G):=\{f\in\mathcal{C}^\infty(G)\,|\,\exists A>0\,:\,\|e^{A\slp^{\frac{1}{\nu s}}}f\|_{L^2(G)}<\infty\}.
\end{align}
The discussion above, that we will thoroughly extend in Chapter \ref{CP:4}, highlights the importance of a deeper investigation of these spaces in order to be able to correctly associate them to the Gevrey spaces. For the state-of-the-art, only when $G=\mathbb{H}_n$ or $SU(2)$ may we speak \textit{properly} of \textit{sub-Laplacian Gevrey spaces} when referring to \eqref{EQ:GsR}.

\section{Dissertation overview}
This dissertation comprises roughly three main parts. 
The first part corresponds to Chapter \ref{CP:2} and is devoted to the background material. In fact, we aim  to introduce all the necessary tools to make the reader familiar with the abstract Fourier analysis on groups and to fix the notation for the rest of this dissertation. Firstly, we introduce the definitions of Lie group and Lie algebra and we present some of their properties. Then we focus on the representation theory on arbitrary groups and on the existence of a suitable invariant measure on locally compact groups. These two topics will be the key ingredients to introduce the global \textit{Fourier analysis} on groups. Subsequently we fix our attention on a subclass of Lie groups, the graded groups, which are particularly interesting to our purposes. In fact, they are a natural habitat to extend Euclidean results. We conclude the Chapter with an introduction to two important groups, the special unitary group $SU(2)$ and the $(2n+1)$-dimensional Heisenberg group $\mathbb {H}_n$. 

The second part of this dissertation concerns the study of the hyperbolic wave equation on graded groups, and is given in Chapter \ref{CP:3}. We present well-posedness results for the initial value problem for  time-dependent variation of the wave equation. The purpose of this Chapter is to extend the study of the well-posedness of the Cauchy problem for the wave equation with time-dependent coefficients to graded groups. The proofs here rely on the global Fourier analysis on $G$. Indeed, a suitable application of the Fourier transform to the equations we are interested in allows one to view such equation as an infinite system of equations whose coefficients depend only on $t$. This leads to a range of sharp results based on the behaviour of the coefficient $a(t)$.

Finally, the last part of this thesis is developed in Chapter \ref{CP:4}, where we introduce our suggested definition of sub-Laplacian Gevrey spaces and we prove the inclusion of these spaces in a bigger class of functions depending on a heat-type operator. In addition, we prove a full characterisation in the cases of sub-Laplacian Gevrey spaces on the special unitary group SU(2) and on the Heisenberg group $\mathbb{H}_n$.
This Chapter is organised as follows. Firstly we recall the Euclidean Gevrey spaces and we state the characterisation of Gevrey functions in terms of their Fourier transform. Then, we move to the non-commutative setting, mentioning the case of Gevrey spaces on compact groups and manifolds studied by Dasgupta and Ruzhansky in \cite{DR2014}. Subsequently, we consider the more general setting of manifolds and develop our original work on sub-Laplacian Gevrey spaces. Indeed, we study these classes of functions, embedding them in bigger spaces explicitly dependent on the sub-Laplacian. We notice that our new classes of functions match the Gevrey type functions we encountered in the study of the well-posedness of the wave equation in Chapter \ref{CP:3}. Finally we consider the case of the special unitary group and the Heisenberg group, where the explicit symbolic calculus allows us to prove a complete characterisation for our sub-Laplacian Gevrey functions.  

\newpage

\textbf{Index of notation}

\begin{abbrv}
	\item[$\R$ $(\C)$]				real (complex) numbers
	\item[$\N_0$] 				natural positive numbers including $0$
	\item[$\alpha\in\N^n_0$] 			multi-index
	\item[$|\alpha|$]					length of the multi-index $\alpha$
	\item[$\delta_{m,n}$] 			Kronecker delta	
	\item[$\partial_x$]				partial derivative with respect to $x$
	\item[$\mathcal{H}$]		complex Hilbert space
	\item[$\|\cdot\|_{\mathcal{H}}$, $(\cdot,\cdot)_{\mathcal{H}}$] 			norm, inner product on $\mathcal{H}$
	\item[$\mathcal{U}(\mathcal{H})$]         space of unitary operators on $\mathcal{H}$
	\item[$X'$]             	dual space of a given topological vector space $X$    											
	\item[$L^*$]					(Hilbert space) adjoint of the operator $L$
	\item[$\|L\|_{\HS}$, $\Tr(L)$]                Hilbert--Schmidt norm, trace of an operator $L$
	\item[ $\|\cdot\|_{\OP}$] 				operator norm
	\item[$M$]					smooth manifold
	\item[$G$, $\mathfrak{g}$, $\mathfrak{U}(\mathfrak{g})$]     (Lie) group, Lie algebra, universal enveloping algebra
	\item[$e$]     					identity of a group $G$
	\item[$X_1,X_2,\dots$] 				left-invariant vector fields 
	\item[${[\cdot,\cdot]}$] 				commutator of vector fields
	\item[$\exp_{G}$]				exponential map 
	\item[$\Gh$] 					unitary dual of $G$	
	\item[$\pi_L$, $\pi_R$]     		left, right regular representation
	\item[$\pi$, ($\xi$)] 				unitary representation of a (\textit{compact}) Lie group $G$	
	\item[$d_\xi$]				dimension of a representation $\xi$ of a compact Lie group
	\item[$\mathcal{H}_\pi$]			Hilbert space associated with the representation $\pi$
	\item[$d\pi$ or $\pi$]  	infinitesimal representation of a Lie algebra
	\item[$\mathcal{H}_\pi^\infty$]			space of smooth vectors of $\pi$
	\item[$\mathcal{F}_G(f)(\pi),\widehat{f}(\pi),\pi(f)$] 				group Fourier transform of $f$ at the representation $\pi$
	\item[$\sigma_L(\xi)=\widehat{L}(\xi)$]				matrix-valued symbol of an operator $L$ on $G$ compact at $\xi$	
	\item[$\Delta$] 					Laplace operator
	\item[$\slp$]   					sum of squares of vector fields, often sub-Laplacian operator
	\item[$\mathcal{R}$]				Rockland operator
	\item[$\mu$]					measure on $M$, Haar measure on a Lie group $G$
	\item[$D_r$] 					dilation on $\mathfrak{g}$ and $G$
	\item[{$[\alpha]$}]				homogeneous degree
	\item[$SU(2)$, $\mathfrak{su}(2)$] 					special unitary group and its Lie algebra
	\item[$\mathbb{H}_n$, $\mathfrak{h}_n$]                Heisenberg group and Heisenberg Lie algebra
	\item[$e_\xi$] 			trigonometric function
	\item[$T_l$] 			representation of $SU(2)$
	\item[$\pi_{\lambda}$]	Schr\"odinger representation of $\h_n$
	\item[$\mathcal{D}(G)$, $\mathcal{C}^\infty(G)$]     smooth functions on a group $G$
	\item[$\mathcal{S}(G)$]  			Schwartz space on $G$
	\item[$\mathcal{D}'(G)$]			distributions on a group $G$
	\item[$L^p(G))$]	    			$p$-integrable functions with respect to the Haar measure on a group $G$			
	\item[$\smash{\displaystyle{\|\cdot\|_{L^p}}}$] 				$L^p$-norm
	\item[$G^s(\Omega)$] 						Gevrey space of order $s$ in $\Omega\subset\Rn$
	\item[$\gamma^s_{\mathbf{X},L^\infty}(M)$, $\gamma^s_{\mathbf{X},L^2}(M)$]  sub-Laplacian Gevrey space of order $s$ on $M$ with respect to $L^\infty$, $L^2$ -norm
	\item[$h_t$]						heat kernel
	\item[$\delta_0$]				Dirac delta function
	\item[$f\star g$]			convolution of $f$ and $g$
\end{abbrv}


\chapter{Preliminaries}\label{CP:2}

In this Chapter we aim to provide all the key notions to present our results, thereby introducing the reader to abstract harmonic analysis and fixing notation for the rest of this dissertation. Firstly we introduce the Lie group and Lie algebra structures that we will work with, then we focus on representation theory on arbitrary groups and finally we show how this allows one to introduce a \textit{Fourier analysis} on groups. We conclude this Chapter with an introduction to two particular groups, the special unitary group $SU(2)$ and the $(2n+1)$-dimensional Heisenberg group $\mathbb {H}_n$. The topics presented are well-known so we will omit most of the proofs, but the interested reader is invited to consult the references provided, such as \cite{CG90,Fol,FR2016,H2003,R2001,T1986}.

\section{Lie Groups and Lie algebras}
To extend the concepts of differentiation to groups, it is necessary to endow them with a differential structure. This is provided by the notion of a \textit{Lie group}, that, roughly speaking, is simultaneously a \textit{group}, i.e. an algebraic abstraction of symmetries, and a  \textit{differentiable manifold}, i.e. a space that locally looks Euclidean. The formal definition is as follows.

\begin{definition}[Lie group]
	A Lie group $G$ is a $\mathcal{C}^\infty$-manifold with a group structure so that the group operations are smooth. More precisely, the multiplication and inversion maps
	\begin{align*}
	(g,g')\in G\times G\rightarrow gg'\in G\,,\\
	g\in G\rightarrow g^{-1}\in G
	\end{align*}
	are smooth as maps of manifolds. We denote by $e$ the identity of a Lie group $G$.
\end{definition}
\begin{remark}
	As they are manifolds, Lie groups are naturally \textit{topological groups}, and consequently all topological results apply to Lie groups. Their underlying topology is locally compact and Hausdorff, since they are locally Euclidean. The local compactness will turn to be an essential property to guarantee the existence of a specific measure, as will become clear in the next sections. Similarly, a Lie group is called \textit{compact} if it is compact as a manifold.
\end{remark}

\begin{example}
	Some examples of Lie groups are:
	\begin{itemize}
		\item the special linear group $SL(n, \mathbb{K}) = \{A \in GL(n, \mathbb{K}) \,|\, \det A = 1\}$, with $\mathbb{K}=\R,\C$;
		
		\item the unitary group $U(n) = \{A \in GL(n,\C)\,|\, AA^* = I\}$, where $A^*$ denotes the Hermitian transpose of $A$;
		
		\item the special unitary group $SU(n)=U(n)\cap SL(n,\C)$.
	\end{itemize}
\end{example}

Arguably, one of the key points in the theory of Lie groups is the possibility of looking at the global object, the group, through its linearised version,  called its \textit{infinitesimal group} and better known as its Lie algebra. We now give a formal definition for a general Lie algebra.
\begin{definition}[Lie algebra]
	\label{Liealgebra}
	A (real) \textit{Lie algebra} $L$ is a vector space over $\R$ equipped with a bilinear map
	$$
	[\cdot,\cdot]:L\times L\rightarrow L,
	$$
	called the \textit{Lie bracket} or \textit{commutator}, satisfying the following properties
	\begin{enumerate}[1.]
		\item $[\cdot,\cdot]$ is skew-symmetric, i.e., $[X,Y]=-[Y,X]$ for all $X,Y\in L$;
		\item for every $X,Y,Z$, the Jacobi identity holds: $[X,[Y,Z]]=[[X,Y],Z]+[Y,[X,Z]]$.
	\end{enumerate}
\end{definition}
Every Lie group can be endowed with a Lie algebra structure related to its tangent vectors. In order to do this, we recall the definitions of tangent vector, tangent space and vector field.

\begin{definition}[Tangent vector, tangent space, tangent bundle, vector field]
	Let $x\in G$. A \textit{tangent vector} to $G$ at $x$ is a map $X_{x}:\mathcal C^{\infty}(G)\rightarrow \R$ such that for all $\alpha,\beta\in\R$ and $f,g \in \mathcal C^{\infty}(G)$ 
	\begin{enumerate}[i.]
		\item $X_{x}(\alpha f+\beta g)=\alpha X_{x}f+\beta X_{x}g$;
		\item $X_{x}$ satisfies the Leibniz formula, i.e., $X_{x}(fg)=g(x)X_{x}f+f(x)X_{x}g$.
	\end{enumerate}
	
	The space of all tangent vectors at $x$ is called the \textit{tangent space} and is denoted $T_xG$. This is a finite-dimensional vector space whose dimension equals the dimension of $G$ as a manifold. The \textit{tangent bundle} over $G$ is the disjoint union of all the tangent spaces
	\[
	TG=\bigcup_{x\in G}T_x G.
	\]
	A (smooth) section of the tangent bundle $X: G\rightarrow TG$ is called a \textit{vector field}. Given $x\in G$, the vector $X(x):=X_x\in T_x G$ acts on $\mathcal C^{\infty}(G)$ as follows
	\[
	X(x)(f)=(Xf)(x):=(X_x f)(x),\quad f\in\mathcal{C}^\infty(G). 
	\]
\end{definition}
As is clear in the above expression, every smooth vector field $X$ can be regarded as a linear map $X:\mathcal{C}^\infty (G)\rightarrow \mathcal{C}^\infty(G)$ such that $X(fg)=fX(g)+gX(f)$, for all $f,g\in\mathcal{C}^\infty(G)$. In this way $X$ acts like a derivative on functions from $\mathcal{C}^\infty(G)$. The space of all vector fields, denoted by $\mathfrak{X}(G)$, can be endowed with a bracket structure. Given $X,Y\in\mathfrak{X}(G)$, we can define a \textit{new} vector field $[X,Y]:G\rightarrow TG$ as follows

\begin{align}\label{bracket}
[X,Y](x)(f):=X(x)Yf-Y(x)Xf,\quad x\in G,\, f\in\mathcal{C}^\infty(G).
\end{align}

It can be proved this bracket structure is skew-symmetric and satisfies the Jacobi identity. Therefore, \eqref{bracket} defines a commutator bracket for vector fields, that induces a Lie algebra structure. We want to relate $\mathfrak{X}(G)$ to the group structure. Given an element $y\in G$, the \textit{left(resp. right)-translation} by $y$ is defined as follows

\[
L_y(\text{resp. }R_y):G\rightarrow G, \quad L_y(x):=yx\quad (\text{resp. }R_y(x):=xy).
\]

The derivatives of these maps act at the level of the tangent spaces $dL_y(\text{resp. }dR_y):TG\rightarrow TG$, in the sense that $dL_y\in\mathcal{L}(T_{\cdot}G,T_{y\cdot}G)$(resp. $dR_y\in\mathcal{L}(T_{\cdot}G,T_{\cdot y}G)$). From now on, we will consider the left-translation, but it is possible to reformulate all the following concepts in terms of $R_y$. 
\begin{definition}
	Let $X\in\mathfrak{X}(G)$. We say that $X$ is \textit{left-invariant} if, for every $g\in G$,
	\begin{align}\label{leftinvariant}
	X\circ L_g=dL_g\circ X.
	\end{align}
	We denote the space of all left-invariant vector fields by $\mathfrak{X}_L(G)$.
\end{definition}
From \eqref{leftinvariant}, a left-invariant vector field $X$ is uniquely determined by its value at the identity, i.e., $X_{e}\in T_e G$. This means there is a bijection between the space of left-invariant vector fields $\mathfrak{X}_L(G)$ and $T_eG$. Furthermore a straightforward calculation shows that if $X,Y\in\mathfrak{X}_L(G)$, then also $[X,Y]\in\mathfrak{X}_L(G)$. Consequently, $\mathfrak{X}_L(G)$ is also a Lie algebra with respect to the bracket commutator \eqref{bracket}.
\begin{definition}
	Let $G$ be a Lie group. The \textit{Lie algebra $\mathfrak g$ of} $G$ is $T_e G$ with the commutator bracket induced by its identification with $\mathfrak{X}_L (G)$, i.e.,
	\[
	[X,Y]_{T_e G}:=[\tilde{X},\tilde{Y}]_{\mathfrak{X}_L(G)}(e),\quad\forall X,Y\in T_e G,
	\]
	where for every $g\in G$ we define $\tilde{X}(g):=dL_g(X)$ and $\tilde{Y}(g):=dL_g(Y)$.
\end{definition}
\begin{example}
	Consider the Lie group $G=GL(n,\Rn)$. Its Lie algebra is $\mathfrak{g}=\R^{n\times n}=M_n(\R)$, equipped with the lie bracket $[A,B]_{M_n(\R)}:=AB-BA$, for all $A,B\in M_n(\R)$.
\end{example}
\subsection{Exponential map}

Let $G$ be a Lie group and $\mathfrak{g}$ its Lie algebra. We construct the \textit{exponential map} $\exp:\mathfrak{g}\rightarrow G$, a diffeomorphism that will allow us to obtain information about the structure of the Lie algebra. Consequently, we will be able to interpret tangent vectors in the Lie algebra as directional derivatives.

\begin{definition}[Exponential map]
	The \textit{exponential map} $\exp_G: \mathfrak{g}\rightarrow G$ is defined at every $X\in\mathfrak{g}$ as
	\[
	\exp_{G}(X):=\gamma_{X}(1),
	\]
	where $\gamma_X:[0,1]\rightarrow G$ is the unique solution to the initial value problem
	\begin{align*}
	\begin{cases}
	\gamma'(t)=X_{\gamma(t)}\\
	\gamma(0)=e,
	\end{cases}
	\end{align*}
	i.e., that solution determined by the left-invariant vector field associated with $X$.
\end{definition}

\begin{example}
	Consider $G=GL(n,\R)$ and $\mathfrak{g}=M_n(\R)$. The exponential map is given by the matrix exponential, expressed by the series expansion below  
	\begin{align}\label{conv_exp}
	\exp_{GL(n,\R)}:M_n(\R)\rightarrow GL(n,\R),\quad A\mapsto\sum_{k=1}^{\infty}\frac{A^k}{k!}.
	\end{align}
	The absolute convergence of the series \eqref{conv_exp} follows from the convergence of the exponential power series for real numbers and hence for square real matrices.
\end{example}

We define the \textit{descending central series} of a Lie algebra $\mathfrak{g}$ by
\begin{align*}
\mathfrak{g}_{(1)}:=\mathfrak{g},\quad \mathfrak{g}_{(m+1)}:=[\mathfrak{g},\mathfrak{g}_{(m)}].
\end{align*}

A Lie algebra is said to be \textit{nilpotent} if there exists an integer $n$ such that $\mathfrak{g}_{(n+1)}=0$. If, additionally, $\mathfrak{g}_{(n)}\neq 0$, then $n$ is minimal and $\mathfrak{g}$ is called \textit{$n$-step nilpotent}.  A Lie group is then called \textit{nilpotent} if its associated Lie algebra is nilpotent.

\begin{proposition}
	Let $G$ be a nilpotent and simply connected Lie group and $\mathfrak{g}$ be its Lie algebra. Then the exponential map $\exp_G:\mathfrak{g}\rightarrow G$ is a diffeomorphism.
\end{proposition}

For a proof of this Proposition, see \cite[Theorem 1.2.1]{CG90} or \cite[Proposition 1.2]{FS}.

As mentioned above, using the exponential map, every vector $X\in\mathfrak{g}$ can be seen as a left-invariant differential operator on $\mathcal{C}^\infty(G)$ defined as follows

\begin{align}\label{vectorfield}
Xf(x):=\frac{d}{dt}f(x\exp_G(tX))\Bigr|_{t=0}.
\end{align}

In the above, we have $\exp_G(tX)=\gamma_X(t)$. In fact a straight calculation shows that the integral curve of the vector field associated with $tX$ satisfies $\gamma_{tX}(s)=\gamma_X(ts)$.

We conclude this Section observing that we can look at the Lie algebra $\mathfrak{g}$ of a Lie group $G$ both as the tangent space at the identity $T_e G$ and as the vector space of first-order, left-invariant partial differential operators on $G$.

\begin{remark}[The Baker--Campbell--Hausdorff formula]
	For details see \cite{CG90,H2003,R2014}. \\
	Given two arbitrary vector fields $X, Y\in\mathfrak{g}$, the identity 
	\begin{align}\notag
	\exp_{G}(X)\exp_{G}(Y)&=\exp_{G}\Big(X+Y+\frac{1}{2}[X,Y]+\frac{1}{12}\big([X,[X,Y]]+[Y,[Y,X]]\big)+\\
	&-\frac{1}{24}[Y,[X,[X,Y]]]+\dots\Big),\label{BakerCampbellHausdorff}
	\end{align}
	where the dots indicate commutators of order higher then four, is said to be the Baker--Campbell--Hausdorff formula. This identity allows us to develop the product of exponentials of vector fields on $G$.
	Dealing with specific groups, in particular with nilpotent groups, the commutators become trivial from a certain order on, generating a finite series in the right-hand side of \eqref{BakerCampbellHausdorff}.
	
\end{remark}
An immediate consequence of the Baker--Campbell--Hausdorff formula is the following:
\begin{proposition}\label{inverseexponential}
	Let $G$ be a Lie group. Then, for every $X\in\mathfrak{g}$ we have
	\[
	\exp_G(X)^{-1}=\exp_{G}(-X).
	\]
\end{proposition}
\begin{proof}
	Let $X\in\mathfrak{g}$. Let us observe that we can apply the Baker--Campbell--Hausdorff formula to $tX$ and $sX$ with $t,s\in \R$, to obtain 
	\begin{align*}
	\exp_{G}(tX)\exp_{G}(sX)=\exp_{G}\big((t+s)X\big),
	\end{align*}
	since $[X,X]=0$.
	For $t=1$ and $s=-1$, the latter identity becomes
	\begin{align*}
	\exp_{G}(X)\exp_{G}(-X)=e,
	\end{align*}
	and the proposition follows.
\end{proof}

\subsection{Universal enveloping algebra} 

For more detail, see \cite[Section 9.3]{H2003} or \cite[Section 1.3]{FR2016}. Let $G$ be a Lie group and $\mathfrak{g}$ its Lie algebra. The \textit{universal enveloping algebra} of $\mathfrak{g}$, denoted by $\mathfrak{U}(\mathfrak{g})$, is an associative algebra with identity over the complex field $\mathbb{C}$ given by the natural non-commutative polynomial algebra on $\mathfrak{g}$. The rough idea is to be able to embed $\mathfrak{g}$ as a subspace of such an associative algebra in a way that the bracket on $\mathfrak{g}$ may be written as $[X,Y]=XY-YX$, where $XY$ and $YX$ are computed in the algebra.\footnote{Here we mean that there exists a linear map $i:\mathfrak{g}\rightarrow \mathfrak{U}(\mathfrak{g})$ such that for every $X,Y\in\mathfrak{g}$ we have $i([X,Y])=i(X)i(Y)-i(Y)i(X)$.} Previously we have interpreted the Lie algebra $\mathfrak{g}$ as the space of left-invariant derivatives on $G$. Analogously the enveloping algebra can be interpreted as the vector space of finite-order left-invariant partial differential operators on $G$. The \textit{Poincar\'e--Birkhoff--Witt theorem} (PBW theorem), that we will state in a moment, clarifies this analogy. 

Let us proceed by formalising the above informal interpretation. Let $\mathfrak{g}^{\mathbb{C}}$ be the complexification of $\mathfrak{g}$, i.e. $\mathfrak{g}^{\mathbb{C}}:=\mathfrak{g}\oplus i \mathfrak{g}$ with the standard rule for multiplication by complex numbers\footnote{If $v_1+iv_2,v_3+iv_4\in \mathfrak{g}^{\mathbb{C}}$, with $v_1,v_2,v_3,v_4\in\mathfrak{g}$, then $(v_1+iv_2)(v_3+iv_4)=(v_1v_3-v_2v_4)+i(v_2v_3+v_1v_4)$.}. Recalling that each Lie algebra is a vector space, we can freely construct the \textit{tensor algebra} of a Lie algebra, that means, an algebra containing all possible tensor products of all possible vectors in the Lie algebra. Thus, the tensor algebra of $\mathfrak{g}$ is defined as 
\begin{align*}
\mathcal{T}(\mathfrak{g}):=\bigoplus_{m=0}^\infty\otimes^m\mathfrak{g}^{\mathbb{C}}=\bigoplus_{m=0}^\infty \underbrace{\mathfrak{g}^{\mathbb{C}}\otimes\dots\otimes\mathfrak{g}^{\mathbb{C}}}_\text{$m$ times}
\end{align*}
The universal enveloping algebra is obtained as a quotient of $\mathcal{T}(\mathfrak{g})$. Indeed, we consider a \textit{two-sided-ideal} in $\mathcal{T}(\mathfrak{g})$, that means, a subspace $\mathcal{J}$ of $\mathcal{T}(\mathfrak{g})$ satisfying that if $\alpha\in\mathcal{J}$ and $\beta\in\mathcal{T}(\mathfrak{g})$ then $\alpha\beta$ and $\beta\alpha$ belong to $\mathcal{J}$. This $\mathcal{J}$ is spanned by
\begin{align*}
\{X\otimes Y-Y\otimes X -[X,Y]\,|\,X,Y\in\mathfrak{g}\}.
\end{align*}
Most importantly this approach induces the commutator structure on $\mathcal{T}(\mathfrak{g})$. We call the quotient algebra
\begin{align*}
\mathfrak{U}(\mathfrak{g}):=\mathcal{T}(\mathfrak{g})/\mathcal{J}
\end{align*}
the \textit{universal enveloping algebra} of $\mathfrak{g}$. Furthermore, consider the quotient mapping
\begin{align*}
i:\mathcal{T}(\mathfrak{g})\rightarrow \mathfrak{U}(\mathfrak{g})=\mathcal{T}(\mathfrak{g})/\mathcal{J},\quad T\in\mathcal{T}(\mathfrak{g})\mapsto T+\mathcal{J}\in\mathfrak{U}(\mathfrak{g}).
\end{align*}
When restricted to $\mathfrak{g}$ this is the canonical mapping of $\mathfrak{g}$ that gives the embedding of $\mathfrak{g}$ into $\mathfrak{U}(\mathfrak{g})$, i.e., $i\big|_{\mathfrak{g}}:\mathfrak{g}\rightarrow \mathfrak{U}(\mathfrak{g})$. Now we may state the pivotal PBW theorem:
\begin{theorem}[Poincar\'e--Birkhoff--Witt]\label{PBW}
	Let $\mathfrak{g}$ be a finite-dimensional Lie algebra and let $\{X_1,\dots,X_n\}$ be a basis of $\mathfrak{g}$. Then for all multi-indices $\alpha=(\alpha_1,\dots,\alpha_n)\in\N^n_0$ the expressions
	\begin{align*}
	i(X_1)^{\alpha_1}\dots i(X_n)^{\alpha_n}
	\end{align*}
	span $\mathfrak{U}(\mathfrak{g})$ and are linearly independent. In particular, the elements $i(X_1),\dots,i(X_n)$ are linearly independent, meaning that the canonical map $i:\mathfrak{g}\rightarrow\mathfrak{U}(\mathfrak{g})$ is injective.
\end{theorem}
Note that the injectivity of the map $i$ allows us to identify $\mathfrak{g}$ with its image under $i$. Therefore, from now on we denote the elements of the universal enveloping algebra by $X$ in place of $i(X)$. In other words, if we regard the elements of the basis of $\mathfrak{g}$ as first-order left-invariant derivatives, we can reformulate the PBW theorem saying that any element of $\mathfrak{U}(\mathfrak{g})$, i.e., any left-invariant differential operator $T$ on $\mathfrak{g}$ can be written in a unique way as a finite sum
\begin{align*}
T=\sum_{\alpha\in\N^n_0}c_\alpha X^\alpha,
\end{align*} 
where all but a finite number of the coefficients $c_\alpha\in\mathbb{C}$ are zero and for every multi-index $\alpha\in\mathbb{N}^n$ we denote $X^\alpha=X_1^{\alpha_1}\dots X_n^{\alpha_n}$, with $X_1,\dots, X_n$ basis for $\mathfrak{g}$. For future reference, given $T\in \mathfrak{U}(\mathfrak{g})$ as above, we define the \textit{formal adjoint operator} of $T$ as the element
\begin{align}\label{formaladjoint}
T^*:=\sum_{\alpha\in\N^n_0}\bar{c}_\alpha(-X_n)^{\alpha_n}\dots(-X_1)^{\alpha_1}\in\mathfrak{U}(\mathfrak{g}).
\end{align} 
We note that this definition coincides with the classical definition of adjoint of the corresponding left-invariant vector fields.

\section{Harmonic analysis on groups}
Broadly speaking, the classical theory of Fourier analysis is concerned with representing certain functions as a superposition of simpler functions, often called simple waves $e^{2\pi ix}$. This is a powerful method to tackle many problems from applied sciences, such as engineering, neurovision or image/sound processing. The main tool, as well as the starting point, of this theory is the notion of \textit{Fourier transform}, a continuous version of the Fourier coefficients. Following Stein and Shakarchi's monograph \cite{SS2003}, if $f$ is an integrable function on $\Rn$, its Fourier transform is the function $\widehat{f}$ defined on the phase space $\Rn$ as follows

\begin{align}\label{euclidianfouriertransform}
\mathcal F(f)(\xi)=\widehat{f}(\xi):=\int_{\Rn}e^{-2\pi i \xi\cdot x}f(x)dx.
\end{align}

Under suitable hypotheses, the function $f$ can be obtained from its Fourier transform via the inverse Fourier transform:
\begin{align*}
\mathcal F ^{-1} \big(\mathcal{F}(f)\big)(x)=\int_{\Rn}e^{2\pi i \xi\cdot x}\widehat{f}(\xi)d\xi=f(x).
\end{align*}

From the definition \eqref{euclidianfouriertransform} several further properties can be derived. This makes the Fourier transform a crucial tool, for example, in the study of partial differential equations. We are interested in extending these notions to suitable groups. In order to do this, we need to integrate on the group, i.e., we need a measure that satisfies an analogue of the translation invariance of the Lebesgue measure in the Euclidean case. We also want to define a Fourier transform on groups, so we need to identify a counterpart for the phase functions $e_{\xi}:=e^{2\pi i \xi\cdot}$ appearing in \eqref{euclidianfouriertransform}. In the next Section we will introduce these as the two ingredients that are the \textit{Haar measure} and the \textit{representation theory} of groups. They will then allow us to generalise the classical Fourier analysis by providing a general form of the \textit{Fourier transform for groups}.

\subsection{First ingredient for harmonic analysis on groups:\\ Haar measure}

Further details for this Section can be found in \cite{DE2009,FR2016}. Every Lie group $G$ is naturally a topological space. Therefore we can consider the \textit{Borel $\sigma$-algebra} $\mathcal{B}$ on $G$, that is, the smallest $\sigma$-algebra containing all open and closed sets. Any measure defined on $\mathcal{B}$ is called a \textit{Borel measure}. A measure $\mu$ on $G$ is \textit{locally finite} if for every element $g\in G$ there exists a neighbourhood $U$ of $g$ whose measure is finite, i.e., $\mu(U)<\infty$. A \textit{Radon measure} $\mu$ on $\mathcal{B}$ is a locally finite Borel measure satisfying that if $A\in\mathcal{B}$ then we have

\begin{align*}
\mu(A)=\inf_{\substack{A\subset U,\,U\in\mathcal{B}, \\ U \text{ open}}}\mu(U),\quad \mu(A)=\sup_{\substack{{K\subset A,\,K\in\mathcal{B}}\\ k\text{ compact}}}\mu(K).
\end{align*}

Roughly speaking, these conditions guarantee that the measure interacts nicely with the topology. A Borel measure $\mu$ is said to be \textit{left(resp. right)-invariant} if for every measurable set $A\in G$ and every $x\in G$ we have
\[
\mu(xA)=\mu(A)\quad(\text{resp. }\mu(Ax)=\mu(A)).
\]

\begin{definition}[Haar measure]\label{DEF:haar}
	Let $G$ be a locally compact group. A non-zero, left(resp. right)-invariant, Radon measure is called a \textit{left(resp. right)-Haar measure}.
\end{definition}

The Riesz representation theorem allows us to characterise Haar measures in terms of positive functionals on the space of continuous functions with compact support. In fact, given a measure $\mu$, it is a left(resp. right)-Haar measure if and only if
\[
\int_Gf(gx)d\mu(x)=\int_Gf(x)d\mu(x),
\left(\text{resp.}~\int_Gf(xg)d\mu(x)=\int_Gf(x)d\mu(x)\right),
\] 
for every $f\in\mathcal{C}_C(G)$ and every $g\in G$, see \cite[Proposition 11.4]{F2013}.

Note that if $\mu$ is a left-Haar measure, then the measure $\tilde{\mu}(E):=\mu(E^{-1})$ is a right-Haar measure. The proposition below is of fundamental importance and its detailed proof can be found in \cite{F1994,DE2009,F2013,R2014}. 

\begin{proposition}
	Let $G$ be a locally compact group. There exists a Haar measure that is uniquely determined up to a multiplicative constant. This means that if $\mu$ and $\nu$ are Haar measures on $G$ then there exists a constant $c>0$ such that $\nu=c\mu$.
\end{proposition}

\begin{example}
	The Lebesgue measure is the\footnote{Up to a multiplicative constant.} Haar measure on $\Rn$ and $\mathbb{T}$. The latter groups are abelian, therefore the left- and right-Haar measures coincide.
\end{example}

Given a locally compact group $G$, we say that it is \textit{unimodular} if the left-Haar measure\footnote{Uniquely defined up to a constant.} is also right-invariant. We are interested in Lie groups that are (naturally) locally compact and nilpotent. Since it is proved that nilpotent groups are unimodular, we can restrict our attention to Haar measures that are simultaneously left and right invariant. For a detailed discussion, we refer to \cite{R2014}.

Once we have a suitable left-invariant measure, we can introduce the usual \textit{Lebesgue spaces} on Lie groups. For every $1\leq p\leq \infty$, we denote by $L^p(G)$ the spaces of functions $f$ that are measurable with respect to the Haar measure, with norms

\begin{align*}
\|f\|_{L^p(G)}&=\Bigg(\int_G |f(x)|^p d\mu(x)\Bigg)^{1/p},\quad\text{for }p\in[1,\infty),\\
\|f\|_{L^\infty(G)}&=\esssup_{x\in G}|f(x)|,\quad\text{for }p=\infty.
\end{align*}

\subsection{Second ingredient for harmonic analysis on groups:\\ Representation theory}

For more detailed treatise see e.g. \cite{F1994,FR2016,H2003,R2014,RT2009}.
To accomplish the desired generalisation of Fourier analysis to Lie groups we need to determine a class of functions corresponding to the group homomorphisms 
\begin{align*}
e_\xi:\Rn\rightarrow \mathcal{U}(\C),\quad x\mapsto e^{2\pi i x\cdot \xi},\quad \text{for every }{\xi\in\Rn}.
\end{align*}

These functions are provided by the \textit{representations}. Roughly speaking, the fundamental idea in the representation theory of groups is to study algebraic abstract structures, such as groups, by representing their elements as linear transformations of vector spaces. In fact a representation is a linear action of a group on a vector space, providing the machinery of linear algebra on groups. We rigorously define representations below.

\begin{definition}[Unitary representation]
	Let $G$ be a group and $\mathcal H_\pi$ be a Hilbert space. A \textit{unitary representation} $\pi$ of $G$ over $\mathcal H_\pi$ is any group homomorphism from $G$ into the space of unitary operator on $\mathcal H_\pi$. This means it is a linear map
	\[
	\pi:G\rightarrow \mathcal{U}(\mathcal{H}_\pi)
	\]
	satisfying the following properties:
	\begin{enumerate}[i.]
		\item for every $g,h\in G$ we have $\pi(gh)=\pi(g)\pi(h)$;
		\item for every $g\in G$, $\pi(g)$ is unitary, i.e., $\pi(g)^{-1}=\pi^*(g)$, where $\cdot^*$ denotes the adjoint operator.
	\end{enumerate}
	The \textit{dimension of the representation} $\pi$ equals the dimension of the Hilbert space, $d_\pi=\dim \pi=\dim \mathcal{H}_\pi$. If the Hilbert space $\mathcal{H}_\pi$ is infinite-dimensional then the representation itself is said to be \textit{infinite-dimensional}.
\end{definition}

Note that condition ii., together with the linearity, implies $\pi^*(g)=\pi(g)^{-1}=\pi(g^{-1})$. 
\begin{example}
	Given a nilpotent locally compact group $G$ endowed with its Haar measure, the most classic example of unitary representation is given by the \textit{left (resp. right) regular representation}, defined by
	\begin{align*}
	\pi_L(\text{resp. }\pi_R): G\rightarrow \mathcal U (L^2(G)), \quad\text{such that} \\
	\pi_L(g)f(h)=f(gh),\quad(\text{resp. }\pi_R(g)f(h)=f(hg)).
	\end{align*} 
\end{example}

In order to classify \textit{all} the representations of a given group, it is necessary to introduce the following extra notions.
\begin{definition}[Invariant subspace]
	Let $\pi$ be a representation of the group $G$ on $\mathcal{H}_\pi$. A vector sub-space $W\subset \mathcal{H}_\pi$ is said to be \textit{$\pi$-invariant} if for every $g\in G$ we have $\pi(g)W\subset W$.
\end{definition}

\begin{definition}[Irreducible representation]
	A representation $\pi$ of the group $G$ on $\mathcal{H}_\pi$ is called \textit{irreducible} if the only $\pi$-invariant subspaces are the trivial ones, i.e., the whole space $\mathcal H_\pi$ and $\{0\}$.
\end{definition}

Irreducible representations are the building blocks of representations, in the sense that, under additional conditions of topological and algebraic nature, every (finite-dimensional) representation can be described as a direct sum\footnote{See \cite[Definition 6.3.14]{RT2009}, \cite[Definition 7.2.7, Proposition 7.2.8]{RT2009}.} of irreducible ones.

\begin{definition}[Strongly continuous representation]
	Let $G$ be a Lie group and $\pi$ be a representation of $G$ on $\mathcal H_\pi$. We say that $\pi$ is \textit{strongly continuous} if for every $x\in\mathcal{H}_\pi$, the following map is continuous:
	\[
	G\rightarrow \mathcal{H}_\pi,\quad g \in G\mapsto \pi(g)x\in \mathcal{H}_\pi.
	\]
\end{definition}

\begin{definition}[Equivalent representations]
	Let $\pi_1$ and $\pi_2$ be representations of $G$ respectively on the Hilbert spaces $\mathcal{H}_{\pi_1}$ and $\mathcal{H}_{\pi_2}$. A linear mapping $A:\mathcal{H}_{\pi_1}\rightarrow \mathcal{H}_{\pi_2}$, denoted by $A\in \Hom(\pi_1,\pi_2)$, is said to be an \textit{intertwining operator} between the representations $\pi_1$ and $\pi_2$ if for every $g\in G$ we have
	\[
	A\pi_1(g)=\pi_2(g)A.
	\]
	This is easily visualised as the commutation of the diagram
	\[ 
	\begin{tikzcd}
	\mathcal{H}_{\pi_1} \arrow{r}{\pi_1(g)} \arrow[swap]{d}{A} & \mathcal{H}_{\pi_1} \arrow{d}{A} \\
	\mathcal{H}_{\pi_2} \arrow{r}{\pi_2(g)}& \mathcal{H}_{\pi_2}
	\end{tikzcd}
	\]
	for every $g\in G$.
	If there exists an invertible intertwining operator $A\in \Hom(\pi_1,\pi_2)$, then $\pi_1$ and $\pi_2$ are called \textit{equivalent representations} and we write $\pi_1\sim\pi_2$.
\end{definition}

\begin{definition}[Unitary dual]
	The \textit{unitary dual} of a locally compact group $G$, denoted by $\widehat G$, is the set consisting of all equivalence classes of strongly continuous, irreducible, unitary representations of $G$, that is,
	\[
	\widehat{G}:=\big\{[\pi]\,|\,\pi \text{ is a strongly continuous unitary irreducible representation of }G \big\},
	\]
	where $[\pi]=\{\pi_j\,|\,\pi_j\sim\pi\}$.
\end{definition}

For the sake of simplicity, we often abuse the notation and use $\pi$ to denote both a specific representative of an equivalence class and the class itself. At the end of this Chapter, we will classify all the representations of two meaningful groups. For now, let us consider a pivotal example underlining the relation between representations and the phase functions appearing in the definition of the Euclidean Fourier transform.

\begin{example}\label{trigonometrics}
	Let $G=\Rn$. Then for every $\xi\in\Rn$ consider the trigonometric functions
	\[
	e_\xi: \Rn\rightarrow \mathcal{U}(\C),\quad e_\xi(x)=e^{2\pi i \xi\cdot x}.
	\]
	These are strongly continuous unitary representations, clearly irreducible since they are one-dimensional. Therefore the unitary dual of $\Rn$ is 
	\[
	\widehat{\Rn}=\big\{e_\xi\,|\,\xi\in\Rn\big\}\simeq\Rn.
	\]
\end{example}

\indent As underlined above Definition \ref{Liealgebra}, when working with Lie groups, it is key to consider the infinitesimal group, or in other words to work at the \textit{infinitesimal level}. This means looking at what happens at the level of the Lie algebra associated to the group. In order to do this we need to introduce some preliminary concepts.

\begin{definition}[Smooth vector]
	Let $\pi$ be a representation of the Lie group $G$ on the Hilbert space $\mathcal H_\pi$. A vector $v\in\mathcal H_\pi$ is said to be \textit{smooth} or \textit{of type $\mathcal C^\infty$} if the function 
	\[
	G\ni g\mapsto \pi(g)v\in\mathcal H_\pi
	\]
	is a $\mathcal C^\infty$ map. The space of all smooth vectors of a representation $\pi$ is denoted $\mathcal H_\pi^\infty$.
\end{definition}

\begin{definition}[Infinitesimal representation]
	Let $\mathfrak g $ be the Lie algebra of $G$ and let $\pi$ be a strongly continuous representation of $G$ on a Hilbert space $\mathcal H_\pi$. For every $X\in\mathfrak g$ and $v\in\mathcal H_\pi^\infty$ we define 
	\begin{align}\label{definfrep}
	d\pi(X)v:=\lim_{t\rightarrow 0}\frac{1}{t}\Big(\pi\big(\exp_G(tX)\big)v-v\Big)=\frac{d}{dt}\Bigr|_{t=0}\pi\big(\exp_{G}(tX)\big)v.
	\end{align}
	Then $d\pi$ is a representation of $\mathfrak g$ on $\mathcal H_\pi^\infty$ called the \textit{infinitesimal representation} associated to $\pi$. Consequently, using the algebra of $\mathfrak{U}(\mathfrak{g})$ we can define  $d\pi(T)$ for any $T\in\mathfrak{U}(\mathfrak{g})$, by considering its corresponding left-invariant differential operator.
\end{definition}
Recalling the expression of vector fields in terms of the derivative \eqref{vectorfield}, we can write the infinitesimal representation as $d\pi(X)=X\pi(e)$.
\begin{remark}
	In \cite[Proposition 1.7.3]{FR2016} the existence of the limit \eqref{definfrep} in the norm topology of $\mathcal{H}_\pi$ is proved. Furthermore, the authors show that every $d\pi(X)$ leaves $\mathcal H_\pi^\infty$ invariant and, in fact, is a representation of $\mathfrak{g}$ on $\mathcal H_\pi^\infty$ such that for every $X,Y\in\mathfrak{g}$ we have
	\begin{align*}
	d\pi\big([X,Y]\big)=d\pi(X)d\pi(Y)-d\pi(Y)d\pi(X).
	\end{align*}
	For more details, see \cite[Chapter 4]{CG90}.
\end{remark}

\begin{proposition}\label{propertyinfrep}
	Let $G$ be a Lie group and $\mathfrak{g}$ be its Lie algebra. Consider a representation $\pi\in\widehat{G}$ on $\mathcal{H}_\pi$ and the corresponding infinitesimal representation $d\pi$ of $\mathfrak{g}$ on $\mathcal{H}_\pi^{\infty}$. Then, for every $X\in\mathfrak{g}$ the operator $\pi(X)$ is skew-hermitian, meaning 
	\[
	d\pi^*(X)=-d\pi(X).
	\]
\end{proposition}

\begin{proof}
	Consider any two smooth vectors $u,v\in\mathcal{H}_\pi^{\infty}$. Then:
	\begin{align*}
	\Big(u,\frac{i}{t}\big(\pi(\exp_{G}(tX))v-v\big)\Big)&=\Big(u,\frac{i}{t}\big(\pi(\exp_{G}(tX))-\Id\big)v\Big)\\
	&=\Big(\frac{-i}{t}\big(\pi(\exp_{G}(tX))-\Id\big)^*u,v\Big).
	\end{align*}
	We recall that the unitarity of the representation $\pi$ means that for every $g\in G$ we have $\pi^*(g)=\pi(g^{-1})$. Therefore applying Proposition \ref{inverseexponential} we can continue the above calculation as follows:
	\begin{align*}
	\Big(u,\frac{i}{t}\big(\pi(\exp_{G}(tX))v-v\big)\Big)=\Big(\frac{i}{-t}\big(\pi(\exp_{G}(-t X))u-u\big),v\Big).
	\end{align*}
	Taking the limit of both sides as $t\rightarrow0$ and keeping in mind the definition of infinitesimal representation, the above equivalence becomes
	\begin{align*}
	\big(u,d\pi(X)v\big)=\big(d\pi(X)u,v\big),
	\end{align*}
	which yields the desired identity $d\pi^*(X)=-d\pi(X)$.
\end{proof}
\begin{remark}
	By abuse of notation, we will often still denote the infinitesimal representation by $\pi$, therefore, for any $X\in\mathfrak g$, we might write $\pi(X)$ meaning $d\pi(X)$.
\end{remark}

\subsection{The group Fourier transform}\label{Sec:groupFouriertransform}

We conclude this Section introducing the definition of the Fourier transform on groups, the so-called 	\textit{group Fourier transform}. This concludes the analogy with the Euclidean Fourier transform, that is
\begin{align}\label{fouriertransformR}
\widehat{f}(\xi)&:=\int_{\Rn} e^{-2\pi i \xi\cdot x}f(x)dx=\int_{\Rn}e_\xi(-x)f(x)dx\notag\\
&=\int_{\Rn}e_\xi(x)^*f(x)dx=:\widehat{f}(e_\xi),
\end{align}
rewritten in the notation introduced in Example \ref{trigonometrics}.

In accordance with \eqref{fouriertransformR}, we define the group Fourier transform of a suitable complex-valued function on $G$ evaluated at a representation $\pi$ as the integral against the Haar measure of $f$ multiplied by the adjoint of $\pi$. We define this formally as follows.
\begin{definition}[Group Fourier transform]\label{def:groupFouriertransform}
	Let $G$ be a locally compact group and $\pi$ an irreducible unitary representation on some Hilbert space $\mathcal H_\pi$. Given $f\in L^1(G)$ we define its \textit{group Fourier transform at $\pi$} as the operator $\widehat f (\pi): \mathcal H_\pi\rightarrow \mathcal H_\pi$ defined through the inner product by
	\begin{align*}
	\Big(\widehat f(\pi)v_1,v_2\Big)_{\mathcal H_\pi}=\int_G\big(\pi^*(x)v_1,v_2)_{\mathcal H_\pi}f(x)dx.
	\end{align*}
	This is written concisely, using both classical notation and equivalent notation specific to this area \cite{FR2016}, as
	\begin{align*}
	\mathcal{F}_G f(\pi)\equiv\widehat f(\pi)\equiv \pi(f):=\int_{G}\pi^*(x)f(x)d\mu.
	\end{align*}
\end{definition}

After choosing a basis for $\mathcal{H}_\pi$, the group Fourier transform can be seen as a matrix $\widehat{f}(\pi)\in \mathbb{C}^{d_\pi\times d_\pi}$, either finite or infinite-dimensional, according to the dimension of the Hilbert space. 

Starting from this, it is possible to derive properties corresponding to those of the Euclidean Fourier transform, such as the relation between translation and modulation through the Fourier transform. Below we state and prove some of the properties relevant to our purposes. 
\begin{proposition}
	\label{Prop:groupFouriertransform}
	Let $f$ be in $L^1(G)$ and $\pi\in\widehat{G}$. Then the following hold:
	\begin{enumerate}[1.]
		\item The group Fourier transform maps translation to `modulation', that means, for every $g\in G$ we have 
		\begin{align}\label{modulation}
		\mathcal{F}_G(f\circ L_g)(\pi)=\pi(g)\mathcal{F}_G f(\pi),
		\end{align}
		where $L_g$ is the left-translation by $g$.
		\item If $f\in \mathcal{D}(G)\cap L^1(G)$ and $X\in \mathfrak{g}$, then:
		\[
		\mathcal{F}_G(Xf)(\pi)=\pi(X)\mathcal{F}_Gf(\pi).
		\]
	\end{enumerate}
\end{proposition}

\begin{proof}
	\noindent 1. Let $g\in G$. We write the group Fourier transform in \eqref{modulation} explicitly and observe that the left-invariance of the Haar measure allows us to perform a `natural' change of variable to obtain
	\[
	\mathcal{F}_G(f\circ L_g)(\pi)=\int_G\pi^*(x)f(gx)d\mu(x)=\int_G\pi^*(g^{-1}x)f(x)d\mu(x).
	\]
	By definition, the representation $\pi$ is a group homomorphism and is unitary, whence
	\begin{align*}
	\mathcal{F}_G(f\circ L_g)(\pi)&=\int_G\pi^*(g^{-1})\pi^*(x)f(x)d\mu(x)=\int_G\pi(g)\pi^*(x)f(x)d\mu(x)=\\
	&=\pi(g)\mathcal{F}_Gf(\pi).
	\end{align*}
	This completes the proof of \eqref{modulation}. 
	
	\noindent 2. Let $X\in\mathfrak{g}$. Recalling the expression of vector fields in terms of time derivative \eqref{vectorfield}, we can express the Fourier transform of $Xf$ as follows
	\begin{align*}
	\mathcal{F}_G(Xf)(\pi)&=\int_G \pi^*(x)\frac{d}{dt}\Bigr|_{t=0}f(x\exp_{G}(tX))d\mu(x)=\\
	&=\frac{d}{dt}\Bigr|_{t=0}\int_G\pi^*(x)f(x\exp_{G}(tX))d\mu(x),
	\end{align*}
	where we have used the dominated convergence theorem to change the order between integral and derivative. Using the invariance of the Haar measure with respect to the group operation, we can perform a change of variable in the above calculation to find
	\begin{align*}
	\mathcal{F}_G(Xf)(\pi)=\frac{d}{dt}\Bigr|_{t=0}\int_G\pi^*(x\exp_{G}(-tX))f(x)d\mu(x),
	\end{align*}
	where we have used Proposition \ref{inverseexponential}. Therefore, moving the derivative back inside the integral, we can conclude
	\begin{align}\label{formulaforvectorfield}
	\mathcal{F}_G(Xf)(\pi)=\int_G \pi^*(x)Xf(x)d\mu(x)=-\int_GX\pi^*(x)f(x)d\mu(x).
	\end{align}
	Using again the expression \eqref{vectorfield} for vector fields, we can observe 
	\begin{align*}
	X\pi^*(x)=\frac{d}{dt}\Bigr|_{t=0}\pi^*\big(x\exp_{G}(tX)\big)=\frac{d}{dt}\Bigr|_{t=0}\pi^*(x)\pi^*\big(\exp_{G}(tX) \big)=\pi^*(x)\pi^*(X).
	\end{align*}
	Proposition \ref{propertyinfrep} states that $\pi^*(X)=-\pi(X)$, therefore these considerations allow us to continue the calculation \eqref{formulaforvectorfield} to obtain
	\begin{align*}
	\mathcal{F}_G(Xf)(\pi)=\int_G\pi^*(x)\pi(X)f(x)d\mu(x)=\pi(X)\mathcal{F}_Gf(\pi),
	\end{align*}
	which is the desired result.
\end{proof}
We proceed now with a crucial result: the Plancherel theorem for simply connected nilpotent Lie groups. By the \textit{orbit method} \cite[Section 1.8.1]{FR2016}, \cite{CG90}, the unitary dual $\widehat{G}$ is described as a subset of an Euclidean space and this makes it possible to construct a measure on $\widehat{G}$. Indeed, $\widehat{G}$ can be endowed with a measure $\mu$ called the \textit{Plancherel measure}. The existence of such a measure implies the possibility of integrating on the unitary dual and, therefore, of determining an inverse Fourier transform.  
We recall that the Schwartz space on a simply connected nilpotent Lie group $G$ is defined by identifying $G$ with the underlying vector space of its Lie algebra. More precisely:
\begin{definition}[Schwartz functions]
	Let $G$ be a simply connected, nilpotent Lie group and $\mathfrak{g}$ be its Lie algebra. A function $f$ on $G$ is said to be a \textit{Schwartz function} on $G$ if $f\circ \exp_{G}$ is a Schwartz function on $\mathfrak{g}$. We denote by $\mathcal{S}(G)$ the space of Schwartz functions. 
\end{definition}
It can be shown\footnote{For more detail, the interested reader can refer to \cite[Section 4.2 and 4.3, Theorem 4.3.9]{CG90}.} that the aforementioned Plancherel measure is such that for every $f\in \mathcal{S}(G)$
\begin{itemize}
	\item the operator $\pi(f)=\mathcal{F}_Gf(\pi)$ is \textit{trace class}, i.e. the trace norm is finite;
	\item the function $\widehat{G}\ni\pi\mapsto \Tr\big(\pi(f)\big)$ is integrable against $\mu$, and satisfies
	\begin{align*}
	f(0)=\int_{\widehat{G}}\Tr\big(\pi(f)\big)d\mu(\pi).
	\end{align*}
\end{itemize}
Thus, we can formulate the Fourier inversion formula, keeping in mind Property 1. of the group Fourier transform, Proposition \ref{Prop:groupFouriertransform}:
\begin{proposition}[Fourier inversion formula]
	Let $G$ be a simply connected, nilpotent Lie group. Let $f\in\mathcal{S}(G)$ and $\pi\in\widehat{G}$. Then for every $g\in G$ the operators $\pi(g)\pi(f)$ and $\pi(f)\pi(g)$ are trace class and the Fourier inversion formula holds
	\begin{align*}
	f(g)=\int_{\widehat{G}}\Tr\big(\pi(g)\widehat{f}(\pi)\big)d\mu(\pi)=\int_{\widehat{G}}\Tr\big(\widehat{f}(\pi)\pi(g)\big)d\mu(\pi).
	\end{align*}
\end{proposition}
\begin{theorem}[Plancherel theorem]\label{Thm:Plancherelformula}
	Let $f\in\mathcal{S}(G)$ and $\pi\in\widehat{G}$. The operator $\widehat{f}(\pi)=\mathcal{F}_Gf(\pi)$ is Hilbert--Schmidt, i.e. the Hilbert--Schmidt norm
	$$\|\widehat{f}(\pi)\|_{HS}^2=\Tr\big(\widehat{f}(\pi)\widehat{f}(\pi)^*\big)$$
	is finite. Furthermore, the function $\widehat{G}\ni\pi\mapsto\|\widehat{f}(\pi)\|^2_{HS}$ is integrable against $\mu$, and the Plancherel formula holds:
	\begin{align}\label{PlancherelFormula}
	\|f\|^2_{L^2(G)}=\int_G|f(x)|^2dx=\int_{\widehat{G}}\|\widehat{f}(\pi)\|^2_{HS}d\mu(\pi)=\big{\|}\|\widehat{f}(\pi)\|_{HS}\big{\|}^2_{L^2(\widehat{G})}.
	\end{align} 
\end{theorem}
As remarked in \cite[Theorem 1.8.11]{FR2016} there is a unitary extension of Formula \eqref{PlancherelFormula} that holds for every $f\in L^2(G)$, defining the group Fourier transform $\mathcal{F}_G$ on every such $f$, and ensuring it is an isometry onto $L^2(\widehat G)$. Therefore, from now on we use the same notation for the group Fourier transform of square integrable functions.

\subsection{Quantization}\label{SEC:quantization}
Let us assume that $G$ is a nilpotent Lie group. In this Section we will consider fields of operator parametrised by $\widehat G$ not necessarily bounded.

\begin{definition}[Field of operators]  \cite[Definition 1.8.14]{FR2016}
	A $\widehat G$-field of operators on smooth vectors is any family $\sigma=\{\sigma_\pi\,|\,\pi=[\pi]\in\widehat G \}$ of classes of operators where for every $\pi \in \widehat G$ the class $\sigma_\pi$ is defined to be
	\[
	\{\sigma_{\pi_1}:\mathcal H_{\pi_1}^\infty \rightarrow \mathcal H_{\pi_1} \,|\, \pi_1\in\pi=[\pi]\},
	\]
	such that for any $\sigma_{\pi_1},\sigma_{\pi_2}\in\sigma_{\pi}$ we have
	\[
	\pi_1\sim_T\pi_2 \implies \sigma_{\pi_1}T=T\sigma_{\pi_2},
	\]
	where $T$ is an intertwining operator between $\pi_1$ and $\pi_2$. 
\end{definition}
Note that this is well defined because of \cite[Lemma 1.8.12.]{FR2016}, which in turn results from a more general argument in \cite{DM78}.  The latter Lemma states that if two representations $\pi_1$ and $\pi_2$ are intertwined by the the unitary operator $T$, that is, $T\pi_1=\pi_2 T$, then $T$ maps $\mathcal{H}_{\pi_1}^\infty$ onto $\mathcal{H}_{\pi_2}^\infty$ bijectively. Often we will use the shorthand notation
\begin{align*}
\sigma=\{\sigma_{\pi}:\mathcal{H}_\pi^\infty\rightarrow\mathcal{H}_\pi\,|\,\pi\in\widehat{G}\}=\{\sigma_{\pi}\,|\,\pi\in\widehat{G}\}.
\end{align*}

We emphasise that given $T\in\mathfrak{U}(\mathfrak{g})$, according to the definition of infinitesimal representations, the family $\{\pi(T)\,|\,\pi\in\widehat{G}\}$ yields a field of operators acting on smooth vectors and parametrised by $\widehat{G}$.

\begin{definition}[Symbol]\label{symbol}
	A symbol is a field of operators $\sigma=\big\{\sigma (x,\pi):\mathcal H_\pi^\infty\rightarrow\mathcal H_\pi\,|\,(x,\pi)\in G\times\widehat G\big\}$, satisfying for every $x\in G$
	\[
	\sigma(x,\cdot):=\{ \sigma(x,\pi): \mathcal H_\pi^\infty\rightarrow\mathcal H_\pi\,|\,\pi\in\widehat G\}\in L^\infty(\widehat G).
	\]
\end{definition}
Every symbol $\sigma$ is associated with an operator given by the following quantisation formula
\[
\Op(\sigma)\phi(x)=\int_{\widehat G}\Tr\big(\pi(x)\sigma(x,\pi)\widehat\phi(\pi)\big) d\mu(\pi),
\]
for every $\phi\in\mathcal S(G)$ and $x\in G$. As showed in Theorem 5.1.39 in \cite{FR2016}, the quantisation map
\[
\sigma\mapsto \Op(\sigma)
\]
is bijective and linear. Then, given an operator $T$, it makes sense to define its symbol as $\sigma=\sigma_T$. 

Now recall that any finite-dimensional Lie group $G$ has Lie algebra $\mathfrak{g}$ endowed with a basis, say, $\{X_1,\ldots,X_n\}$. We can observe that, according to the Poincar\'e--Birkhoff--Witt Theorem \ref{PBW}, any left-invariant differential operator $T$ on $G$ can be written in a unique way as the finite sum
\[
T=\sum_{|\alpha|\leq M}c_\alpha X^\alpha,
\]
where $\alpha=(\alpha_1,\ldots,\alpha_n)$ is a multi-index, all the coefficients $c_\alpha\in\C$, at least one $c_\alpha$ with $|\alpha|=M$ is non-zero, and $X^\alpha=X_1^{\alpha_1}\dots X_n^{\alpha_n}$.
\begin{definition}
	The non-negative integer $M$ is called the \textit{order} of $T$.
\end{definition}
This allows us to look at any left-invariant differential operator $T$ on $G$ as an element of the enveloping algebra $\mathfrak U(\mathfrak g)$ of the Lie algebra of $G$. Therefore, the family of infinitesimal representations $\{\pi(T)\,|\,\pi\in\widehat G \}$ yields a field of operators that turns out to be the symbol associated with our operator $T$. Hence $T$ can be quantised as follows:
\[
T=\sum_{|\alpha|\leq M}c_\alpha X^\alpha=\Op\Bigg(\sum_{|\alpha|\leq M}c_\alpha\pi(X)^\alpha\Bigg).
\]

\section{Graded Lie groups}\label{Sec:gradedLiegroup}

A natural habitat for extending certain classical results from Euclidean harmonic analysis is provided by graded groups. The main reason for this is the possibility of 'naturally' endowing these groups with a dilation structure and a non-elliptic but hypoelliptic operator, that will play an analogous r\^ole as the elliptic Laplace operator in the Riemannian setting. Moreover, these groups are extremely relevant for applications, since they occur naturally in the geometry of certain symmetric domains, in specific non-elliptic partial differential equations and in complex analysis in several variables, see, e.g., \cite{BFG2016,BLU,CRS2005,VSC1992} and the references therein.

In this Section we begin with the concept of homogeneous Lie group, of which $\Rn$ is the simplest example, leading us to the notion of a graded Lie group. This is a subclass of the homogeneous Lie groups where we can introduce the hypoelliptic Rockland operators. Once all the necessary tools have been introduced, we will present certain properties of the symbols of the Rockland operators that will be fundamental for our studies. We will omit most of the proofs, and present the results in a way congenial to our purposes. More detailed treatments can be found in the recent monograph by Fischer and Ruzhansky \cite{FR2016}, in the notes by Ricci \cite{Ric} or in the more classical reference of Folland and Stein \cite{FS}. 
Let us start with this first subclass of Lie groups, that is, the class of homogeneous Lie groups.
\begin{definition}[Dilations]
	Let $\mathfrak{g}$ be a Lie algebra. A family of \textit{dilations} on $\mathfrak{g}$ is a family $\{D_r\,|\, r>0\}$ of linear mappings from $\mathfrak{g}$ to itself, having the form 
	\begin{align*}
	D_r=\exp(A\ln r)=\sum_{l=0}^{\infty}\frac{1}{l!}(\ln(r)A)^l,
	\end{align*}
	where $\exp$ expresses the exponential of matrices, $A$ is a diagonalisable linear operator on $\mathfrak{g}$ with positive eigenvalues and $\ln(r)$ denotes the natural logarithm of $r>0$. 
\end{definition}
\begin{remark}
	Each $D_r$ is a Lie algebra homomorphism, i.e., a linear map from $\mathfrak{g}$ to itself which respects the Lie bracket on $\mathfrak{g}$:
	\begin{align*}
	[D_rX,D_rY]=D_r[X,Y],\quad\text{for all }X,Y\in\mathfrak{g}.
	\end{align*}
\end{remark}
The eigenvalues of the operator $A$, say $v_1,\dots,v_k$, are called \textit{weights} of the dilations. We can realise the maps $A$ and $D_r$ for every $r>0$ in a matrix form. Choosing a basis of eigenvectors for $A$, we obtain the diagonal matrices
\begin{align*}
A=\begin{pmatrix}
v_1 & \\
& v_2 \\
&		  & \ddots \\
&		  &        & v_k\\
\end{pmatrix}
\quad\text{and}\quad
D_r=\begin{pmatrix}
r^{v_1} & \\
& r^{v_2} \\
&		  & \ddots \\
&		  &        & r^{v_k}\\
\end{pmatrix}.
\end{align*}
Furthermore, we can observe that we can always `adjust' the eigenvalues of $A$. In fact, if $\{D_r\}_{r>0}$ is a family of dilations on $\mathfrak{g}$, then so is $\{\tilde{D}_{r}\}_{r>0}$, with $\tilde{D}_r=D_{r^\alpha}=\exp(\alpha A\ln r)$. Hence, suitably choosing $\alpha$, we can always assume that the minimum eigenvalue of $A$ is $1$.
\begin{definition}[Homogeneous group]
	A simply connected Lie group $G$ whose Lie algebra is endowed with a family of dilations is called a \textit{homogeneous group}.
\end{definition}
We note that we talk about a homogeneous group instead of a homogeneous algebra, because the dilations can be transported to the level of the group by means of the exponential map $\exp_{G}:\mathfrak{g}\rightarrow G$. Indeed, for every $r>0$, we can define the map
\begin{align}\label{dilationonG}
\exp_{G}\circ D_r\circ \exp_{G}^{-1},
\end{align}
which is an automorphism of the group $G$. Therefore, from now on, we will denote by $D_r$ both the dilations on the group $G$ and on the algebra $\mathfrak{g}$.

Furthermore, for the sake of interest, we mention the \textit{homogeneous dimension} $Q$, of a homogeneous group $G$, with respect to a given family of dilations $\{D_r\}_{r>0}$, defined by
\begin{align*}
Q:=\sum_{j=1}^{k}v_j\dim W_{v_j},
\end{align*}
where $v_1,\dots,v_k$ are the eigenvalues (without repetition) of the operator $A$ associated to the family of dilations and $W_1,\dots,W_k$ are the corresponding eigenspaces. From the definition it follows that the topological dimension is always less than or equal to the homogeneous dimension. Furthermore, for any integrable function $f$ on $G$, it holds
\begin{align}\label{integrationdilation}
\int_G f(D_r(x))d\mu(x)=r^{-Q}\int_Gf(x)d\mu(x),
\end{align}
where we have integrated with respect to the Haar measure $d\mu$ on $G$. 
\begin{proposition}\label{nilpotent}
	Every homogeneous Lie group is nilpotent.
\end{proposition}
\begin{proof}
	Consider a homogeneous Lie group $G$ endowed with a family of dilations $\{D_r\}_{r>0}$. Suppose that $v_1,\dots,v_n$ and $W_{v_1},\dots,W_{v_n}$ are, respectively, the corresponding weights and eigenspaces. From the definition of dilation, it follows that for every $j=1,\dots,n$ the following identity holds:
	\begin{align*}
	D_r\bigr|_{W_{v_j}}=r^{v_j}\Id.  
	\end{align*}
	Hence, if we take $X\in W_{v_j}$ and $Y\in W_{v_k}$, then we have
	\begin{align*}
	D_r[X,Y]=[D_r X, D_r Y]=r^{v_j+v_k}[X,Y],
	\end{align*}
	meaning $[X,Y] \in W_{v_j+v_k}$. Therefore, we can conclude that for every $m\in\N_0$ such that $m\times\big(\min\limits_{j=1,\dots,n} v_j\big)>\max\limits_{j=1,\dots,n} v_j$, any Lie bracket of order $m$ in $\mathfrak{g}$ is zero, i.e., $\mathfrak{g}^{(m)}=0$. This means that $\mathfrak{g}$ is nilpotent.
\end{proof}
Another important result, due to ter Elst and Robinson, is as follows.
\begin{proposition}[Lemma 2.2 of \cite{tER:97}]\label{adapted basis}
	Let $\mathfrak{g}$ be a homogeneous Lie algebra endowed with a family of dilations $\{D_r\}_{r>0}$. Then there exist a basis $\{X_1,\dots, X_n\}$ of $\mathfrak{g}$, positive numbers $v_1,\dots,v_n>0$ and an integer $n'\in\{1,\dots,n\}$, such that 
	\begin{enumerate}[i.]
		\item $[\mathfrak{g},\mathfrak{g}]\subset \SPAN\{X_{n'+1},\dots,X_n\}=\R X_{n'+1}\oplus\dots\oplus\R X_{n}$;
		\item for every $1\leq j\leq n$ and for every $r>0$, we have
		\begin{align}\label{basisofeigenvectors}
		D_r(X_j)=r^{v_j}X_j.
		\end{align}
		Furthermore, $X_1,\dots,X_{n'}$, together with all their iterated commutators, generate the Lie algebra $\mathfrak{g}$.
	\end{enumerate}
\end{proposition}
We will refer to the basis $\{X_1,\dots,X_{n'},\dots,X_n\}$ appearing in Proposition \ref{adapted basis} as an \textit{adapted basis} for $\mathfrak{g}$. Moreover, from \eqref{basisofeigenvectors}, we can deduce that $X_1,\dots, X_n$ form a basis of eigenvectors for the operator $A$ associated with the dilations of $\mathfrak{g}$.

Now we introduce the subclass, of the class of homogeneous Lie groups, that we will work with.
\begin{definition}[Graded Lie algebra and graded Lie group]\label{defGraded}
	A Lie algebra $\mathfrak g$ is \textit{graded} if it is endowed with a vector space decomposition:
	\begin{align*}
	\mathfrak g = \oplus_{j=1}^\infty V_j,
	\end{align*}
	such that 
	\begin{itemize}
		\item all but finitely many of the $V_j$'s are $\{0\}$;
		\item $[V_i,V_j]\subset V_{i+j}$.
	\end{itemize}
	A Lie group $G$ is \textit{graded} when it is a simply connected Lie group whose Lie algebra is graded.
\end{definition}
In the above definition, in order to guarantee the exponential map to be a global diffeomorphism between the Lie algebra and the corresponding Lie group, we require the graded group to be simply connected. 

Every graded Lie group is a homogeneous group. In fact, it can be naturally endowed with a family of dilations. Writing $\mathfrak{g}=\oplus_{j=1}^\infty V_j$ we can define the dilations $D_r:=\exp(A\ln r)$, where the operator $A$ is defined as follows:
\begin{align}
\label{EQ:naturaldilationsA}
AX:=jX,\quad\text{for every }X\in V_j.
\end{align}

From the homogeneity of graded groups and Proposition \ref{nilpotent} we can deduce that graded groups are nilpotent. This nilpotency underlines that a gradation may not even exist since, contrapositively, non-nilpotent groups are automatically non-graded. Nonetheless, it is possible to exhibit examples of nilpotent groups which are not graded, for example in \cite{G76,FR2016}. 

A classical but trivial example of a graded Lie group is the abelian group $G=\Rn$ whose Lie algebra is again $\mathfrak{g}=\Rn$. Indeed, we can immediately identify the trivial gradation $V_1=\Rn$. Another pivotal example is the Heisenberg group $\mathbb{H}_n$ that we will describe in detail at the end of this Chapter. We can immediately observe that a gradation is not unique. In fact, if we go back to $\Rn$ any of its vector space decompositions yields a different graded structure. Nevertheless, the natural homogeneous structure for the graded Lie algebra is the same for every gradation.

Another interesting subclass of graded group is given by the stratified Lie group, whose formal definition is as follows.
\begin{definition}[Stratified Lie algebra and stratified Lie group]\label{DefStratified}
	A Lie algebra $\mathfrak g$ is \textit{stratified} if it is graded as $\oplus_{j=1}^\infty V_j$ and its first stratum $V_1$ generates $\mathfrak g$ as an algebra. Thus, every element of the Lie algebra can be written as a linear combination of elements in $V_1$ and their iterated commutators.
	
	A Lie group is \textit{stratified} if it is simply connected and its Lie algebra is stratified.
	
	If there are moreover $p$ non-zero $V_j$'s in the vector space decomposition of the Lie algebra, then both group and algebra are said to be \textit{stratified of step $p$}.
\end{definition}
The relevance to us of stratified Lie groups owes to H\"ormander's work on hypoelliptic operators \cite{H1967}. Before we can state his key theorem, we must introduce the latter concept.
\begin{definition}[Hypoelliptic operator]\label{defHypoellipticity}
	Let $L$ be a linear differential operator on a manifold $M$ with smooth coefficients. Then $L$ is called \textit{hypoelliptic} if for any distribution $u\in\mathcal D'(M)$, the condition $Lu\in\mathcal{C}^{\infty}(N)$ with $N\subseteq M$ yields $u\in\mathcal{C}^\infty(N)$. We can equivalently express this as saying that L is hypoelliptic if for every $u\in\mathcal D'(M)$ we have
	\begin{align}\label{singsupp}
	\singsupp u\subseteq\singsupp Lu, 
	\end{align}
	where the \textit{singular support} of a distribution, denoted by $\singsupp$, is the complement of the largest open subset $N\subseteq M$ in which the given distribution is $\mathcal{C}^\infty$.
\end{definition}
The elliptic \textit{Laplace operator} $\Delta=\sum_{j=1}^{n}\partial_{x_j}^2$ in $\Rn$ or the \textit{heat operator} $\partial_t-\Delta$ in $\R^{n+1}$ are examples of hypoelliptic operator, see \cite{B2014}. More generally, it follows immediately from the definition that every elliptic operator with smooth coefficients is also hypoelliptic, since it satisfies \eqref{singsupp}.

H\"ormander characterised the large class of second order hypoelliptic operators with constant coefficients. His foundational result is as follows.

\begin{theorem}[H\"ormander's theorem]\label{HormanderTheorem}
	Let $M$ be a $\mathcal{C}^\infty$ connected manifold. We consider a family of $\mathcal{C}^\infty$ vector fields on $M$, $\textbf{X}=\{X_0,X_1,\dots,X_k\}$, such that, at every point $x\in M$ the Lie algebra generated by $\textbf{X}$---i.e., generated by $X_1, \dots, X_k$ together with their iterated commutators---is the whole tangent space at this point. Let $c\in\mathcal{C}^\infty(M)$. Then, the second order differential operator
	\begin{align}\label{HormanderSublapl}
	\mathcal{L}=\sum_{j=1}^{k}X_j^2+X_0+c
	\end{align}
	is hypoelliptic in $M$.
\end{theorem}
In this work, we call an operator of the form \eqref{HormanderSublapl} \textit{sub-Laplacian} when $X_0=0$ and $c=0$. A proof of H\"ormander's theorem can be found in \cite[Appendix]{Ric}. 

Recalling the definition of stratified Lie algebra, we observe that any basis for the first stratum $V_1$ forms a H\"ormander system. More precisely, assuming that $V_1$ has dimension $k$ and can be given the basis $\{X_1,\dots,X_k\}$, we can consider its associated hypoelliptic sub-Laplacian operator, given by
\begin{align}\label{slp}
\slp:=-\big(X_1^2+\dots+X_k^2\big).
\end{align}

\subsection{Operators on graded Lie groups}\label{Sec:Operators}
In this Subsection, we briefly present some notions related to operators defined on graded Lie groups. More detail can be found in \cite[Section 3.1]{FR2016}.

Let $G$ be a graded Lie group and $\mathfrak{g}=\oplus_{j=1}^{k}V_j$ its Lie algebra. In the previous Section, we observed that $G$ is naturally endowed with the homogeneous structure, given by
\begin{align*}
D_r(X)=r^j X,\quad \text{for every }X\in V_j\text{ and }j\in\{1,\dots,k\}.
\end{align*} 
Moreover, these dilations can be carried to the level of the group using the exponential map of $G$, as in \eqref{dilationonG}. Therefore, given a function $f$ on $G$ it can be composed with $D_r$. Recalling property \eqref{integrationdilation} relating the Haar measure $d\mu$ and dilations, for all measurable functions $f$ and $\phi$ on $G$, whenever the integrals exist, we have the identity
\begin{align*}
\int_G (f\circ D_r)(x)\phi(x)d\mu(x)=r^{-Q}\int_Gf(x)(\phi\circ D_{\frac{1}{r}})(x)d\mu(x),
\end{align*}     
where $Q$ is the homogeneous dimension of $G$. This allows us to extend the composition of \textit{maps} with dilations, i.e., $f\mapsto f\circ D_r$, to \textit{distributions} in $\mathcal{D}'(G)$ as follows: 
\begin{align*}
\big(f\circ D_r,\phi\big):=r^{-Q}\big(f,\phi\circ D_{\frac{1}{r}}\big),\quad f\in\mathcal{D}'(G),\phi\in\mathcal{D}(G).
\end{align*} 
From the above, we can formulate the definition of homogeneity for operators.
\begin{definition}[Homogeneous operator]
	Let $\nu\in \C$. A linear operator $T:\mathcal{D}(G)\rightarrow \mathcal{D}'(G)$ is \textit{homogeneous of degree $\nu$} or \textit{$\nu$-homogeneous} if for every $\phi\in\mathcal{D}(G)$ and $r>0$ we have
	\begin{align*}
	T(\phi\circ D_r)=r^\nu(T\phi)\circ D_r.
	\end{align*}
\end{definition}
\begin{example}
	Consider a left-invariant vector field $X\in\mathfrak{g}=\oplus_{j=1}^k V_j$ and assume that $X\in V_j$ for a certain $j\in\{1,\dots,k\}$. Then $X$ is an eigenvector corresponding to the eigenvalue $j$ of the operator $A$ associated to the natural dilations of the graded algebra, defined in \eqref{EQ:naturaldilationsA}. Observe that $X$ is homogeneous of degree $j$, since
	\begin{align*}
	X(f\circ D_r)(x)&=\frac{d}{dt}\Bigr|_{t=0}(f\circ D_r)(x\exp_G(tX))=\frac{d}{dt}\Bigr|_{t=0}f(rx\exp_G(r^{j}tX))=\\
	&=r^{j}\frac{d}{dt'}\Bigr|_{t'=0}f(rx\exp_{G}(t'X))=r^j(Xf)(rx).
	\end{align*} 
	Furthermore, let $X_1, \dots,X_n\in \mathfrak{g}$ be homogeneous differential operators of degree $\nu_1,\dots,\nu_n$ respectively. For every multi-index $\alpha\in\N_0^n$, the operator
	\begin{align*}
	X^\alpha=X_1^{\alpha_1}\dots X_n^{\alpha_n}
	\end{align*}  
	has homogeneous degree
	\begin{align*}
	[\alpha]:=\nu_1\alpha_1+\dots+\nu_n\alpha_n.
	\end{align*}
	This differs from the length of $\alpha$, given by $|\alpha|=\alpha_1+\dots+\alpha_n$, which is the order of the operator $X^\alpha$.
\end{example}

\subsubsection{Rockland operators}\label{SectionRockland} 
We conclude our overview by introducing \textit{Rockland operators}, a generalisation of sub-Laplacians to the bigger class that is graded groups, instead of stratified Lie groups. The study of these operators began with Rockland's paper \cite{Rockland}, where he analysed differential operators on the Heisenberg group. More details and a historical review can be found in the fourth Chapter of the monograph \cite{FR2016} or in the work by ter Elst and Robison \cite{tER:97}.

Firstly we will introduce these operators in general. Then we will focus on the case of \textit{positive} Rockland operators. Subsequently we will define Sobolev spaces in terms of Rockland operators and show their independence from the particular choice of operator. Finally we will talk about the associated symbols, i.e., the infinitesimal representations, and present some properties.  

\begin{definition}\label{rockland}
	A differential operator  $\mathcal R$ on a homogeneous group $G$ is defined to be a \textit{Rockland operator} if it is left-invariant, homogeneous of degree $\nu>0$ and injective in each non-trivial irreducible unitary representation. The latter means that for each representation $\pi\in\widehat G$ except the trivial one, the operator $\pi(\mathcal R)$ is injective on the space of smooth vectors $\mathcal H^\infty_\pi$, meaning for every $v\in\mathcal H^\infty_\pi$ we have
	\begin{align}\label{Rockland}
	\pi(\mathcal R)v=0\implies v=0.
	\end{align}
\end{definition}
It can be shown that graded groups provide the most natural context to study Rockland operators. In fact, the existence of a Rockland operator on a homogeneous group forces the group to be graded, as proved in \cite[Proposition 4.1.3.]{FR2016} (the original statement of this result is due to Miller \cite{Miller:80}, even if his proof presents a gap that was filled by ter Elst and Robinson in \cite{tER:97}). Therefore, from now on, we only consider a graded Lie group $G$ of dimension $n$ and its Lie algebra $\mathfrak{g}$. 

\begin{example}\label{exampleRockland}
	\noindent 1. Stratified Lie groups.\\
	\noindent A meaningful example of a Rockland operator is provided by the sub-Laplacian $\slp$ on a stratified Lie group $G$, specified in \eqref{slp}. We recall it is given, in terms of the basis $\{X_1,\dots,X_r\}$ for the first stratum $V_1$ of the Lie algebra $\mathfrak{g}$, by
	\[
	\slp=-\sum_{j=1}^{r}X_j^2.
	\]
	Then $\slp$ is a Rockland operator of homogeneous degree $2$. In fact, it is clearly a homogeneous left-invariant differential operator of degree $2$, so we only need to show the Rockland injectivity condition. Let us consider a non-trivial representation $\pi\in\widehat{G}$ and a smooth vector $v\in\mathcal{H}_\pi^\infty$, such that 
	\begin{align}\label{RC}
	\pi(\slp)v=0.
	\end{align} 
	We want to prove that $v=0$. Considering the inner product on the Hilbert space $\mathcal{H}_\pi$, from \eqref{RC} it follows that 
	\begin{align*}
	0&=\big(\pi(\slp)v,v\big)=-\Big(\big(\pi(X_1)^2v,v\big)+\dots+\big(\pi(X_r)^2v,v\big)\Big)=\\
	&=\big(\pi(X_1)v,\pi(X_1)v\big)+\dots+\big(\pi(X_r)v,\pi(X_r)v\big)=\\
	&=\|\pi(X_1)\|^2_{\mathcal{H}_\pi}+\dots+\|\pi(X_r)\|^2_{\mathcal{H}_\pi},
	\end{align*}
	implying
	\begin{align*}
	\pi(X_1)v=\dots=\pi(X_r)v=0.
	\end{align*}
	We now recall that the system $\{X_1,\dots,X_r\}$ forms a basis for the first stratum $V_1$, which in turn generates $\mathfrak{g}$ as a Lie algebra. This implies that for any vector field $X\in\mathfrak{g}$, we have $\pi(X)v=0$. Finally, since the representation $\pi$ is non-trivial and irreducible, the vector $v$ is forced to be zero. We conclude that every sub-Laplacian on a stratified group is a Rockland operator.
	
	\noindent 2. Graded groups.\\
	\noindent The argument above can be slightly adjusted to construct a `classical' Rockland operator on an $n$-dimensional graded group $G$. Indeed, we can consider a basis $\{X_1,\ldots,X_n\}$ for the Lie algebra $\mathfrak{g}$ and the natural family of dilations $\{D_r\}_{r>0}$ with weights $v_1,\ldots,v_n$, such that for every $j\in\{1,\ldots,n\}$ and for every $r>0$ we have
	\begin{align*}
	D_r X_j=r^{v_j} X_j.
	\end{align*}
	Then, for any common multiple $v_0$ of the weights $v_1,\ldots,v_n$ and constants $c_j>0,~j\in\{1,\ldots,n\}$, the operator
	\begin{align*}
	\mathcal{R}=\sum_{j=1}^n (-1)^{v_0/v_j} c_j X_j^{2v_0/v_j}
	\end{align*}
	is a Rockland operator of homogeneous degree $2v_0$. In fact, an analogous argument to the case of the sub-Laplacian allows us to conclude that for every non-trivial representation $\pi\in\widehat{G}$ and for every smooth vector $v\in\mathcal{H}_\pi^{\infty}$ such that $\pi(\mathcal{R})v=0$, we have $\pi(X_j)^{v_0/v_j}v=0$ for every $j\in\{1,\dots,n\}$. Now we can observe that, for any $p\in\N$ and any $Z\in\mathfrak{U}(\mathfrak{g})$, if $\pi(Z)^pv=0$ then
	\begin{itemize}
		\item for \textit{odd} $p$,
		\begin{align*}
		\pi(Z)^{p+1}v=\pi(Z)\pi^p(Z)v=0;
		\end{align*}
		\item for \textit{even} $p$,
		\begin{align*}
		0=\big(\pi(Z)^pv,v\big)=(-1)^{\frac{p}{2}}\big(\pi(Z)^{\frac{p}{2}}v,\pi(Z)^{\frac{p}{2}}v\big)=(-1)^{\frac{p}{2}}\|\pi(Z)^{\frac{p}{2}}v\|^2,
		\end{align*}
		and therefore $\pi(Z)^{p/2}v=0$.
	\end{itemize}
	The argument above can be applied repeatedly to $Z=X_j$ and $p=\frac{v_0}{v_j},\frac{v_0}{2v_j},\dots$, until we obtain $\pi(X_j)v=0$ for every $j=1,\dots,n$. Then $v$ must be $0$, and therefore $\mathcal{R}$ is a Rockland operator.  
	
	\noindent 3. Integer powers of Rockland operators.\\
	\noindent Given a Rockland operator $\mathcal R$, any power of it is again a Rockland operator. Indeed, given $k\in\N$, $\mathcal R^k$ satisfies the Rockland condition. This owes to the last example yielding $\pi(\mathcal{R})v=0$ whenever $\pi(\mathcal{R}^k)v=0$, for any non-trivial $\pi\in\widehat{G}$ and any $v\in\mathcal{H}_\pi^{\infty}$.
\end{example}

As underlined in the introduction, hypoelliptic differential operators play a crucial r\^ole in several problems both from applied and theoretical fields. This makes the Rockland operators highly relevant to us. In fact, Rockland himself in his work \cite{Rockland} conjectured the equivalence between the Rockland condition and hypoellipticity. This was eventually proved in full generality by Helffer and Nourrigat in \cite[Theorem 0.1]{HN-79}, where they demonstrated the following pivotal result.
\begin{theorem}\label{HelfferNourrigat}
	Let $\mathcal R$ be a left-invariant homogeneous differential operator of degree $\nu$ on a homogeneous Lie group $G$. Then
	\begin{center}
		$\mathcal{R}$ is Rockland $\iff$ $\mathcal{R}$ is hypoelliptic.
	\end{center}
	Furthermore, any operator of the form 
	\begin{align*}
	\mathcal{R}+\sum_{[\alpha]<\nu}c_\alpha X^\alpha, \quad\text{with constants }c_\alpha\in\C,
	\end{align*}
	is also hypoelliptic.
\end{theorem}
\begin{remark}
	Helffer and Nourrigat's theorem, together with H\"ormander's theorem (Theorem \ref{HormanderTheorem}), allows us to deduce immediately that every sub-Laplacian is a Rockland operator.
\end{remark}

The next part of our brief overview about Rockland operators relates their functional and symbolic calculus. We begin by stating the following property that relies on the hypoellipticity.
\begin{proposition}\label{essentiallyselfadjoint}
	Let $\mathcal{R}$ be a Rockland operator on a graded Lie group $G$ and $\pi\in\widehat{G}$. We assume that $\mathcal{R}$ is formally self-adjoint, that is, $\mathcal{R}^*=\mathcal{R}$ as an element of the univeral enveloping algebra $\mathfrak{U}(\mathfrak{g})$ (see Definition \ref{formaladjoint}). Then the operators $\mathcal{R}$ and $\pi(\mathcal{R})$, respectively densely defined on $\mathcal{D}(G)\subset L^2(G)$ and $\mathcal{H}_\pi^\infty\subset\mathcal{H}_\pi$, are essentially self-adjoint.
\end{proposition}
A detailed proof of the above proposition can be found in \cite[Proposition 4.1.15]{FR2016}. Crucially it allows us to obtain the functional calculus of Rockland operators, since we can apply the spectral theorem for unbounded operators \cite[Theorem VIII.6]{RS1980} to the self-adjoint extensions of $\mathcal{R}$ and $\pi(\mathcal{R})$, and will do so momentarily.

However we emphasise a key point here. If we start with a formally self-adjoint Rockland operator and take any self-adjoint extension, the latter admits many useful properties. By restriction of the domain, these properties will also \textit{automatically hold for the original Rockland operator}, where they are well defined.
Hence, for the remainder of this thesis we employ an abuse of notation, by denoting any self-adjoint extensions of $\mathcal{R}$ and $\pi(\mathcal{R})$, respectively on $L^2(G)$ and $\mathcal H_\pi$, respectively by $\mathcal{R}$ and $\pi(\mathcal{R})$ themselves. 
\begin{theorem}[Functional calculus of Rockland operators]\label{RockFuncCalc}
	Let $\mathcal{R}$ be a Rockland operator on a graded Lie group $G$ and $\pi\in\widehat{G}$. We assume that $\mathcal{R}$ is formally self-adjoint. Then there exist spectral measures $E$ and $E_\pi$, corresponding to the self-adjoint extensions of $\mathcal R$ and $\pi(\mathcal R)$ respectively, such that
	\begin{align*}
	\mathcal{R}=\int_{\R} \lambda dE(\lambda)\quad\text{and}\quad\pi(\mathcal R)=\int_\R\lambda dE_\pi(\lambda).
	\end{align*} 
	For any Borel subset $B\subset\R$, we have 
	\begin{align*}
	\pi(E(B))=E_\pi(B).
	\end{align*}
	Furthermore, if $\phi$ is a measurable function on $\R$, the spectral multiplier operator is defined by
	\begin{align}\label{EQ:fc}
	\phi(\mathcal{R}):=\int_{\R}\phi(\lambda)dE(\lambda),
	\end{align}
	and its domain $\Dom\big(\phi(\mathcal{R})\big)$ is the space of functions $f\in L^2(G)$ such that the integral $\int_{\R}|\phi(\lambda)|^2d\big(E(\lambda)f,f\big)$ is finite. In addition, if $\phi\in L^\infty(\R)$, then the spectral multiplier $\phi(\mathcal{R})$ is bounded on $L^2(G)$ and left-invariant, and for any $f\in L^2(G)$, it holds:
	\begin{align}
	\mathcal F\Big(\phi\big(\mathcal{R}\big)f\Big)(\pi)=\phi\big(\pi(\mathcal{R})\big)\widehat f(\pi).
	\end{align}
\end{theorem}
\begin{remark}
	We point out one more time that is remarkable to be able to find the spectral projection formula for Rockland operators $\mathcal{R}$ and their infinitesimal representations $\pi(\mathcal{R})$ when they are only assumed to be formally self-adjoint. The abuse of notation simplifies the topic, by ensuring one uses less heavy notation. A similar abuse is in Fischer--Ruzhansky \cite[Corollary 4.1.16]{FR2016}, but only for the infinitesimal representations of $\mathcal{R}$. Nevertheless, we highlight that in formula \eqref{EQ:fc}, to define $\phi(\mathcal{R})$ we need to consider the proper self-adjoint extension (of) $\mathcal{R}$.
\end{remark}

We are interested in working with \textit{positive} Rockland operators, where the positivity is meant in the operator sense:
\begin{definition}
	An operator $T$ on a Hilbert space $\mathcal{H}$ is \textit{positive} if for any vectors $v,v_1,v_2\in\Dom(T)\subset\mathcal H$ it holds
	\begin{align*}
	\big(Tv_1,v_2\big)_{\mathcal{H}}=\big(v_1,Tv_2\big)_{\mathcal{H}} \quad \text{and} \quad \big(Tv,v\big)_{\mathcal{H}} \geq 0.
	\end{align*}
\end{definition} 
Let us observe that the linear combination with non-negative coefficients of positive operators is clearly a positive operator. Furthermore, an easy calculation\footnote{More details can be found in the proof of part $2$ of Proposition \ref{Prop:groupFouriertransform}.} shows that first order differential operators are formally skew-symmetric, i.e. $\int_G \big(Xf_1\big)f_2=-\int_Gf_1\big(Xf_2\big)$. This is because
\begin{align*}
\int_G(Xf_1)f_2d\mu&=\frac{d}{dt}\Bigl|_{t=0}\int_Gf_1(xe^{tX})f_2(x)d\mu(x)=\frac{d}{dt}\Bigl|_{t=0}\int_Gf_1(y)f_2(ye^{-tX})d\mu(y)=\\
&=-\int_Gf_1(xf_2)d\mu.
\end{align*}
Therefore, given a left-invariant vector field $X\in\mathfrak{g}$ and a positive power $p\in 2\N_0$, the operator $(-1)^{\frac{p}{2}}X^p$ is positive on $G$. Hence the Rockland operators presented in Example \ref{exampleRockland} are all positive. This implies also that \textit{any graded Lie group admits a positive Rockland operator}.

Note that the positivity yields the formal self-adjointness of the operator, and therefore the hypotheses of Proposition \ref{essentiallyselfadjoint} are fulfilled, making available the functional calculus for positive Rockland operators. In addition, we also gain information on their spectra, since it is included in $[0,\infty)$, as is the spectrum of any positive operator. 

If $\mathcal R$ is a positive Rockland operator, then, for every $\pi\in\widehat{G}$, also the operator $\pi(\mathcal{R})$ is positive. In fact, from Theorem \ref{RockFuncCalc}, we know that $\pi(E(B))=E_\pi(B)$. Therefore, since $E$ is supported in $[0,\infty)$, so is $E_\pi$ and thus
\begin{align*}
\big(\pi(\mathcal{R})v,v\big)_{\mathcal{H}_\pi}=\int_{0}^{\infty}\lambda d(E_\pi(\lambda)v,v)_{\mathcal{H}_\pi}\geq 0.
\end{align*}
Furthermore, Hulanicki, Jenkins and Ludwig showed in \cite{HJL1985} that the spectrum of the infinitesimal representations of a Rockland operator $\pi(\mathcal R)$ is \textit{discrete}. This implies that, once we fix an orthonormal basis for the Hilbert space associated with the chosen representation $\pi$, the infinite matrix associated to the operator $\pi(\mathcal{R})$ is a diagonal matrix of the form
\begin{align}\label{eq:symbRock}
\pi(\mathcal R)=\begin{pmatrix}
\pi^2_1 & 0      & \dots    & \dots \\
0      & \pi^2_2 & 0    & \dots \\
\vdots&   0       & \ddots  &          \\
\vdots& \vdots  &             & \ddots
\end{pmatrix},
\end{align}
where the $\pi_j$'s are strictly positive real numbers. Note that, according to Definition \ref{symbol}, the family of infinite representations, i.e. the positive, self-adjoint operators $\pi(\mathcal{R})$ with $\pi\in\widehat{G}$, gives rise to the symbol associated to the positive Rockland operator $\mathcal R$:
\[
\sigma(\mathcal R)=\big\{\pi(\mathcal R)\,|\, \pi\in\widehat G\big\}.
\]
\subsubsection{(Inhomogeneous) Sobolev spaces on graded Lie groups}

Let $\mathcal R$ be a positive Rockland operator of homogeneous degree $\nu$ on a graded Lie group $G$. As before, we also use the notation $\mathcal R$ to denote any self-adjoint extensions of $\mathcal R$ on $L^2(G)$. 

\begin{definition}[Sobolev space]\label{defSobolev}
	For any real number $s\in\R$, the \textit{Sobolev space} associated to $\mathcal R$, denoted by $H^s_\mathcal R(G)$, is the subspace of $\mathcal S'(G)$ obtained by completing $\mathcal S(G)$ with respect to the \textit{Sobolev norm}
	\begin{align*}
	\|f\|_{H^s_\mathcal R(G)}:=\|(I+\mathcal R)^{\frac{s}{\nu}}f\|_{L^2(G)},\quad \forall f\in \mathcal S(G).
	\end{align*}
\end{definition}
We omit the details that can be found in \cite[Chapter 4.4]{FR2016} (or in \cite[Section 4]{F75} where these concepts are introduced in the particular case of stratified Lie groups and sub-Laplacians), but it is possible to show that Sobolev spaces on graded Lie groups satisfy our requirements, in the sense that they behave in an analogous way to---and share many properties with---their Euclidean counterparts. We mention below only the results that will be useful to our purposes.
\begin{itemize}
	
	\item Given a positive Rockland operator of homogeneous degree $\nu$ and $s\in\R$ the Sobolev norm is equivalent to the following norm
	\begin{align}\label{eqSobolevnorm}
	f\mapsto \|f\|_{L^2(G)}+\|\mathcal{R}^{\frac{s}{\nu}}f\|_{L^2(G)}.
	\end{align}
	The interested reader can find a proof in \cite[Theorem 4.4.3]{FR2016}.
	
	\item Left-invariant differential operators act continuously on Sobolev spaces on graded Lie groups. In fact, if we fix a positive Rockland operator $\mathcal{R}$, a left-invariant differential operator $T$ of homogeneous degree $\nu_T$ maps continuously $H^{s+\nu_T}_\mathcal{R}$ to $H^s_\mathcal{R}$ for any $s\in\R$. In terms of norm, this means that there exists a constant $C$ depending on $s$ and $T$ such that for every $\phi\in\mathcal{S}(G)$ we have
	\begin{align}\label{continuosmapping}
	\|T\phi\|_{H^s_{\mathcal{R}}} \leq C \|\phi\|_{H^{s+\nu_T}_{\mathcal{R}}}.
	\end{align}
	The reader can find a proof in \cite[Theorem 4.4.16]{FR2016}.
	
	\item Given a positive Rockland operator $\mathcal R$, for every $s\in\R$ the \textit{dual space} of $H^s_\mathcal{R}(G)$ is isomorphic to $H^{-s}_{\bar{\mathcal{R}}}(G)$ via the distributional duality. For more details consult \cite[Lemma 4.4.7]{FR2016}.
	
	\item Also the interpolation between Sobolev spaces on graded Lie groups works in the same way as for their Euclidean counterparts. In particular, the following holds.
	\begin{theorem}\label{interpolation}
		Let $G_1$ and $G_2$ be graded Lie groups, and $\mathcal{R}_1$ and $\mathcal{R}_2$ be positive Rockland operators respectively on $G_1$ and $G_2$. Let $a_0,a_1,b_0,b_1$ be real numbers. Consider a linear mapping $T$ from $H^{a_0}_{\mathcal{R}_1}(G_1)+H^{a_1}_{\mathcal{R}_1}(G_1)$ to locally integrable functions on $G_2$, and suppose that $T$ maps $H^{a_0}_{\mathcal{R}_1}(G_1)$ boundedly into $H^{b_0}_{\mathcal{R}_2}(G_2)$, and $H^{a_1}_{\mathcal{R}_1}(G_1)$ boundedly into $H^{b_1}_{\mathcal{R}_2}(G_2)$. Then $T$ extends uniquely to a bounded mapping from $H^{a_t}_{\mathcal{R}_1}(G_1)$ to $H^{b_t}_{\mathcal{R}_2}(G_2)$ for $t\in[0,1]$, with
		\begin{align*}
		a_t:=ta_1+(1-t)a_0 \quad\text{and}\quad b_t:=tb_1+(1-t)b_0.
		\end{align*}
	\end{theorem}
	One can find a proof in \cite[Theorem 4.4.9]{FR2016} or \cite[Thereom 4.7]{F75}.
	
	\item The definition of Sobolev spaces is independent of the choice of a Rockland operator. 
\end{itemize}
We are going to prove this last result in Theorem \ref{THM:SobolevindepRockland}, since we will use the same argument later. In order to do this, we need the following.
\begin{lemma}\label{lemma1}
	Let $\mathcal{R}$ be a positive Rockland operator of homogeneous degree $\nu$ on a graded Lie group $G$, and $l\in\N$. Then
	\begin{align*}
	H_{\mathcal{R}}^{l\nu}=\Big\{f\in L^2(G)\,|\, \|X^\alpha f\|_{L^2}<\infty, \forall\alpha\in\N_0^n \text{ such that }[\alpha]=l\nu\Big\},
	\end{align*} 
	and the map
	\begin{align*}
	f\mapsto \|f\|_{L^2}+\sum_{[\alpha]=l\nu}\|X^\alpha f\|_{L^2}
	\end{align*}
	is a norm on $H^{l\nu}_{\mathcal{R}}$ equivalent to the Sobolev norm.
\end{lemma}
\begin{proof}
	According to the Poincar\'e--Birkhoff--Witt theorem (Theorem \ref{PBW}), we can write the Rockland operator as follows
	\begin{align*}
	\mathcal{R}^l=\sum_{[\alpha]=l\nu}c_{\alpha,l}X^\alpha.
	\end{align*} 
	Therefore, for every $\phi\in\mathcal{S}(G)$ the inequality below is straightforward
	\begin{align}\label{eq4}
	\|\mathcal{R}^l\phi\|_{L^2}\leq \max\limits_{[\alpha]=l\nu} |c_{\alpha,l}|\sum_{[\alpha]=l\nu}\|X^\alpha\phi\|_{L^2}.
	\end{align}
	Now define $C=\max_{[\alpha]=l\nu}\{|c_{\alpha,l}|,1\}$. Then if we add $\|\phi\|_{L^2}$ to both sides of \eqref{eq4} and recall the norm equivalence given by \eqref{eqSobolevnorm}, the inequality \eqref{eq4} becomes
	\begin{align}\label{eq5}
	\|\phi\|_{H^{l\nu}_\mathcal{R}}\leq C\Big(\|\phi\|_{L^2}+\sum_{[\alpha]=l\nu}\|X^\alpha \phi\|_{L^2}\Big).
	\end{align}
	On the other hand, according to  \eqref{continuosmapping} the operator $X^\alpha$ continuously maps $H^{[\alpha]}_\mathcal{R}(G)$ to $L^2(G)$. Hence, there exists a constant $c>0$ such that for every $\phi\in\mathcal S(G)$ we have
	\begin{align}\label{eq6}
	\sum_{[\alpha]=l\nu}\|X^\alpha\phi\|_{L^2}\leq c\|\phi\|_{H^{l\nu}_{\mathcal{R}}}.
	\end{align}
	Inequalities \eqref{eq5} and \eqref{eq6} prove the Lemma.
\end{proof}
\begin{theorem}
	\label{THM:SobolevindepRockland}
	The Sobolev spaces on a graded Lie group $G$ associated with \textit{any} positive Rockland operators of the same order coincide.
\end{theorem}
\begin{proof}
	We have previously observed that on a graded Lie group we can always construct positive Rockland operators. Therefore, let $\mathcal{R}_1$ and $\mathcal{R}_2$ be two positive Rockland operators on $G$ of homogeneous degrees $\nu_1$ and $\nu_2$ respectively. Powers of positive Rockland operators are again positive Rockland operators, and in particular $\mathcal{R}_1^{\nu_2}$ and $\mathcal{R}_2^{\nu_1}$ are two positive Rockland operators both of homogeneous degree $\nu=\nu_1\nu_2$. 
	
	For every $l\in\N$, the Sobolev spaces of exponent $l\nu$ associated to both Rockland operators $\mathcal{R}_{1}^{\nu_2}$ and $\mathcal{R}_{2}^{\nu_1}$ will coincide and have equivalent norms, because of Lemma \ref{lemma1}. The interpolation result, Theorem \ref{interpolation}, allows us to extend the equality to Sobolev spaces of all orders $s\geq 0$ and the duality to any $s\in\R$.
\end{proof}
\section{Examples of Lie groups}
In the last Section of this Chapter, we present two examples of Lie groups and their well-known symbolic calculus. The first is the compact special unitary group, $SU(2)$. The second is a pivotal example of a non-compact, locally compact, nilpotent Lie group. It represents a starting model for the study of stratified Lie groups.

The theory for compact groups is well-known and we have not gone into details in this dissertation since it is less interesting and more restrictive for our purposes. Nevertheless, the special unitary group $SU(2)$ has been the first case where we have been able to achieve the results in Chapter \ref{CP:4}. Furthermore, $SU(2)$ serves as base model for the more thought-provoking case of the Heisenberg group. 

\subsection{Compact case: $SU(2)$}\label{SEC:SU2}

The \textit{special unitary group}, denoted by $SU(2)$, is the space of all the matrices $A\in\C^{2\times 2}$ such that
\[
A^*=A^{-1}\quad\text{and}\quad\text{det}(A)=1.
\]
This is a compact Lie group, whose elements can be characterised as follows: a matrix $A\in\C^{2\times2}$ is in $SU(2)$ if and only if it is of the form
\[
A=
\begin{pmatrix}
\alpha & \beta\\
-\bar{\beta} & \bar{\alpha}
\end{pmatrix},
\]
where $\alpha$ and $\beta$ are complex numbers satisfying the identity $|\alpha|^2+|\beta|^2=1$.
From this characterisation, it follows that there exists a bijective homomorphism between the three-dimensional sphere $\mathbb{S}^3=\{(z_1,z_2)\in\C^2~|~|z_1|^2+|z_2|^2=1\}$ and $SU(2)$. Furthermore, the recently proved \textit{Poincar\'e conjecture}\footnote{At the beginning of the $20^{th}$ century, Henri Poincar\'e was interested in the topology of three-dimensional manifolds, and in 1904 he formulated a question that became known as the \textit{Poincar\'e conjecture}, see \cite[pp. 498 and 370]{Poincare}. It is one of the seven Millennium Prize problems defined by the Clay Mathematics Institute in 2000, and is the only one to have been solved: Grigori Perelman proved the conjecture in 2003 \cite{Perelman1,Perelman2,Perelman3}.} shows that every closed, simply connected, three-dimensional manifold $M$ is globally homeomorphic to the three-dimensional sphere, and consequently to $SU(2)$. This means that the results obtained for the special unitary group may be applied to more general manifolds. 

Let us consider the Lie algebra $\mathfrak{su}(2)$. A basis for $\mathfrak{su}(2)$ is given by the so-called \textit{Pauli matrices}
\[
X=\begin{pmatrix}
0 & i\\
i & 0
\end{pmatrix}, \quad
Y=\begin{pmatrix}
0 & -1\\
1 & 0
\end{pmatrix}\quad\text{and }
Z=\begin{pmatrix}
i & 0\\
0 & -i
\end{pmatrix}.
\]
Their commutation relations satisfy the identities below 
\[
[X,Y]=2Z,\qquad[Y,Z]=2X,\qquad\text{and }[Z,X]=2Y.
\]
We define the positive sub-Laplacian operator associated with this vector fields, as follows
\begin{equation}\notag
\slp:=-(X^2+Y^2).
\end{equation}
As seen for a general Lie group, every element $K\in\mathfrak{su}(2)$ can be regarded  as a left-invariant differential operator $D_K:\mathcal{C}^\infty(SU(2))\rightarrow\mathcal{C}^\infty(SU(2))$ satisfying formula  \eqref{vectorfield}. Therefore $D_K$ is defined as follows
\[
D_Kf(U):=\frac{d}{dt}f(U\text{exp}(tK))\big |_{t=0},
\] 
for every $f\in \mathcal{C}^\infty(SU(2))$ and  $U\in SU(2)$. 

The representation theory for $SU(2)$ and its Lie algebra $\mathfrak{su}(2)$ is well known. It provides a simple yet illuminating example of how commutation relations can be used. Furthermore, this representation theory is explicitly used in the study of the more complicated setting of semisimple Lie algebras, see \cite[Part II]{H2003}.

Let us present a classification of all the unitary irreducible representations of $SU(2)$, up to equivalences. Since $SU(2)$ is a compact group, it is possible to apply the \textit{Peter--Weyl Theorem} which decomposes $L^2(G)$ into finite-dimensional subspaces that are invariant under the left and right actions of $G$. See \cite{H2003, RT2009} for more details. Applying this theorem to $SU(2)$, it is possible to obtain a classification of all the representations and, therefore a non-commutative Fourier analysis for $SU(2)$. 

For every $l\in\frac{1}{2}\N_0$, let $V_l$ denote the space of homogeneous polynomials in two complex variables of total degree $2l$, that is, the space of functions of the form
\[
f(z_1,z_2)=a_0z_1^{2l}+a_1z_1^{2l-1}z_2+\dots+a_{2l}z_2^{2l},
\]
where $z_1,z_2\in\C$ and the $a_k$'s are arbitrary complex constants. Then, for every $l\in\frac{1}{2}\N_0$ there exists a unitary irreducible representation of $SU(2)$ on $V_l$
\[
T_l:SU(2)\rightarrow GL(V_l),
\]
given by
\[
T_l(U)f(z):=f(U^{-1}z),
\]
for every $U\in SU(2)$, $f\in V_l$ and $z=(z_1,z_2)^T\in \C^2$. Its finite dimension equals the dimension of the vector space $V_l$, that is, $2l+1$. It is possible to show that every finite dimensional, unitary, irreducible representation of $SU(2)$ is equivalent to one and only one of the $T_l$'s. We omit the details, but we refer the interested reader to \cite{RT2009}. The representations $T_l$ can be interpreted in a matrix form given by:
\begin{align}\notag
\begin{cases}
T_l:=t^l=(t^l_{mn})_{m,n}\in\C^{(2l+1)\times(2l+1)},\quad l\in\frac{1}{2}\N_0,\,-l\leq m,n \leq l,\,l-m,l-n\in\Z\\
t^l(uv)=t^l(u)t^l(v),\hfill \text{with }u,v\in SU(2). 
\end{cases}
\end{align}
We conclude this Subsection, providing an explicit expression for the symbols of the vector fields $X$ and $Y$ according to \cite[Theorem 12.2.1]{RT2009}
\begin{align}\label{EQ:x}
&\sigma_X(l)_{m,n}=-\sqrt{(l-n)(l+n+1)}\delta_{m,n+1}=-\sqrt{(l-m+1)(l+m)}\delta_{m-1,n};\\\label{EQ:y}
&\sigma_Y(l)_{m,n}=-\sqrt{(l+n)(l-n+1)}\delta_{m,n-1}=-\sqrt{(l+m+1)(l-m)}\delta_{m+1,n}.
\end{align}
Consequently, we can calculate the symbol of the sub-Laplacian operator associated with $X$ and $Y$, that is given by the diagonal matrix
\begin{equation}\label{EQ:slp}
\sigma_\slp(l)_{m,n}=(l(l+1)-m^2)\delta_{mn}.
\end{equation}

\subsection{The Heisenberg group}\label{Heisenberg}

The Heisenberg\footnote{Its name owes to the equivalence between its commutator relation \eqref{EQ:commrelHeis} and the commutator relation from quantum mechanics commonly called \textit{Heisenberg's uncertainty principle}. See \cite[Chapter 1]{Fol} for more details.} group $\mathbb{H}_n$ is arguably the most important---and even simplest---example of a non-abelian, non-compact, locally compact, nilpotent, unimodular and stratified Lie group. For a deeper introduction and discussion the reader is invited to consult \cite{FH1987,Fol,T1998,T2008}.

The Heisenberg group has several equivalent realisations one can choose according to context and purpose. The most common approach is to define it as the manifold
\[
\mathbb{H}_n:=\R^n\times\R^n\times\R,
\]
endowed  with the group law
\[
(x,y,t)(x',y',t')=\big(x+x',y+y',t+t'+\alpha(xy'-yx')\big),
\]
where $\alpha \in\R$\textbackslash$\{0\}$, $x,x',y,y'\in\Rn$, $t,t'\in\R$ and $xy=\sum_{j=1}^nx_jy_j$ denotes the usual scalar product on $\Rn$. Its group operation is then a perturbation of the sum in $\R^{2n+1}$ on the third component. In the quantum mechanical analogy, the first two components denote respectively the momentum and the position of a particle, and describe the state of the particle. 

A second possible realisation is to identify the set of triples $(x,y,t)\in\R^n\times\R^n\times\R$ with the group of $(n+2)\times(n+2)$ upper triangular matrices of the form
\[
\left(\begin{array}{c|c|c}
1 	& x^T 	& t+\alpha xy \\
\hline
\mathbf{0}	& I		& 2\alpha y \\
\hline
0	&\mathbf{0}^T		& 1
\end{array}\right)
\]
with the standard matrix product, where $I$ denotes the identity matrix on $\Rn$, $x^T$ is the transpose of $x$ so that $x^Ty=xy$ is again the scalar product and $\mathbf{0}$ denotes the zero-vector in $\Rn$. 

The last realisation we mention rephrases the first in terms of complex numbers. By setting $z=x+ i y$ with $x,y\in\Rn$, we can define
\[
\mathbb{H}_n:=\C^n\times\R,
\] 
with the product
\[
(z,t)(z',t')=(z+z',t+t'-\alpha \, \text{Im}(z\overline{z'})),
\]
where $z\overline{z'}=(x+iy)(x'-iy')=xx'+yy'-ixy'+ix'y$ and $xy$ is the scalar product on $\Rn$.

We will use the first realisation of these three, and as is standard we will take $\alpha=\frac{1}{2}$. It is well known that the Haar measure of $\mathbb{H}_n$ coincides with the Lebesgue measure of $\R^{2n+1}$. The identities 
\begin{align}\notag
(x,y,t)^{-1}&=(-x,-y,-t)\quad\text{and}\\\notag
(x,y,t)(x',y',t')(x,y,t)^{-1}&=(x',y',t'+xy'-yx')
\end{align}
underline that $\mathbb{H}_n$ is not commutative, and its centre $\mathbf{Z}$, i.e., the set of commuting elements, comprises all elements of the form $(0,0,t)$ with $t\in \R$. 

The space of all left-invariant vector fields of $\mathbb{H}_n$, that is, its Lie algebra $\mathfrak{h_n}$, admits the canonical basis
\begin{align}\notag
&X_j=\partial_{x_j}-\frac{y_j}{2}\partial_t,\\\notag
&Y_j=\partial_{y_j}+\frac{x_j}{2}\partial_t,\\\notag
&T=\partial_t,
\end{align}
where $j\in\{1,\ldots,n\}$. These satisfy the well-known commutation relations
\begin{equation}
\label{EQ:commrelHeis}
[X_j,Y_j]=T,\quad j\in\{1,\ldots,n\},
\end{equation}
and $[X_j,Y_k]=0$ for $j\neq k$. Therefore, it is straightforward that the Heisenberg Lie algebra admits the stratification
\[
\mathfrak{h}_n=V_1\oplus V_2,
\]
where $V_1=\sum_{j=1}^n \R X_j \oplus \R Y_j$ and $V_2=\R T$.

The representation theory of the Heisenberg group is well-understood. Indeed we can get a complete classification of all unitary, irreducible representations using the \textit{Stone--von Neumann Theorem}. Before stating the latter we construct a family of representations that will give rise to the classification.

The following lemma, proved using \textit{Schur's Lemma}, assigns to each representation of $\mathbb{H}_n$ a real number related to the action of the representation on the centre of the Heisenberg group.
\begin{lemma}\label{lemmaCentre}
	Let $\pi$ be a unitary, irreducible representation of $\mathbb{H}_n$. Then there exists a non-zero $\lambda=\lambda(\pi)\in\R$ such that
	\[
	\pi(0,0,t)=e^{i\lambda t},
	\]
	interpreted as a multiplication operator on $L^2(\Rn)$.
	Furthermore, if $\pi_1$ and $\pi_2$ are two equivalent unitary irreducible representations then $\lambda(\pi_1)=\lambda(\pi_2)$.
\end{lemma}
The next result provides a characterisation of one-dimensional representations.
\begin{lemma}\label{trivialoncentre}
	If $\pi$ is a unitary, irreducible representation of $\mathbb{H}_n$ and is trivial on the centre, i.e. $\lambda(\pi)=0$, then the representation $\pi$ is one-dimensional and there exists $(\xi,\eta)\in\Rn\times\Rn$ such that
	\[
	\pi(x,y,t)=e^{i(\xi x+\eta y)},
	\]
	for every element $(x,y,t)$ in $\mathbb{H}_n$.
\end{lemma}
Then the following result constructs the explicit form of the desired representations.
\begin{proposition}\label{Prop1}
	Let $\sigma_1,\sigma_2$ be unitary representations of $\Rn$ on the same Hilbert space $\mathcal{H}$, such that
	\begin{align}\label{comrel}
	\sigma_1(x)\sigma_2(y)=e^{i\lambda x y}\sigma_2(y)\sigma_1(x),
	\end{align}
	for a certain $\lambda\in\R\setminus\{0\}$. Then there exists a unique unitary irreducible representation $\pi$ of $\mathbb{H}_n$ on $\mathcal{H}$ such that
	\begin{align}\notag
	&\pi(x,0,0)=\sigma_1(x),\\\notag
	&\pi(0,y,0)=\sigma_2(y),\\\notag
	&\pi(0,0,t)=e^{i \lambda t},
	\end{align}
	for every $(x,y,t)\in\mathbb{H}_n$.
\end{proposition}
From a physical perspective the unitary representations $\sigma_1$ and $\sigma_2$ are generated respectively by momentum and position operators in quantum mechanics, see \cite{T2008}. 

Proofs for these three results can be found in, e.g., \cite[Chapter VI]{R2001}. The standard proof of Proposition \ref{Prop1} uses the following explicit form for the representation:
\[
\pi(x,y,t)=e^{i\lambda (t-\frac{1}{2}x y)}\sigma_1(x)\sigma_2(y).
\]
In particular we can choose
\[
\mathcal{H}= L^2(\Rn),
\]
and
\begin{align}\notag
&[\sigma_1(x)\phi](u)=\phi(u+x),\\\notag
&[\sigma_2(y)\phi](u)=e^{i y u }\phi(u).
\end{align}
Then the property \eqref{comrel} is satisfied for $\lambda=1$ and the resulting representation is given by
\begin{align}\label{schrodinger}
\pi_1(x,y,t)\phi(u)=e^{i(t-\frac{1}{2}xy+yu)}\phi(u+x).
\end{align}
Furthermore, it is possible to prove that the representation \eqref{schrodinger} is irreducible \cite{T1998}.

By dilating this irreducible representation, we can obtain a representation corresponding to each $\lambda\in\R\setminus\{0\}$. In order to do this, for every $\lambda \ne 0$ consider the following continuous automorphism of $\mathbb{H}_n$:
\[
\tau_{\lambda}(x,y,t)=(x,\lambda y, \lambda t ).
\]   
The composition of a unitary irreducible representation with $\tau_\lambda$ is again a unitary, irreducible representation, and thus we recover the desired representation
\begin{align}\label{schrodinger2}
\pi_\lambda(x,y,t)\phi(u)=[\pi_1(x,\lambda y,\lambda t)](u)=e^{i\lambda(t-\frac{1}{2}xy+yu)}\phi(u+x).
\end{align}
\begin{definition}
	Let $\lambda\in\R\setminus\{0\}$. The infinite-dimensional representation
	\[
	\pi_\lambda:\mathbb{H}_n\rightarrow\mathcal{U}(L^2(\Rn)),
	\]
	given by \eqref{schrodinger2} is called the \textit{Schr\"odinger representation}.
\end{definition}
Thus the Heisenberg group is constructed as a group of unitary operators acting on $L^2(\Rn)$.

Now we state the Stone--von Neumann Theorem, which says that any irreducible unitary representation of $\mathbb{H}_n$ that is non-trivial on the centre is equivalent to some $\pi_\lambda$. We roughly follow the presentation in \cite[Section 1.5]{Fol}.
\begin{theorem}[Stone--von Neumann]
	Let $\pi$ be a unitary representation of $\mathbb{H}_n$ on a Hilbert space $\mathcal H$, such that $\pi(0,0,t)$ is the multiplication operator $e^{i\lambda t}$ for some $\lambda\in\R\setminus\{0\}$. Then there is an orthonormal basis $\{v_\alpha\}$ for $\mathcal{H}$ such that the spans $\mathcal{H}_\alpha=\SPAN\{v_\alpha\}$ give the orthogonal decomposition
	\[
	\mathcal{H}=\oplus \mathcal H_{\alpha}.
	\]
	Each $\mathcal{H}_\alpha$ is invariant under $\pi$, and each $\pi|_{\mathcal H_\alpha}$ is unitarily equivalent to $\pi_\lambda$.\\
	In particular, if $\pi$ is irreducible then $\pi$ is equivalent to $\pi_\lambda$.
\end{theorem}
We conclude the Section by stating that, together with Lemma \ref{trivialoncentre}, the Stone--von Neumann Theorem gives rise to a complete classification of all unitary, irreducible representations of $\mathbb{H}_n$.
\newpage
\begin{theorem}
	Let $\pi$ be a unitary, irreducible representation of $\mathbb{H}_n$ on a Hilbert space $\mathcal H$. If $\pi$ is trivial on the centre, then $\pi$ is one-dimensional and is given by
	\[
	\pi(x,y,t)=e^{i(\xi x +\eta y)},
	\]
	for some $(\xi,\eta)\in\Rn\times\Rn$. Otherwise, $\pi$ is equivalent to $\pi_\lambda$ for some $\lambda\in\R\setminus\{0\}$.
\end{theorem}


\chapter{Rockland wave equation}
\label{CP:3}
In the `80s, the question of the well-posedness of the Cauchy problem associated to the wave equation
\begin{equation}\label{Cauchy}
\left\{\begin{array}{rclc}
\partial ^2_{t}u(t,x)-a(t)\partial ^2_{x}u(t,x)	&=&0,			&(t,x)\in [0,\tau]\times\Rn,\\
u(0,x)	&=&\varphi(x),	&x\in\Rn,\\
\partial _t u(0,x)	&=&\psi(x),		&x\in\Rn,
\end{array}\right.
\end{equation}
with a time-dependent coefficient $a(t)\geq 0$ has been extensively studied. For the sake of completeness, we recall the definition of well-posedness:
\begin{definition}
	A Cauchy problem is said to be \textit{well-posed} in a certain class of real functions or distributions, say $\mathcal{F}{(\Rn_x)}$, if, for any initial data $\varphi$ and $\psi$ in $\mathcal{F}{(\Rn_x)}$, it admits a unique solution $u(t,x)$ in $\mathcal{C}^1([0,\tau],\mathcal{F}(\Rn_x))$, and the mapping 
	\begin{align*}
	(\varphi,\psi)\mapsto u 
	\end{align*} 
	is continuous. Furthermore, if the Cauchy problem is well-posed, then the corresponding equation is said to be \textit{hyperbolic}.
\end{definition}

In 1979, Colombini, de Giorgi and Spagnolo proved in their seminal paper \cite{CDS79} well-posedness results for the Cauchy problem \eqref{Cauchy}. In fact, it became well known that \eqref{Cauchy} is well posed in $\mathcal{C}^\infty(\Rn)$ and $\mathcal{D}'(\Rn)$ provided that there is a constant $\lambda$ such that
\begin{align*}
a(t)\geq\lambda>0,
\end{align*}
and that $a$ is `regular enough'. But as soon as the coefficient $a$ is allowed either to vanish at some points, or to be `not regular enough'---for example the H\"older regularity $a\in \mathcal C^\alpha$ for $0<\alpha<1$ is sufficient---then the problem \eqref{Cauchy} might not be well-posed in $\mathcal C^\infty$ or $\mathcal{D}'$. We note that Gevrey regularity appears naturally in the latter cases, postponing the discussion of Gevrey functions to Chapter \ref{CP:4}.

Shortly after, Colombini and Spagnolo \cite{CS1982} and Colombini, Jannelli and Spagnolo \cite{Colombini-Jannelli-Spagnolo:Annals-low-reg} considered the one-dimensional case  and the $n$-dimensional case of the Cauchy problem \eqref{Cauchy}, respectively. In particular the former explicitly constructed a smooth, positive $a(t)$ such that $a(t)\equiv 0$ for $\rho< t\leq \tau$, and smooth initial data $\varphi$ and $\psi$ in a way that \eqref{Cauchy} has no solutions in $\mathcal{C}^1([0,\tau],\mathcal{D}'(\R))$. The latter proved the existence of a smooth positive function $a$ satisfying $a(0)=0$ such that the corresponding wave equation loses the uniqueness of the solution.

The purpose of this Chapter is to extend the study of the well-posedness of the Cauchy problem for the wave equation with time-dependent coefficients to graded groups. Indeed, graded groups are the natural setting to generalise many questions of Euclidean harmonic analysis, as shown in Chapter \ref{CP:2}.

We underline that the whole Chapter is based on the preprint paper \cite{RT2017}.

\section{Main results}\label{SEC:weR}

Let $G$ be a graded Lie group, as defined in Definition \ref{defGraded}. We are interested in the well-posedness of the Cauchy problem
\begin{equation}\label{CP}
\left\{\begin{array}{rclc}
\partial_t^2 u(t,x)+a(t)\mathcal R u(t,x)	&=&0,		&(t,x)\in[0,\tau]\times G,\\
u(0,x)	&=&u_0(x),	&x\in G,\\
\partial_t u(0,x)	&=&u_1(x),	&x\in G,
\end{array}\right.
\end{equation}
for the time-dependent propagation speed $a=a(t)\geq 0$, where $\mathcal{R}$ is a positive Rockland operator, which is defined in Definition \ref{rockland}.

In the case of $G=\Rn$ and $\mathcal R=-\Delta=-\sum_{j=1}^n\frac{\partial^2}{\partial x_j^2}$, the equation in \eqref{CP} is the usual wave equation with time-dependent propagation speed, and the full Cauchy problem becomes equivalent to \eqref{Cauchy}. However in the following situations we will obtain new results:
\begin{itemize}
	\item[(i)] $G=\mathbb H^n$ is the Heisenberg group and $\mathcal R$ is the positive Kohn--Laplacian on $G$.
	\item[(ii)] $G$ is a stratified Lie group and $\mathcal R$ is a sub-Laplacian on $G$.
	\item[(iii)] $G$ is a graded Lie group in the sense of Folland and Stein \cite{FS} and $\mathcal R$ is any positive Rockland operator on $G$, i.e. any positive left-invariant homogeneous hypoelliptic differential operator on $G$.
\end{itemize} 
In fact, we prove our results for the latter case (iii), the other two being special instances. In particular, if $G$ is the Euclidean space $\Rn$, the Heisenberg group $\mathbb H^n$, or any stratified Lie group, the case (iii) allows one to take $\mathcal R$ to be an operator of \textit{any order}, as long as it is a positive left-(or right-)invariant homogeneous hypoelliptic differential operator. In $\Rn$ the related so-called $p$-evolution equations have been studied in, e.g., \cite{CHR08a,CHR08b, CC10}, with more restrictive conditions on $a(t)$ than those considered in this work.

For $a(t)\equiv 1$, $G$ the Heisenberg group $\mathbb H^n$ and $\mathcal R$ a positive sub-Laplacian, the wave equation \eqref{CP} was studied by M\"uller and Stein \cite{Muller-Stein:Lp-wave-Heis} and Nachman \cite{Nachman:wave-Heisenberg-CPDE-1982}.
Other noncommutative settings with $a(t)\equiv 1$ have been analysed as well, see, e.g., Helgason \cite{Helgason:wave-eqns-hom-spaces-1984}.
For $G$ a compact Lie group and $-\mathcal R$ any H\"ormander sum of squares of vector fields on $G$ the problem \eqref{CP} was studied in \cite{GR2015}, and so the results of this Chapter provide a nilpotent counterpart of the results there.

Apart from the intrinsic value of studying subelliptic and related operators on stratified or graded Lie groups, in view of the celebrated lifting theorem of Rothschild and Stein \cite{Rothschild-Stein:AM-1976}, these settings are the model cases for many corresponding problems for general partial differential operators on manifolds.

From the point of view of the time-dependent coefficient $a(t)$, we aim to carry out a comprehensive analysis. Thus we distinguish between the following four cases:
\begin{enumerate}[\text{$\quad$ Case} 1:]
	\item $a\in\mathcal {\rm Lip}([0,\tau])$, $a(t)\geq a_0>0$;
	\item $a\in\mathcal C^\alpha([0,\tau])$, $0<\alpha<1$, $a(t)\geq a_0 >0$;
	\item $a\in\mathcal C^l([0,\tau])$, $l \geq 2$, $a(t)\geq 0$;
	\item $a\in\mathcal C^\alpha([0,\tau])$, with $0<\alpha<2$, $a(t)\geq0$.
\end{enumerate}
The first is the simplest situation since $a$ is regular and time-non-degenerate, while in the fourth we have an irregular coefficient allowed to vanish at some points. The second and third situations are `intermediate' cases, in the sense that we have either sufficient regularity or strict positivity. We distinguish between these cases because the results and methods of proofs are rather different. In Case 1, we develop a `classical' approach based on the use of the symmetriser, although in Case 3 the presence of degeneracies is dealt with by constructing the so-called \textit{quasi-symmetriser}. In Case 2, we proceed with an argument based on the regularisation and separation of characteristic roots. We observe that in the fourth Case we choose the order of the H\"older regularity of the coefficient to vary between zero and two, $0<\alpha<2$, in order to provide a setting similar to the one in Case 2, and develop a similar proof. In fact if $a\in\mathcal C^\alpha([0,\tau])$ with $0<\alpha<2$, then the characteristic roots belong to $\mathcal C^{\frac{\alpha}{2}}([0,\tau])$, where $0<\frac{\alpha}{2}<1$, analogously to Case 2.

An outline of this Chapter is as follows. Shortly we will recall the notions necessary to stating the main theorem, some comments and a corollary. In Section \ref{SEC:ODE} we examine and prove energy estimates for certain second-order ordinary differential equations whose parameter dependence is explicitly given. In Section \ref{SEC:proof} we finish the proofs of all four claims.

Fundamental to all the proofs is application of the global Fourier analysis on $G$, developed in Chapter \ref{CP:2}, to the Cauchy problem \eqref{CP}. This allows one to view the Rockland wave equation as an infinite system of equations with coefficients depending only on $t$, leading to a range of sharp results based on the behaviour of the coefficient $a(t)$. We underline that the infinite system of equations is due to the fact that the group Fourier transform produces operators. Nevertheless, the properties of symbols of positive Rockland operators allow one to decouple the system completely into independent scalar equations, as we will see in the third Section.

We note that if the operator $\mathcal R$ is not elliptic, the local approach to the Cauchy problem \eqref{CP} is problematic since the equation is only {\em weakly hyperbolic} in Case 1 above. Consequently, since equation \eqref{CP} in local coordinates has principal symbol with variable multiplicities, very little is known about its well-posedness. For such problems, only a few special results are available for second-order operators, for example from Nishitani \cite{Nishitani:BSM-1983} or Melrose \cite{Melrose:wave-subelliptic-1986}.
Non-Lipschitz coefficients have been also seen analysis, e.g., Colombini and M\'etivier \cite{CM} or Colombini and Lerner \cite{CL}.

In $\Rn$ with $\mathcal R$ being the Laplacian, the case of $a$ being less than H\"older-regular---such as discontinuous or measure-valued---has been considered in \cite{Garetto-Ruzhansky:ARMA}. However, such low regularity requires very different methods from those applied here.

To formulate our results, let us briefly recall some necessary notation that has been introduced in Chapter \ref{CP:2}. (We refer mainly to Folland and Stein \cite{FS} and Fischer and Ruzhansky \cite{FR2016}).
Let $G$ be a {\em graded Lie group}, i.e., a simply connected Lie group such that its
Lie algebra $\mathfrak g$ has  a vector space decomposition
\begin{equation}\label{EQ:graded}
\mathfrak g = \oplus_{j=1}^\infty V_j,
\end{equation} 
such that all but finitely many of the $V_j$'s are $\{0\}$ and 
$[V_i,V_j]\subset V_{i+j}$, see Definition \ref{defGraded}. A special case analysed in detail by Folland \cite{F75} is that of {\em stratified Lie groups}, where the first stratum $V_1$ generates $\mathfrak g$ as an algebra, as in Definition \ref{DefStratified}. The prototypical example of such a Lie group is the Heisenberg group, which is the object of study in Subsection \ref{Heisenberg}. In general, graded Lie groups are necessarily homogeneous and nilpotent. In this way any graded Lie group can be viewed as $\Rn$ with a polynomial group law. 

Let $\mathcal R$ be a positive Rockland operator on $G$, that is, a positive (in the operator sense) left-invariant differential operator which is homogeneous of degree $\nu>0$ and which satisfies the so-called Rockland condition. The latter means that for each non-trivial representation $\pi\in\widehat G$ the operator $\pi(\mathcal R)$ is injective on the space of smooth vectors $\mathcal H^\infty_\pi$, i.e.,
\begin{align}\label{Rockland}
\forall v\in\mathcal H^\infty_\pi \quad \pi(\mathcal R)v=0\implies v=0.
\end{align}
Alternative characterisations of such operators have been considered by Rockland \cite{Rockland} and Beals \cite{Beals-Rockland}, until the definitive result of Helffer and Nourrigat \cite{HN-79} saying that {\em Rockland operators are precisely the left-invariant homogeneous hypoelliptic differential operators on $G$.} A Rockland operator exists on a homogeneous Lie group if and only if the Lie group is graded \cite{Miller:80, tER:97}, as discussed in Subsection \ref{SectionRockland}. As mentioned there, if $G$ is a stratified Lie group and $\{X_1,\dots,X_k\}$ is a basis for the first stratum of its Lie algebra, then the positive sub-Laplacian 
\[
\slp=-\sum_{j=1}^{k}X_j^2
\]
is a positive Rockland operator. Moreover, for any $m\in\mathbb N$, the operator
$$
\mathcal R=(-1)^{m}\sum_{j=1}^{k}X_j^{2m}
$$
is also a positive Rockland operator on the stratified Lie group $G$. More generally, for any $n$-dimensional graded Lie group $G$, given a basis $X_1,\ldots,X_n$ of its Lie algebra $\mathfrak g$ with dilation weights $\nu_1,\ldots,\nu_n$ such that
\begin{equation}\label{EQ:dils}
D_r X_j=r^{\nu_j} X_j,\quad j=1,\ldots,n,\; r>0,
\end{equation} 
where $D_r$ are dilations on $\mathfrak g$, then for any common multiple $\nu_0$ of $\nu_1,\ldots,\nu_n$
the operator
$$
\mathcal R=\sum_{j=1}^n (-1)^{\frac{\nu_0}{\nu_j}} a_j X_j^{2\frac{\nu_0}{\nu_j}},\quad a_j>0,
$$
is a Rockland operator of homogeneous degree $2\nu_0$. 

To formulate our results we will need two scales of spaces, namely, Sobolev and Gevrey spaces, adapted to the setting of graded Lie groups. In Chapter \ref{CP:2}, we have introduced the (inhomogeneous) Sobolev spaces, see Definition \ref{defSobolev}. We postpone a deeper discussion about the so-called Gevrey spaces to the next Chapter. 

Let $G$ be a graded Lie group and let $\mathcal R$ be a positive Rockland operator of homogeneous degree $\nu$. For any real number $s\in\R$, we consider the Sobolev space $H^s_\mathcal R(G)$. In Theorem \ref{THM:SobolevindepRockland}, we saw that these spaces do not depend on the choice of the Rockland operator $\mathcal R$. We will prove later on in this Chapter that these spaces are well suited to treat Case 1 above. But even in the Euclidean case, with the elliptic Laplace operator in place of the hypoelliptic Rockland operator in the wave equation \eqref{CP}, if the coefficient $a(t)$ is not Lipschitz regular or may become zero, the Gevrey spaces appear naturally (see, e.g., Bronshtein \cite{Bronshtein:TMMO-1980}) since we can no longer expect well-posedness in $\mathcal C^\infty(G)$ or $\mathcal D'(G)$. Thus, it is necessary to introduce another class of functions: given any $s\geq 1$, we define the Gevrey-type space of functions on the graded Lie group $G$
\begin{align}\label{G}
\mathcal G^s_{\mathcal R}(G):=\{f\in\mathcal{C}^\infty(G)\,|\,\exists A>0\,:\,\|e^{A\mathcal R^{\frac{1}{\nu s}}}f\|_{L^2(G)}<\infty\}.
\end{align}

This space is the subelliptic version of the usual Gevrey space. Indeed the class of $\mathcal{G}^s_{\mathcal{R}}(G)$ generalises the class of Gevrey spaces in the sense that $\mathcal{R}$ being the Laplacian would imply \eqref{G} is the `classical' Gevrey space of index $s$. Garetto and Ruzhansky \cite{GR2015} were the first to describe spaces similar to \eqref{G} but constructed from a smaller class of operators---a sum of squares of vector fields---and from compact Lie groups. Dasgupta and Ruzhansky considered classical Gevrey spaces \eqref{G} and corresponding spaces of ultradistributions, on compact Lie groups \cite{DR2014} and on compact manifolds\footnote{In fact they considered a general elliptic operator, not simply the Laplacian.} \cite{DR2016}. In keeping with the next Chapter, we refer to these spaces as the \textit{classical Gevrey spaces} when they are defined with respect to all the vector fields generating the Lie algebra, or, equivalently, with respect to the Laplacian operator.

By an argument similar to both \cite[Theorem 2.4]{DR2014} for elliptic operators and to one for sub-Laplacians that we will develop thoroughly in the next Chapter, it can be shown that if $\mathcal R$ is a positive Rockland operator of homogeneous degree $\nu$, then a function $f$ belongs to $\mathcal G^s_{\mathcal R}(G)$ if and only if there exist constants $B,C>0$ such that for every  $k\in\mathbb N_0$ we have
\begin{equation}\label{EQ:Gevchar}
\|\mathcal R^k f\|_{L^2(G)}\leq C B^{\nu k} ((\nu k)!)^s.
\end{equation} 
Since Sobolev spaces do not depend on the particular choice of Rockland operator used in their definition, as was shown in Theorem \ref{THM:SobolevindepRockland}, the characterisation \eqref{EQ:Gevchar} of Gevrey spaces implies that the same holds for $\mathcal G^s_{\mathcal R}(G)$. 

Let us proceed by formulating the main well-posedness results for the Cauchy problem \eqref{CP} for the Rockland wave equation. For convenience in stating and reading the next theorem we recall below the four types of coefficient $a(t)$ considered, being the four cases to which we will refer throughout the rest of this Chapter.
\begin{enumerate}[\textbf{$\qquad$ Case }\bf 1:]
	\item $a\in\mathcal {\rm Lip}([0,\tau])$, $a(t)\geq a_0>0$;
	\item $a\in\mathcal C^\alpha([0,\tau])$, $0<\alpha<1$, $a(t)\geq a_0 >0$;
	\item $a\in\mathcal C^l([0,\tau])$, $l \geq 2$, $a(t)\geq 0$;
	\item $a\in\mathcal C^\alpha([0,\tau])$, with $0<\alpha<2$, $a(t)\geq0$.
\end{enumerate}

\begin{theorem}\label{THM:main}
	Let $G$ be a graded Lie group and let $\mathcal R$ be a positive Rockland operator of homogeneous degree $\nu$. 
	Let $T>0$.
	Then in the above Cases 1-4 the following results hold:
	\begin{enumerate}[\text{Case} 1:]
		\item Given $s\in\R$, if the initial Cauchy data $(u_0,u_1)$ are in $H_\mathcal R^{s+\frac{\nu}{2}}(G)\times H_\mathcal R^s(G)$, then there exists a unique solution of \eqref{CP} in 
		$$\mathcal C\left([0,\tau],H_\mathcal R^{s+\frac{\nu}{2}}(G)\right)\cap\,\mathcal C^1([0,\tau],H_\mathcal R^s(G)),$$
		satisfying the following inequality for all $t\in [0,\tau]$:
		\begin{align}\label{inequality case 1}
		\|u(t,\cdot)\|^2_{H_{\mathcal R}^{s+\frac{\nu}{2}}}+\|\partial_t u(t,\cdot)\|^2_{H_{\mathcal R}^{s}}\leq C(\|u_0\|^2_{H_{\mathcal R}^{s+\frac{\nu}{2}}}+\|u_1\|^2_{H_{\mathcal R}^{s}});
		\end{align}
		
		\item If the initial Cauchy data $(u_0, u_1)$ are in $\mathcal G^s_\mathcal R(G)\times \mathcal G^s_\mathcal R(G)$, then there exists a unique solution of \eqref{CP} in $\mathcal C^2([0,\tau],\mathcal G_\mathcal R^s(G))$, provided that
		\[
		1\leq s <1 +\frac{\alpha}{1-\alpha};
		\]
		
		\item If the initial Cauchy data $(u_0, u_1)$ are in $\mathcal G^s_\mathcal R(G)\times \mathcal G^s_\mathcal R(G)$, then there exists a unique solution of \eqref{CP} in $\mathcal C^2([0,\tau],\mathcal G_\mathcal R^s(G))$, provided that
		\[
		1\leq s <1 +\frac{l}{2};
		\]
		
		\item If the initial Cauchy data $(u_0, u_1)$ are in $\mathcal G^s_\mathcal R(G)\times \mathcal G^s_\mathcal R(G)$ then there exists a unique solution of \eqref{CP} in $\mathcal C^2([0,\tau],\mathcal G_\mathcal R^s(G))$, provided that
		\[
		1\leq s <1 +\frac{\alpha}{2}. 
		\]
	\end{enumerate}
\end{theorem}

As will follow from the proof, in Cases 2 and 4 one can respectively take the equalities $s=1 +\frac{\alpha}{1-\alpha}$ and $s=1 +\alpha$, provided that $T>0$ is small enough.
We refer to \cite{GR2012, GR2013} concerning the sharpness of the above Gevrey indices in the case of $G=\Rn$ and $\mathcal R=\Delta$, and for further relevant references for that case.

Let us formulate a corollary from Theorem \ref{THM:main} showing the local loss of regularity for the Cauchy problem \eqref{CP}. We recall that any graded Lie group $G$ can be identified, for example through the exponential mapping, with the Euclidean space $\Rn$ where $n$ is the topological dimension of $G$. Then, if $\nu_1,\ldots,\nu_n$ are the dilation weights on $G$ as in \eqref{EQ:dils}, for any $s\in\R$ we have the local Sobolev embedding theorems \cite[Theorem 4.4.24]{FR2016}
\begin{equation}\label{EQ:Sobemb}
H^{s/\nu_1}_{loc}(\Rn)\subset H^s_{\mathcal R, loc}(G)\subset H^{s/\nu_n}_{loc}(\Rn),
\end{equation} 
where we assume $\nu_1$ and $\nu_n$ to be respectively the smallest and largest weights of the dilations, and the local Sobolev spaces are defined as follows
\begin{align*}
H^s_{\mathcal R, loc}(G):=\big\{f\in\mathcal{D}'(G)\,\big|\,\varphi f\in\ H^s_{\mathcal{R}}(G)\text{ for all }\varphi\in\mathcal{D}(G)\big\}.
\end{align*}
If $G$ is a stratified Lie group, we have $\nu_1=1$ and $\nu_n$ is the step of $G$, i.e., the number of steps in the stratification of its Lie algebra. In other words, if $G$ is a stratified Lie group of step $r$ and $H^s(G)$ is the Sobolev space defined using (any) sub-Laplacian on $G$, then the embeddings \eqref{EQ:Sobemb} are reduced to
\begin{equation}\label{EQ:Sobembs}
H^{s}_{loc}(\Rn)\subset H^s_{\mathcal R, loc}(G)\subset H^{s/r}_{loc}(\Rn).
\end{equation} 
These embeddings are sharp, as proved by Folland in \cite[Theorem 4.16]{F75}. Consequently, using the characterisation \eqref{EQ:Gevchar} of $\mathcal G^s_{\mathcal R}(G)$, we also obtain the embeddings 
\begin{equation}\label{EQ:Gevemb}
\mathcal G^{s \nu_1}(\Rn)\subset \mathcal G^s_{\mathcal R}(G)\subset \mathcal G^{s\nu_n}(\Rn),
\end{equation} 
where the space $\mathcal G^{\sigma}(\Rn)$ is the classical Euclidean Gevrey space, namely, the space of all smooth functions $f\in \mathcal C^\infty(\Rn)$ such that for every compact set $K\subset\Rn$ there exist two constants $B,C>0$ such that for every $\alpha$ we have
\begin{equation}\label{gevreyRn}
|\partial^\alpha f(x)|\leq C B^{|\alpha|}(\alpha!)^\sigma \quad\textrm{ for all } x\in K.
\end{equation}
Consequently, if $G$ is a stratified Lie group of step $r$ we have the embeddings
\begin{equation}\label{EQ:Gevembs}
\mathcal G^{s}(\Rn)\subset \mathcal G^s_{\mathcal R}(G)\subset \mathcal G^{s r}(\Rn).
\end{equation} 
Using the embeddings \eqref{EQ:Gevembs}, we obtain the following spatially local well-posedness result using the usual Euclidean Gevrey spaces. Here we may also assume that the Cauchy data are compactly supported due to the finite propagation speed of singularities. To emphasise the resulting phenomenon of local loss of Euclidean regularity we formulate it in the simplified setting of stratified Lie groups, with the topological identification $G\sim\Rn$.

\begin{corollary}\label{COR:main}
	Let $G\sim \Rn$ be a stratified Lie group of step $r$ and let $\mathcal R$ be a positive Rockland operator of homogeneous degree $\nu$ (for example, $\mathcal R$ can be a positive sub-Laplacian, for which $\nu=2$). 
	Assume that the Cauchy data $(u_0,u_1)$ are compactly supported.
	Then in the Cases 1-4 above the following hold:
	\begin{enumerate}[\text{Case} 1:]
		\item Given $s\in\R$, if $(u_0,u_1)$ are  in  $H^{s+\frac{\nu}{2}}(\Rn)\times H^s(\Rn)$, then there exists a unique solution of \eqref{CP} in $\mathcal C([0,\tau],H^{(s+\frac{\nu}{2})/r}(\Rn))\cap\mathcal C^1([0,\tau],H^{s/r}(\Rn))$, satisfying the following inequality for all values of $t\in [0,\tau]$:
		\begin{align*}
		\|u(t,\cdot)\|^2_{H^{(s+\frac{\nu}{2})/r}}+\|\partial_t u(t,\cdot)\|^2_{H^{s/r}}\leq C(\|u_0\|^2_{H^{s+\frac{\nu}{2}}}+\|u_1\|^2_{H^{s}});
		\end{align*}
		
		\item If $(u_0, u_1)$ are in $\mathcal G^s(\Rn)\times \mathcal G^s(\Rn)$, then there exists a unique solution of \eqref{CP} in $\mathcal C^2([0,\tau],\mathcal G^{sr}(\Rn))$, provided that
		\[
		1< s <1 +\frac{\alpha}{1-\alpha};
		\]
		
		\item If $(u_0, u_1)$ are in $\mathcal G^s(\Rn)\times \mathcal G^s(\Rn)$, then there exists a unique solution of \eqref{CP} in $\mathcal C^2([0,\tau],\mathcal G^{sr}(\Rn))$, provided that
		\[
		1< s <1 +\frac{l}{2};
		\]
		
		\item If $(u_0, u_1)$ are in $\mathcal G^s(\Rn)\times \mathcal G^s(\Rn)$ then there exists a unique solution of \eqref{CP} in $\mathcal C^2([0,\tau],\mathcal G^{sr}(\Rn))$, provided that
		\[
		1< s <1 +\frac{\alpha}{2}. 
		\]
	\end{enumerate}
\end{corollary}
The statements in Cases 2-4 for $s=1$ are simply alternative proofs of the results of Bony and Shapira in \cite{Bony-Schapira:analytic-IM-1972}.
For $G=\Rn$ and $\mathcal R$ being the Laplacian, we have $r=1$ and there is no loss of regularity in any of the Cases 1-4
\cite{CDS79,CK2002,GR2012,GR2013,KS2006}.
However on the Heisenberg group with step $r=2$ we already observe the local loss of regularity in Euclidean Sobolev and Gevrey spaces in all statements of Cases 1-4 in Corollary \ref{COR:main}.
We also note that using the Sobolev and Gevrey embeddings \eqref{EQ:Sobemb} and \eqref{EQ:Gevemb}, it is easy to formulate an extension of Corollary \ref{COR:main} to general graded Lie groups.

\section{Parameter dependent energy estimates}
\label{SEC:ODE}

In this Section we prove certain energy estimates for second order ordinary differential equations with explicit dependence on parameters. This will be crucial in the proof of Theorem \ref{THM:main} where the parameters will correspond to the following spectral decomposition of the infinitesimal representations of Rockland operators\footnote{See Section \ref{Sec:Operators}, Formula \eqref{eq:symbRock}.}:
\begin{align}\label{EQ:R-spec}
\pi(\mathcal R)=\begin{pmatrix}
\pi^2_1 & 0      & \dots    & \dots \\
0      & \pi^2_2 & 0    & \dots \\
\vdots&   0       & \ddots  &          \\
\vdots& \vdots  &             & \ddots
\end{pmatrix},
\end{align}
where $\pi_j$ are strictly positive real numbers and $\pi\in\widehat G\setminus\{1\}$.\

Results of the following type have been of use in different estimates related to weakly hyperbolic partial differential equations, such as \cite{CK2002} and \cite{GR2015}. However, in the aforementioned papers the conclusions rely on more general results, see also \cite{GR2012}. We partly follow the argument in \cite{GR2015} based on a standard reduction to a first order system. Consequently, we carry out different types of arguments depending on the specific assumptions in each of the cases. This allows us to formulate the precise dependence on parameters for ordinary differential equations corresponding to the propagation coefficient $a(t)$ as in Cases 1-4 of Theorem \ref{THM:main}, to which we refer in the following statement.
\begin{proposition}\label{lemma}
	Let $\beta>0$ be a positive constant and let $a(t)$ be a function that behaves according to Cases $1,2,3$ and $4$. Let $T>0$. Consider the Cauchy problem
	\begin{align}\label{ODE}
	\begin{cases}
	v''(t)+\beta^2 a(t)v(t)=0\quad\text{with }t\in(0,\tau],\\
	v(0)=v_0\in\C,\\
	v'(0)=v_1\in\C.
	\end{cases}
	\end{align}  
	Then the following holds: 
	\begin{enumerate}[\text{Case} 1:]
		\item There exists a constant $C>0$ such that for all $t\in [0,\tau]$ we have
		\[
		\beta^2|v(t)|^2+|v'(t)|^2\leq C(\beta^2|v_0|^2+|v_1|^2).
		\]
		\item There exist two positive constants $C,K>0$ such that for all $t\in [0,\tau]$ we have
		\begin{equation}\label{EQ:case2}
		\beta^2|v(t)|^2+|v'(t)|^2\leq Ce^{Kt \beta^{\frac{1}{s}}}(\beta^2|v_0|^2+|v_1|^2),
		\end{equation} 
		for any
		$1\leq s < 1+\frac{\alpha}{1-\alpha}$. Moreover, there exists a constant $k>0$ such that
		for any $\beta_0\geq 1$ the estimate \eqref{EQ:case2} holds for $K= k\beta_0^{1-\alpha-\frac1s}$ for all $\beta\geq \beta_0$.
		\item There exist two positive constants $C,K>0$ such that for all $t\in [0,\tau]$ we have
		\[
		\beta^2|v(t)|^2+|v'(t)|^2\leq C(1+\beta^\frac{l}{\sigma})e^{K\beta^\frac{1}{\sigma}}\big(\beta^2|v_0|^2+|v_1|^2\big),
		\]
		with $\sigma= 1+\frac{l}{2}$. 
		\item There exist two positive constants $C,K>0$ such that 
		\begin{equation}\label{EQ:case4}
		\beta^2|v(t)|^2+|v'(t)|^2\leq C(1+\beta^\frac{\alpha}{\alpha+1})e^{Kt\beta^\frac{1}{s}}(\beta^2|v_0|^2+|v_1|^2),
		\end{equation} 
		for any $1\leq s<1+\frac{\alpha}{2}$.
		Moreover, there exists a constant $k>0$ such that
		for any $\beta_0\geq 1$ the estimate \eqref{EQ:case4} holds for $K= k\beta_0^{\frac{1}{1+\alpha}-\frac1s}$ for all $\beta\geq \beta_0$.
	\end{enumerate}
	The constants $C$ in the above inequalities may depend on $a$ and $T$ but not on $\beta, v_0$ or $v_1$.
\end{proposition}

\begin{remark}
	This proposition and its proof are a non-explicitly stated part of \cite{KS2006}, in which a more general situation is treated. We adapt the reasoning and clarify it where necessary.
\end{remark}

\begin{proof}
	First we reduce the problem \eqref{ODE} to a first order system. In order to do this we rewrite it in a standard way as a matrix-valued equation. Thus we define the column vectors
	\[
	V(t):=\begin{pmatrix}i\beta v(t)\\ \partial_t v(t)\end{pmatrix},\quad V_0:=\begin{pmatrix}i\beta v_0\\ v_1\end{pmatrix},
	\]
	and the matrix 
	\begin{align*}
	A(t):=\begin{pmatrix}
	0 & 1\\
	a(t) & 0
	\end{pmatrix},
	\end{align*}
	that allow us to reformulate the second order \eqref{ODE} as the first order system
	\begin{align}\label{first}
	\begin{cases}
	V_t(t)=i\beta A(t)V(t),\\
	V(0)=V_0.
	\end{cases}
	\end{align}
	Now we recall that the \textit{Picard--Lindel\"of Theorem} guarantees the existence and uniqueness of solutions to the first order initial value problem \eqref{first}. Hence, let $V(t)$ be the unique solution of \eqref{first}. We will now treat each case separately, noting that we set $(\cdot,\cdot)$ and $|\cdot|$ respectively to denote the inner product and resulting norm on $\mathbb{C}^2$.\\
	
	\textbf{Case 1:} $\bm{a\in\mathbf{Lip}([0,\tau])}$\textbf{,} $\bm{a(t)\geq a_0>0}$\textbf{.}
	
	This is the simplest case that can be treated by a classical argument.
	We observe that the eigenvalues of our matrix $A(t)$ are given by $\pm \sqrt{a(t)}$. The symmetriser $S$ of $A$, i.e., the matrix such that 
	\[
	SA-A^*S=0,
	\]
	is given by
	\[
	S(t)=\begin{pmatrix}2 a(t) & 0\\ 0 & 2\end{pmatrix}.
	\]
	Thus we define the energy as 
	\[
	E(t):=\Big(S(t)V(t),V(t)\Big),
	\] 
	and we want to estimate its variations in time. A straightforward calculation yields the inequality
	\begin{equation}\label{en1}
	2 |V(t)|^2 \min_{t\in[0,\tau]} \{a(t),1\}\leq E(t)\leq 2|V(t)|^2\max_{t\in[0,\tau]}\{a(t),1\}. 
	\end{equation}
	In particular, in this Case the continuity and strict positivity of $a(t)$ ensure the existence of the strictly positive constants
	\[
	a_0=\min_{t\in[0,\tau]} a(t)\quad\text{and}\quad a_1=\max_{t\in[0,\tau]}a(t).
	\]
	Thus setting $c_0:=2 \min \{a_0,1\}$ and $C_0:=2\max\{a_1,1\}$, the inequality \eqref{en1} becomes
	\begin{align}\label{en2}
	c_0|V(t)|^2\leq E(t)\leq C_0 |V(t)|^2.
	\end{align}
	Using the standard notation that a subscript $t$ denotes the first derivative with respect to $t$, a straightforward calculation alongside \eqref{en2} give the estimate
	\begin{align}\notag
	E_t(t)&=\big( S_t(t)V(t),V(t)\big)+\big(S(t)V_t(t),V(t)\big)+\big(S(t)V(t),V_t(t)\big)=\\\notag
	&=\big( S_t(t)V(t),V(t)\big) +i\beta\big( S(t)A(t)V(t),V(t)\big)-i\beta\big( S(t)V(t),A(t)V(t)\big)=\\\notag
	&=\big( S_t(t)V(t),V(t)\big)+i\beta\Big( \big(S(t)A(t)-A^*(t)S(t)\big)V(t),V(t)\Big)=\\
	&=\big( S_t(t)V(t),V(t)\big)\leq \|S_t(t)\| |V(t)|^2.\label{der}
	\end{align}
	In this proof $\|\cdot\|$ denotes the matrix norm. Thus setting $c':=c_0^{-1}\sup_{t\in [0,\tau]}\|S_t(t)\|$, we find from \eqref{der} and \eqref{en2} that
	\begin{align}\label{gron}
	E_t(t)\leq c' E(t).
	\end{align}
	Applying Gr\"onwall's lemma to \eqref{gron}, we deduce that there exists a constant $c>0$ independent of $t\in[0,\tau]$ such that 
	\begin{align}\label{gron2}
	E(t)\leq c E(0).
	\end{align}
	Therefore, putting together \eqref{gron2} and \eqref{en2} we obtain
	\begin{align}\notag
	c_0 |V(t)|^2\leq E(t)\leq c E(0)\leq c C_0 |V(0)|^2.
	\end{align} 
	We can then rephrase this, asserting that there exists a constant $C>0$ independent of $t$ such that $|V(t)|^2\leq C |V(0)|^2$. Recalling the definition of $V(t)$ we obtain
	\[
	\beta^2|v(t)|^2+|\partial_t v(t)|^2\leq C\big(\beta^2 |v_0|^2 + |v_1|^2\big),
	\]
	as required.\\
	
	\textbf{Case 2:} $\bm{a\in\mathcal C^\alpha([0,\tau])}$\textbf{, with}  $\bm{ 0<\alpha<1,\,\,a(t)\geq a_0>0}$\textbf{.}
	
	Here we follow the method developed by Colombini and Kinoshita \cite{CK2002} for $n=1$ and subsequently extended \cite{GR2012} for any $n\in\mathbb N$. We look for solutions of the form 
	\begin{align}\label{sol}
	V(t)=e^{-\rho(t)\beta^{\frac{1}{s}}}(\det H(t))^{-1}H(t)W(t),
	\end{align}
	where
	\begin{itemize}
		\item $s\in\R$  depends on $\alpha$ in a manner determined below;
		\item the function $\rho=\rho(t)\in\mathcal C^1([0,\tau])$ is real-valued and will be chosen below;
		\item $W(t)$ is the energy vector, assumed to be $\mathcal{C}^1$;
		\item $H(t)$ is the matrix defined by
		\[
		H(t):=\begin{pmatrix}1 & 1\\ \lambda^\epsilon_1(t) & \lambda^\epsilon_2(t)\end{pmatrix},
		\]
		where for all $\epsilon>0$, $\lambda^\epsilon_1(t)$ and $\lambda^\epsilon_2(t)$ are regularisations of the eigenvalues of the matrix $A(t)$ of the form
		\begin{align*}
		&\lambda^\epsilon_{1}(t):=(-\sqrt{a}*\varphi_\epsilon)(t),\\\notag
		&\lambda^\epsilon_{2}(t):=(+\sqrt{a}*\varphi_\epsilon)(t),
		\end{align*}
		with $\{\varphi_\epsilon(t)\}_{\epsilon>0}$ being a family of cut-off functions defined, in terms of a non-negative function $\varphi\in\mathcal C^\infty_c(\R)$ with $\int_\R\varphi=1$, by setting $\varphi_\epsilon(t):=\frac{1}{\epsilon}\varphi\big(\frac{t}{\epsilon}\big)$. By construction, it follows that $\lambda^\epsilon_{1},\lambda^\epsilon_{2}\in\mathcal C^\infty([0,\tau])$.

	\end{itemize}
	We can easily verify, using the H\"older regularity of $a(t)$ of order $\alpha$ and, therefore, of $\sqrt{a(t)}$ of the same order $\alpha$, the  inequality
	\begin{equation}\label{strait}
	\det H(t)=\lambda^\epsilon_2(t)-\lambda^\epsilon_1(t)\geq 2\sqrt{a(0)}\geq 2\sqrt{a_0},
	\end{equation}
	Moreover for all $t\in[0,\tau]$ and $\epsilon>0$ there exist two constants $c_1,c_2>0$ such that
	\begin{align}
	&|\lambda^\epsilon_1(t)+\sqrt{a(t)}|\leq c_1\epsilon^\alpha,\\
	&|\lambda^\epsilon_2(t)-\sqrt{a(t)}|\leq c_2 \epsilon^\alpha,\label{difference}
	\end{align}
	uniformly in $t$ and $\epsilon$.
	
	Now we substitute our suggested solution \eqref{sol} in \eqref{first}, yielding
	\begin{align*}
	&-\rho'(t)\beta^\frac{1}{s}e^{-\rho(t)\beta^\frac{1}{s}}\frac{H(t)W(t)}{\det H(t)}+e^{-\rho(t)\beta^\frac{1}{s}}\frac{H_t(t)W(t)}{\det H(t)}+e^{-\rho(t)\beta^\frac{1}{s}}\frac{H(t)W_t(t)}{\det H(t)}+\\
	&-e^{-\rho(t)\beta^\frac{1}{s}}(\det H)_t(t)\frac{H(t)W(t)}{\big(\det H(t)\big)^2}=i\beta A(t)e^{-\rho(t)\beta^\frac{1}{s}}\frac{H(t)W(t)}{\det H(t)}.
	\end{align*}
	Multiplying both sides of this equality  by $e^{\rho(t)\beta^{\frac{1}{s}}}\det H(t)H^{-1}(t)$ we get
	\begin{align}\notag
	W_t(t)=\rho&'(t)\beta^{\frac{1}{s}}W(t)-H^{-1}(t)H_t(t)W(t)+(\det H)_t(t)\big(\det H(t)\big)^{-1}W(t)+\\
	&+i\beta H^{-1}(t)A(t)H(t)W(t).\label{W}
	\end{align}
	This leads to the estimate
	\begin{align*}
	\frac{d}{dt}|W(t)|^{2}&=\big(W_{t}(t),W(t)\big)+\big(W(t),W_t(t)\big)=2\Rep\big{(}W_t(t),W(t)\big{)}=\\
	&=2\Big(\rho'(t)\beta^{\frac{1}{s}}|W(t)|^2-\Rep\big(H^{-1}(t)H_t(t)W(t),W(t)\big)+\\
	&+\big(\det H(t)\big)^{-1}(\det H)_t(t)|W(t)|^2+\beta\Imp\big(H^{-1}(t)A(t)H(t)W(t),W(t)\big)\Big).
	\end{align*}
	We observe that
	\begin{align*}
	&2\Imp\big(H^{-1}AHW,W\big)=\big(H^{-1}AHW,W\big)-\overline{\big(H^{-1}AHW,W\big)}=\\
	&=\big(H^{-1}AHW,W\big)-\big(W,H^{-1}AHW\big)=\big(H^{-1}AHW,W\big)-\big((H^{-1}AH)^*W,W\big)=\\
	&=\Big(\big(H^{-1}AH-(H^{-1}AH)^*\big)W,W\Big)\leq\|H^{-1}AH-(H^{-1}AH)^*\|\|W\|^2.
	\end{align*}
	Thus we obtain
	\begin{align}\notag
	\frac{d}{dt}|W(t)|^2\leq \Big(2\rho'(t)\beta^{\frac{1}{s}}+2\|H^{-1}(t)H_t(t)\|+2\big|\big(\det H(t)\big)^{-1}(\det H)_t(t)\big|+\\\label{enW}
	+\beta \|[H^{-1}AH-(H^{-1}AH)^*](t)\| \Big)|W(t)|^2.
	\end{align}
	To proceed we need to estimate the following quantities:
	\begin{enumerate}[I)]
		\item $\|H^{-1}(t)H_t(t)\|$;
		\item $\big|\big(\det H(t)\big)^{-1}(\det H)_t(t)\big|$;
		\item $\|[H^{-1}AH-(H^{-1}AH)^*](t)\|$.
	\end{enumerate}
	
	In \cite{GR2012} and \cite{CK2002}, the authors determine estimates for similar functions in a more general setting, i.e. starting from an equation of arbitrary order $m$. In this particular case, we can proceed by straightforward calculations without relying on the aforementioned works. 
	
	We deal with these three terms as follows:
	\begin{enumerate}[I)]
		\item Since $H^{-1}=\frac{1}{\lambda_2^\epsilon-\lambda_1^\epsilon}\begin{pmatrix}\lambda_2^\epsilon &-1\\-\lambda_1^\epsilon & 1
		\end{pmatrix}$ and $H_t=\begin{pmatrix}0 & 0\\\partial_t\lambda_1^\epsilon & \partial_t\lambda_2^\epsilon
		\end{pmatrix}$, it follows that the entries of the matrix $H^{-1}H_t$ are given by the functions $\frac{\partial_t \lambda_j^\epsilon}{\lambda_2^\epsilon-\lambda_1^\epsilon}$. We have, for example,
		\begin{align}\notag
		|\partial_t\lambda_2^\epsilon(t)|&=\left|\sqrt a * \partial_t \varphi_\epsilon(t)\right|=\left|\frac{1}{\epsilon^2}\sqrt a *\varphi'\Big(\frac{t}{\epsilon}\Big)\right|=\left|\frac{1}{\epsilon}\int\sqrt{a(t-\rho\epsilon)}\varphi'(\rho)d\rho\right|=\\\label{derL}
		&=\left|\frac{1}{\epsilon}\int\big(\sqrt{a(t-\rho\epsilon)}-\sqrt{a(t)}\big)\varphi'(\rho)d\rho+\frac{1}{\epsilon}\sqrt{a(t)}\int\varphi'(\rho)d\rho\right|\leq k\epsilon^{\alpha -1},
		\end{align}
		where we are using the H\"older continuity of $\sqrt a$ for the first term and the fact that the second term is zero, since $\int \varphi'=0 $.
		Combining the inequalities \eqref{derL} and its equivalent for $\lambda_1^\epsilon$ with \eqref{strait}, we obtain for a suitable positive constant $k_1$ independent of $t$ that
		\[
		\|H^{-1}(t)H_t(t)\|\leq k_1\epsilon^{\alpha-1}.
		\]
		
		\item First we can estimate
		\begin{align*}
		\big|\big(\det H(t)\big)^{-1}(\det H)_t(t)\big|= \frac{\partial_t\lambda_2^\epsilon-\partial_t\lambda_1^\epsilon}{\lambda_2^\epsilon-\lambda_1^\epsilon}=\frac{2\partial_t\lambda_2^\epsilon}{\lambda_2^\epsilon-\lambda_1^\epsilon}\leq\frac{2k\epsilon^{\alpha-1}}{2\sqrt{a_0}},
		\end{align*}
		therefore,
		\[
		\big|\big(\det H(t)\big)^{-1}(\det H)_t(t)\big|\leq k_2 \epsilon^{\alpha-1},
		\]
		for a $t$-independent constant $k_2>0$.
		\item The matrix in which we are interested is
		\begin{equation}\notag
		H^{-1}AH-\big(H^{-1}AH\big)^*=\begin{pmatrix}0&\frac{-2a+(\lambda_1^\epsilon)^2+(\lambda_2^\epsilon)^2}{\lambda_1^\epsilon-\lambda_2^\epsilon}\\ \frac{2a-\big((\lambda_1^\epsilon)^2+(\lambda_2^\epsilon)^2\big)}{\lambda_1^\epsilon-\lambda_2^\epsilon}&0\end{pmatrix}.
		\end{equation}
		Observing that $(\lambda_1^\epsilon)^2=(\lambda_2^\epsilon)^2$, in light of inequality \eqref{strait}, to obtain the desired norm estimate, it is enough to consider the function $|a-(\lambda_2^\epsilon)^2|$. A straightforward calculation, using inequality \eqref{difference}, shows that
		\begin{align*}
		|a(t)-\lambda_2^\epsilon(t)^2|&=|\big(\sqrt{a(t)}-\lambda_2^\epsilon(t)\big)\big(\sqrt{a(t)}+\lambda_2^\epsilon(t)\big)|\leq\\
		&\leq c_2\epsilon^\alpha\Big(\sqrt{a(t)}+\int\sqrt{a(t-s)}\varphi_\epsilon(s)ds\Big)=\\&=c_2\epsilon^\alpha\int \big(\sqrt{a(t)} +\sqrt{a(t-s\epsilon)}\big)\varphi(s)ds\leq 2c_2\|\sqrt a\|_{L^\infty} \epsilon^\alpha.
		\end{align*}
		It follows that there is a $k_3>0$ such that
		\[
		\|[H^{-1}AH-(H^{-1}AH)^*](t)\|\leq k_3 \epsilon^\alpha.
		\]
	\end{enumerate}
	
	Combining \eqref{enW} with estimates I)--III), we obtain an estimate for the derivative of the energy, that is
	\begin{align}\label{enW2}
	\frac{d}{dt}|W(t)|^2\leq \Big(2\rho'(t)\beta^{\frac{1}{s}}+2k_1\epsilon^{\alpha-1}+2k_2 \epsilon^{\alpha-1}+ k_3 \beta \epsilon^\alpha \Big)|W(t)|^2.
	\end{align}
	Now we choose $\epsilon=\frac{1}{\beta}$, observing that we can always consider $\beta$ large enough, say $\beta>1$, in order to have $\epsilon\in(0,1]$. 
	Indeed, for $\beta\leq\beta_0$ for some fixed $\beta_0>0$, a modification of the argument below gives estimate \eqref{EQ:case2} with constants depending only on $\beta_0$ and $T$. So we may assume that $\beta>\beta_0$ for $\beta_0$ to be specified.
	We define also $\rho(t):=\rho(0)-Kt$ for some $K>0$ to be specified. 
	Substituting this in \eqref{enW2} we obtain for a suitable constant $k>0$ that 
	\begin{align*}
	\frac{d}{dt}|W(t)|^2\leq \Big(2\rho'(t)\beta^{\frac{1}{s}}+2k\beta^{1-\alpha}\Big)|W(t)|^2= \big(-2K+2k\beta^{1-\alpha-\frac1s} \big)\beta^{\frac{1}{s}} |W(t)|^2.
	\end{align*}
	If we have
	\[
	\frac{1}{s}>1-\alpha \quad \iff\quad s<1+\frac{\alpha}{1-\alpha},
	\]
	and we set  
	\begin{equation}\label{EQ:Kchoice}
	K:= k\beta_0^{1-\alpha-\frac1s}\geq k\beta^{1-\alpha-\frac1s},
	\end{equation} 
	then for all $t\in[0,\tau]$ we have
	\begin{equation}\label{EQ:Wmonotonicity}
	0\geq\frac{d}{dt}|W(t)|^2=2|W(t)|\frac{d}{dt}|W(t)|.
	\end{equation}
	Supposing there is a $t_0$ for which $\frac{d}{dt}|W(t_0)|>0$ we deduce from $W\in\mathcal{C}^1$ that there is a $\delta$-neighbourhood of $t_0$ in which $\frac{d}{dt}|W(t)|>0$. For every $t$ in this neighbourhood, we deduce from \eqref{EQ:Wmonotonicity} that $0\geq|W(t)|=0$, which contradicts $\frac{d}{dt}|W(t)|>0$. Hence $\frac{d}{dt}|W(t)|\leq0$ everywhere. This monotonicity of $|W|$ yields a bound on the solution vector $V(t)$:
	\begin{align}\notag
	|V(t)|&=e^{-\rho(t)\beta^{\frac{1}{s}}}\big(\det H(t)\big)^{-1}\|H(t)\||W(t)|\leq\\\notag
	&\leq e^{-\rho(t)\beta^{\frac{1}{s}}}\big(\det H(t)\big)^{-1}\|H(t)\||W(0)|=\\
	&=e^{K t\beta^{\frac{1}{s}}}\frac{\det H(0)}{\det H(t)}\frac{\|H(t)\|}{\|H(0)\|}|V(0)|.\label{V}
	\end{align}
	Note that, according to property \eqref{strait}, the function $\big(\det H(t)\big)^{-1}$ is bounded. Furthermore, the  behaviour of the convolution and the definition of $H(t)$ guarantee that both $\|H(t)\|$ and $\|H^{-1}(0)\|$ are bounded. Therefore, there exists a constant $C>0$ such that
	\begin{align*}
	|V(t)|\leq C e^{Kt\beta^{\frac{1}{s}}}|V(0)|,
	\end{align*}
	that means, by the definition of $V(t)$, that
	\begin{align*}
	\beta^2|v(t)|^2+|v_t(t)|^2\leq C e^{Kt\beta^{\frac{1}{s}}}\big(\beta^2|v_0|^2+|v_1|^2\big),
	\end{align*}
	proving the statement of Case 2. \\

	\textbf{Case 3:} $\bm{a\in\mathcal C^l([0,\tau])}$\textbf{, with} $\bm{l\geq 2,\; a(t)\geq 0}$
	
	In this case we extend the technique developed for Case 1. First we perturb the symmetriser of the matrix $A(t)$, in that we consider the so-called \textit{quasi-symmetriser} of $A(t)$. This idea was introduced for such problems by D'Ancona and Spagnolo in \cite{DS1998}.
	
	Consider the quasi-symmetriser of $A(t)$, that is, a family of coercive, Hermitian matrices of the form
	\begin{align*}
	Q_\epsilon^{(2)}(t):=S(t)+2\epsilon^2\begin{pmatrix}1 & 0\\ 0 & 0\end{pmatrix}=\begin{pmatrix}2 a(t) & 0\\ 0 & 2\end{pmatrix}+2\epsilon^2\begin{pmatrix}1 & 0\\ 0 & 0\end{pmatrix},
	\end{align*}
	for all $\epsilon\in(0,1]$, and such that $\big(Q_\epsilon^{(2)}A-A^{*} Q_\epsilon^{(2)}\big)$ goes to zero as $\epsilon$ goes to zero.\footnote{The ``$(2)$'' in the notation $Q^{(2)}_\epsilon$ refers to the order of the differential equation underlying the problem, and is standard \cite{DS1998,KS2006,GR2015}.} The associated perturbed energy is given by 
	\begin{align*}
	E_\epsilon(t):=\Big(Q_\epsilon^{(2)}V(t),V(t)\Big).
	\end{align*}
	We proceed estimating the energy. Firstly we calculate its derivative in time:
	\begin{align}\notag
	&\frac{d}{dt}E_\epsilon(t)= \Big(\frac{d}{dt}Q_\epsilon^{(2)}(t)V(t),V(t)\Big)+(Q_\epsilon^{(2)}(t)V_t(t),V(t))+(Q_\epsilon^{(2)}(t)V(t),V_t(t))=\\ \label{energyest}
	&=\Big(\frac{d}{dt}Q_\epsilon^{(2)}(t)V(t),V(t)\Big)+i\beta\Big(\big(Q_\epsilon^{(2)}A-A^*Q_\epsilon^{(2)}\big)(t)V(t),V(t)\Big).
	\end{align}
	To estimate the second term in the right hand side, we set 
	\[
	V=\begin{pmatrix}
	i\beta v\\
	\partial_t v
	\end{pmatrix}=:\begin{pmatrix}
	V_1\\
	V_2
	\end{pmatrix}.
	\]
	Algebraic calculations give
	\begin{align}\notag
	Q_\epsilon^{(2)}(t)A(t)-A^*(t)Q_\epsilon^{(2)}(t)=2\epsilon^2 
	\begin{pmatrix}
	0 & 1\\
	-1 & 0
	\end{pmatrix},
	\end{align}
	whence
	\begin{align}\notag
	\Big|\Big(\big(Q_\epsilon^{(2)}A&-A^*Q_\epsilon^{(2)}\big)(t)V(t),V(t)\Big)\Big|\leq 2\epsilon^2 2|\Imp(V_2\overline{V_1})| \leq\\\notag
	&\leq 2\epsilon  2|\epsilon V_1||V_2|\leq 2\epsilon (\epsilon^2|V_1|^2+|V_2|^2)\leq\\\notag
	&\leq 2\epsilon\Big(\big(\epsilon^2+a(t)\big)|V_1|^2+|V_2|^2\Big)= 2\epsilon \big(Q_\epsilon^{(2)}(t)V(t),V(t)\big).
	\end{align}
	Using this estimate in \eqref{energyest}, we obtain
	\begin{align}\notag
	\frac{d}{dt}E_\epsilon(t)&= \Big(\frac{d}{dt}Q_\epsilon^{(2)}(t)V(t),V(t)\Big)+i\beta\Big(\big(Q_\epsilon^{(2)}A(t)-A^*(t)Q_\epsilon^{(2)}\big)V(t),V(t)\Big)\leq\\ \notag
	&\leq \Big(\frac{d}{dt}Q_\epsilon^{(2)}(t)V(t),V(t)\Big) +2\beta\epsilon E_\epsilon(t)=\\
	&= \Bigg[\frac{\Big(\frac{d}{dt}Q_\epsilon^{(2)}(t)V(t),V(t)\Big)}{\big(Q_\epsilon^{(2)}(t)V(t),V(t)\big)}+ 2 \beta \epsilon\Bigg]E_\epsilon(t).\label{eq:quasiEnergy}
	\end{align}
	
	Before applying Gr\"onwall's lemma, we first estimate the integral
	\begin{align}\label{lemma2}
	\int_0^T\frac{\big(\frac{d}{dt}Q_\epsilon^{(2)}(t)V(t),V(t)\big)}{\big(Q_\epsilon^{(2)}(t)V(t),V(t)\big)}dt.
	\end{align}
	Let us recall that from the definition of the quasi-symmetriser, it follows that
	\begin{align}\label{qs}
	\Big( Q_\epsilon^{(2)}(t)V(t),V(t)\Big)=2\big(\big(a(t)+\epsilon^2\big)\beta^2|v|^2+|\partial_t v|^2\big).
	\end{align} 
	Thus, setting $c_1:=\max\big(1, 2(\|a\|_{L^\infty}+\epsilon^2)\big)$, we obtain a bound from above for \eqref{qs}, that is
	\[
	\Big( Q_\epsilon^{(2)}(t)V(t),V(t)\Big)\leq c_1 |V(t)|^2.
	\]
	Observing that $\epsilon^2c_1^{-1}\leq 1$ and $\epsilon^2c_1^{-1}\leq c_1$ for small enough $\epsilon$, we can also deduce an inequality from below of the form
	\[
	\epsilon^2c_1^{-1}|V(t)|^2\leq \Big( Q_\epsilon^{(2)}(t)V(t),V(t)\Big).
	\]
	Hence, there exists a constant $c_1\geq 1$ such that for $t\in[0,\tau]$ we have
	\begin{align}\label{property1}
	c_1^{-1}\epsilon^2|V(t)|^2\leq \big(Q_\epsilon^{(2)}(t)V(t),V(t)\big)\leq c_1 |V(t)|^2.
	\end{align}
	The lower bound, together with \cite[Lemma 2]{GR2013} (see \cite[Lemma 2]{KS2006} for a detailed proof), allows us to estimate the integral \eqref{lemma2} as follows
	\begin{align}\notag
	&\int_0^T\frac{\big(\frac{d}{dt}Q_\epsilon^{(2)}(t)V(t),V(t)\big)}{\big(Q_\epsilon^{(2)}(t)V(t),V(t)\big)}dt =
	\int_0^T\frac{\big(\frac{d}{dt}Q_\epsilon^{(2)}(t)V(t),V(t)\big)}{\big(Q_\epsilon^{(2)}(t)V(t),V(t)\big)^{1-\frac{1}{l}}(Q_\epsilon^{(2)}(t)V(t),V(t)\big)^\frac{1}{l}}dt\leq\\ \notag
	&\leq c_1^\frac{1}{l}\epsilon^{-\frac{2}{l}} \int_0^T\frac{\big(\frac{d}{dt}Q_\epsilon^{(2)}(t)V(t),V(t)\big)}{\big(Q_\epsilon^{(2)}(t)V(t),V(t)\big)^{1-\frac{1}{l}}|V(t)|^\frac{2}{l}}dt \leq c_1^\frac{1}{l}\epsilon^{-\frac{2}{l}} c_T\|Q_\epsilon^{(2)}\|^{\frac{1}{l}}_{\mathcal C^l([0,\tau])}\leq c_3\epsilon^{-\frac{2}{l}}.
	\end{align}
	Thus, applying the Gr\"onwall lemma to \eqref{eq:quasiEnergy} and using the estimates for the quasi-symmetriser just derived, we obtain
	\[
	E_\epsilon(t)\leq E_\epsilon(0)\exp\left( c_3\epsilon^{-\frac{2}{l}}+2\beta\epsilon T\right).
	\]
	Combining the latter inequality with \eqref{property1} we obtain
	\[
	c_1^{-1}\epsilon^2 |V(t)|^2\leq E_\epsilon(t)\leq E_\epsilon(0)e^{c_T(\epsilon^{-\frac{2}{l}}+\beta\epsilon)} \leq c_1|V(0)|^2 e^{c_T(\epsilon^{-\frac{2}{l}}+\beta\epsilon)}.
	\]
	We choose $\epsilon$ such that $\epsilon^{-\frac{2}{l}}=\beta\epsilon$, thus $\epsilon=\beta^{-\frac{l}{2+l}}$ and $\epsilon\beta=\beta^{\frac{2}{2+l}}$. 
	As argued in Case 2, we can assume that $\beta$ is large enough.
	Setting $\sigma=1+\frac{l}{2}$, for a suitable constant $K\in\R$ it follows that
	\[
	|V(t)|^2\leq C \beta^{\frac{l}{\sigma}}e^{K\beta^{\frac{1}{\sigma}}}|V(0)|^2.
	\]
	This means that
	\[
	\beta^2|v(t)|^2+|v'(t)|^2\leq C \beta^\frac{l}{\sigma}e^{K\beta^{\frac{1}{\sigma}}}\big(\beta^2 |v_0|^2+|v_1|^2\big),
	\]
	as required.\\
	
	\textbf{Case 4:} $\bm{a\in\mathcal C^\alpha([0,\tau])}$\textbf{, with} $\bm{0<\alpha<2}$\textbf{,} $\mathbf{a(t)\geq0}$.
	
	In this last case we extend the proof of Case 2. However, under these assumptions the eigenvalues $\pm\sqrt{a(t)}$ of the matrix $A(t)$ might coincide, and hence they are H\"older-regular of order $\frac{\alpha}{2}$ instead of $\alpha$. In order to adapt the proof of Case 2 to this situation we will assume without loss of generality that $a\in\mathcal C^{2\alpha}([0,\tau])$ with $0<\alpha<1$, so that  $\sqrt{a}\in\mathcal C^\alpha([0,\tau])$. At the end, to formulate the final result, we will change $\alpha$ into $\frac\alpha{2}.$
	
	Following the argument developed for Case 2, we look again for solutions of the form
	\begin{align*}
	V(t)=e^{-\rho(t)\beta^{\frac{1}{s}}}\big(\det  H(t)\big)^{-1}H(t)W(t),
	\end{align*}
	with the real-valued function $\rho(t)$, the exponent $s$ and the energy $W(t)$ to be chosen later, while $H(t)$ is the matrix given by
	\begin{align*} 
	H(t)=\begin{pmatrix}
	1 & 1\\
	\lambda^\epsilon_{1,\alpha}(t) & \lambda^\epsilon_{2,\alpha}(t)
	\end{pmatrix},
	\end{align*}
	where the regularised eigenvalues $\lambda^\epsilon_{1,\alpha}(t)$ and $\lambda^\epsilon_{2,\alpha}(t)$ of $A(t)$ differ from those defined in the Case 2 in the following way:
	\begin{align*}
	&\lambda^\epsilon_{1,\alpha}(t):=(-\sqrt{a}*\varphi_\epsilon)(t)+\epsilon^\alpha,\\
	&\lambda^\epsilon_{2,\alpha}(t):=(+\sqrt{a}*\varphi_\epsilon)(t)+2\epsilon^\alpha.
	\end{align*}
	
	Arguing as in Case 2, we can easily see that the smooth functions $\lambda^\epsilon_{1,\alpha}(t)$ and $\lambda^\epsilon_{2,\alpha}(t)$ satisfy uniformly in $t$ and $\epsilon$ the following inequalities
	\begin{itemize}
		\item $\det  H(t)=\lambda^\epsilon_{2,\alpha}(t)-\lambda^\epsilon_{1,\alpha}(t)\geq c_1\epsilon^\alpha$;
		\item $|\lambda^\epsilon_{1,\alpha}(t)+\sqrt {a(t)}|\leq c_2 \epsilon^\alpha$;
		\item $|\lambda^\epsilon_{2,\alpha}(t)-\sqrt{a(t)}|\leq c_3 \epsilon^\alpha$.
	\end{itemize}
	We now look for the energy estimates. In order to do this, recalling the calculations \eqref{W} and \eqref{enW} derived before, we obtain
	\begin{align}\notag
	&\frac{d}{dt}\big|W(t)\big|^2=2\Rep\big(W_t(t),W(t)\big)\leq \Big(2\rho'(t)\beta^{\frac{1}{s}}+2\|H^{-1}(t)H_t(t)\|+\\&+2\big|\big(\det H(t)\big)^{-1}\det H_t(t)\big|+\beta \big\|\big(H^{-1}AH-(H^{-1}AH)^*\big)(t)\big\| \Big)|W(t)|^2.\label{W4}
	\end{align}
	Similar arguments to those in Case 2 allow us to obtain the bounds
	\begin{enumerate}[I)]
		\item $\|H^{-1}(t)H_t(t)\|\leq k_1\epsilon^{-1}$;
		\item $\big|\big(\det H(t)\big)^{-1}\det H_t(t)\big|\leq k_2 \epsilon^{-1}$;
		\item $\|H^{-1}AH-(H^{-1}AH)^*\|\leq k_3 \epsilon^\alpha$.
	\end{enumerate}
	Combining \eqref{W4} with I)--III) we obtain 
	\begin{align*}
	\frac{d}{dt}|W(t)|^2\leq \Big(2\rho'(t)\beta^{\frac{1}{s}}+2k_1\epsilon^{-1}+2k_2 \epsilon^{-1}+ k_3 \beta \epsilon^\alpha \Big)|W(t)|^2.
	\end{align*}
	We choose $\epsilon^{-1}=\beta\epsilon^{\alpha}$ which yields $\epsilon=\beta^{-\frac{1}{\alpha+1}}$. Thus, setting $\gamma:=\frac{1}{\alpha+1}$, we obtain for a constant $c>0$ the estimate
	\begin{align*}
	\frac{d}{dt}|W(t)|^2\leq \Big(2\rho'(t)\beta^{\frac{1}{s}}+2c \beta^\gamma \Big)|W(t)|^2.
	\end{align*}
	We take $\rho(t):=\rho(0)-Kt$ with $K>0$ to be chosen later. 
	Considering 
	\begin{align*}
	\frac{1}{s}>\gamma\quad\iff\quad s< 1 +\alpha,
	\end{align*}
	we have that
	\begin{align}\label{eq:mon}
	\frac{d}{dt}|W(t)|^2\leq (-2K+2c\beta^{\gamma-\frac1s}) \beta^{\frac{1}{s}} |W(t)|^2\leq 0,
	\end{align}
	provided that $\beta$ is large enough. Similarly to Case 2, the monotonicity \eqref{eq:mon} of $|W|$ yields the bound
	\begin{align}\notag
	|V(t)|&=\,e^{-\rho(t)\beta^{\frac{1}{s}}}\big(\det  H(t)\big)^{-1}\|H(t)\||W(t)|=\\\notag
	&=\,e^{-\rho(t)\beta^{\frac{1}{s}}}\big(\det  H(t)\big)^{-1}\|H(t)\||W(0)|\leq\\ 
	&\leq e^{Kt\beta^{\frac{1}{s}}}\big(\det H(t)\big)^{-1}\det H(0)\|H(t)\|\big(\|H(0)\|\big)^{-1}|V(0)|.\label{V2}
	\end{align}
	Since 
	\[
	\big(\det H(t)\big)^{-1}\det H(0)\|H(t)\|\big(\|H(0)\|\big)^{-1}\leq c\epsilon^{-\alpha}=c\beta^{\frac{\alpha}{\alpha+1}},
	\]
	the inequality \eqref{V2} becomes 
	\begin{align*}
	|V(t)|\leq c \beta^{\frac{\alpha}{\alpha+1}}e^{Kt\beta^{\frac{1}{s}}}|V(0)|,
	\end{align*}
	which means 
	\begin{align*}
	\beta^2|v(t)|^2+|v'(t)|^2\leq c\beta^{\frac{\alpha}{\alpha+1}}e^{Kt\beta^{\frac{1}{s}}}\big(\beta^2|v_0|^2+|v_1|^2\big).
	\end{align*}
	Combining this with a remark for small $\beta$ similar to Case 2 yields the result.
	Thus Proposition \ref{lemma} is proved.
\end{proof}

\section{Proof of Theorem \ref{THM:main}}
\label{SEC:proof}

In this Section we present a proof of Theorem \ref{THM:main}. In order to do this, we combine the results for second order ordinary differential equations obtained in Section \ref{SEC:ODE} with the theory for graded Lie groups discussed in Chapter \ref{CP:2}, specifically Sections \ref{Sec:groupFouriertransform} and \ref{Sec:gradedLiegroup}.

We briefly recall the crucial aspects related to the Fourier analysis on groups that we are going to use in the proof. The main ingredient that we need is the Fourier transform on groups, introduced in Section \ref{Sec:groupFouriertransform}.

From now on we consider a graded Lie group $G$. Let $f\in L^1(G)$ and $\pi\in\widehat G$. Then, the group Fourier transform of $f$ at $\pi$ is defined by\footnote{See Definition \ref{def:groupFouriertransform}.}
\[
\mathcal F_G f(\pi)\equiv\widehat f(\pi)\equiv\pi(f):=\int_G f(x)\pi(x)^*dx,
\] 
where we integrate against the bi-invariant Haar measure on $G$.
Hence, the Fourier transform produces a linear mapping $\widehat f(\pi):\mathcal H_\pi\rightarrow\mathcal H_\pi$ that can be represented by an infinite matrix  once we choose a basis for the Hilbert space $\mathcal H_\pi$. By Kirillov's orbit method (see e.g. \cite{CG90}), one can explicitly construct the Plancherel measure $\mu$ on the dual $\widehat G$. Therefore we can use the Fourier inversion formula. In addition,  the operator $\pi(f)=\widehat f (\pi)$ is Hilbert--Schmidt, meaning
\begin{align*}
\|\pi(f)\|^2_{\HS}={\rm Tr}\big(\pi(f)\pi(f)^*\big)<\infty,
\end{align*}
and the function $\widehat G\ni\pi \mapsto \|\pi(f)\|^2_{\HS}$ is integrable with respect to $\mu$. These considerations lead us to an equivalent of the Euclidean Plancherel formula, discussed in Theorem \ref{Thm:Plancherelformula} and re-formulated as
\begin{align}\label{plancherel}
\int_G|f(x)|^2dx=\int_{\widehat G}\|\pi(f)\|^2_{\HS}d\mu(\pi).
\end{align}
Furthermore, given a positive Rockland operator $\mathcal{R}$, its functional calculus\footnote{See Theorem \ref{RockFuncCalc}.} yields the identity
\begin{align*}
\mathcal F_G \big(\mathcal R f\big)(\pi)=\pi({\mathcal R})\widehat f(\pi),
\end{align*}
for every representation $\pi\in\widehat{G}$. Hulanicki, Jenkins and Ludwig showed in \cite{HJL1985} that the spectrum of any infinitesimal representation $\pi(\mathcal R)$ of a Rockland operator $\mathcal{R}$ is \textit{discrete}. Therefore for each representation $\pi$ there is an orthonormal basis for the Hilbert space associated with $\pi$ in which the operator $\pi(\mathcal{R})$ is represented by the diagonal matrix
\begin{align}\label{eq:symbRock1}
\pi(\mathcal R)=\begin{pmatrix}
\pi^2_1 & 0      & \dots    & \dots \\
0      & \pi^2_2 & 0    & \dots \\
\vdots&   0       & \ddots  &          \\
\vdots& \vdots  &             & \ddots
\end{pmatrix},
\end{align}
where the $\pi_j$'s are strictly positive real numbers.

We can now proceed with the demonstration of Theorem \ref{THM:main}.
\begin{proof}[Proof of Theorem \ref{THM:main}]
	Our aim is to reduce the Cauchy problem \eqref{CP} 
	\begin{align}\label{CP1}
	\left\{\begin{array}{rclc}
	\partial_t^2 u(t,x)+a(t)\mathcal R u(t,x)&=&0,&(t,x)\in[0,\tau]\times G,\\
	u(0,x)&=&u_0(x),&x\in G,\\
	\partial_t u(0,x)&=&u_1(x),&x\in G,
	\end{array}\right.
	\end{align}
	to a form in which we can apply Proposition \ref{lemma}. To do so, we take the group Fourier transform of the Rockland wave equation and of the initial data in \eqref{CP1} with respect to $x\in G$. In this way, the equation \eqref{CP1} becomes
	\begin{align}\label{fourier}
	\partial_t^2 \widehat u (t,\pi)+a(t)\pi(\mathcal R)\widehat u (t,\pi)=0.
	\end{align}
	According to Formula \eqref{eq:symbRock1}, the infinitesimal representations of a Rockland operator $\mathcal R$ can be seen as diagonal matrices. This allows us to interpret equation \eqref{fourier} component-wise as an infinite system of equations of the form
	\begin{align}\label{single}
	\partial_t^2 \widehat u (t,\pi)_{m,k}+a(t)\pi_m^2\widehat u (t,\pi)_{m,k}=0,
	\end{align}
	for any representation $\pi\in\widehat G$, and $m,k\in\N$. The key point of the following argument is to decouple the system generated by the matrix equation \eqref{fourier}. In order to do this, we fix an arbitrary representation $\pi$ and a general entry $(m,k)\in\N\times\N$. Then, we treat each equation given by \eqref{single} individually. Formally, recalling the notation used in Proposition \ref{lemma} and noting that $\pi$, $m$ and $k$ are fixed, we define
	\[
	v(t):=\widehat u(t,\pi)_{m,k},\quad \beta^2:=\pi_m^2,
	\]
	and
	\[
	v_0:=\widehat u_0 (\pi)_{m,k},\quad v_1:=\widehat u_1(\pi)_{m,k}.
	\]
	Then \eqref{single} becomes
	\[
	v''(t)+a(t)\beta^2 v(t)=0.
	\]
	We proceed discussing implications of Proposition \ref{lemma} separately in each case.
	
	\medskip
	\textbf{Case 1:} $\bm{a\in\mathbf{Lip}([0,\tau])}$\textbf{,} $\bm{a(t)\geq a_0>0}$\textbf{.}
	
	Applying Proposition \ref{lemma}, we obtain that there exists a positive constant $C>0$  such that 
	\[
	\beta^2|v(t)|^2+|v'(t)|^2\leq C(\beta^2|v_0|^2+|v_1|^2),
	\]
	which is equivalent to
	\begin{align}\label{single inequality}
	|\pi_m\widehat u(t,\pi)_{m,k}|^2+|\widehat u'(t,\pi)_{m,k}|^2\leq C \big(|\pi_m\widehat u_0(\pi)_{m,k}|^2+|\widehat u_1(\pi)_{m,k}|^2\big).
	\end{align}
	Since in Proposition \ref{lemma} the constant $C$ is independent of $\beta,v_0$ and $v_1$, the estimate \eqref{single inequality} holds uniformly in $\pi\in\widehat G$ and $m, k\in\N$. We multiply inequality \eqref{single inequality} by $\pi_m^{4s/\nu}$ yielding
	\begin{align}
	|\pi_m^{1+\frac{2s}{\nu}}\widehat u(t,\pi)_{m,k}|^2&+|\pi_m^{\frac{2s}{\nu}}\widehat u'(t,\pi)_{m,k}|^2\leq\notag\\ 
	&\leq C \big(|\pi_m^{1+\frac{2s}{\nu}}\widehat u_0(\pi)_{m,k}|^2+|\pi_m^{\frac{2s}{\nu}}\widehat u_1(\pi)_{m,k}|^2\big).\label{general single inequality}
	\end{align}
	
	Thus, recalling that for any Hilbert--Schmidt operator $A$ we have 
	$$\|A\|^2_{\HS}=\sum_{m,k}|(A\varphi_m,\varphi_k)|^2$$ for any orthonormal basis $\{\varphi_1,\varphi_2,\dots\}$, we can consider the infinite sum over $m$ and $k$ of the inequalities provided by \eqref{general single inequality}, to obtain
	\begin{align}\label{HS inequality}
	\|\pi(\mathcal R)^{\frac{1}{2}+\frac{s}{\nu}}\widehat u (t,\pi)\|_{\HS}^2&+\|\pi(\mathcal R)^{\frac{s}{\nu}}\partial_t \widehat u(t,\pi)\|_{\HS}^2\leq\\\notag
	&\leq C\big(\|\pi(\mathcal R)^{\frac{1}{2}+\frac{s}{\nu}}\widehat u_0(\pi)\|_{\HS}^2+\|\pi(\mathcal R)^{\frac{s}{\nu}}\widehat u_1(\pi)\|_{\HS}^2\big).
	\end{align}
	We can now integrate both sides of \eqref{HS inequality} against the Plancherel measure $\mu$ on $\widehat G$, so that the Plancherel identity yields estimate \eqref{inequality case 1}.
	
	\medskip
	\textbf{Case 2:} $\bm{a\in\mathcal C^\alpha([0,\tau])}$\textbf{, with } $\bm{ 0<\alpha<1,\,\,a(t)\geq a_0>0}$\textbf{.}
	
	The application of Proposition \ref{lemma} implies the existence of two positive constants $C,K>0$ such that for all $m,k\in\N$ and for every representation $\pi\in\widehat G$ we have
	\begin{align}\label{lemma case 2}
	|\pi_m\widehat u (t,\pi)_{m,k}|^2+|\widehat u' (t,\pi)_{m,k}|^2\leq Ce^{Kt\pi_m^{\frac{1}{s}}}(|\pi_m\widehat u_0 (\pi)_{m,k}|^2+|\widehat u_1 (\pi)_{m,k}|^2),
	\end{align} 
	where 
	\[
	s < 1+\frac{\alpha}{1-\alpha}.
	\]
	If the Cauchy data $(u_0, u_1)$ are in $\mathcal G^s_\mathcal R(G)\times \mathcal G^s_\mathcal R(G)$, then there exist two positive constants $A_0$ and $A_1$ such that\footnote{See \eqref{G} for the definition of the spaces $\mathcal{G}^s_\mathcal{R}(G)$, of Gevrey type.}
	\begin{equation*}
	\|e^{A_0\mathcal R^{\frac{1}{2s}}}u_0\|_{L^2}<\infty\quad\text{and}\quad\|e^{A_1\mathcal R^{\frac{1}{2s}}}u_1\|_{L^2}<\infty.
	\end{equation*}
	
	Take now $A=\min\{A_0,A_1\}$. We can always assume $K$ in Case 2 of Proposition \ref{lemma} is small enough, so that we have some $B>0$ such that $KT=A-B$. Therefore, we can rewrite inequality \eqref{lemma case 2} as 
	\begin{align}\notag
	e^{B\pi_m^{\frac{1}{2s}}}\big(|\pi_m\widehat u (t,\pi)_{m,k}|^2&+|\widehat u' (t,\pi)_{m,k}|^2\big)\leq\\\label{lemma case 3}
	&\leq Ce^{A\pi_m^{\frac{1}{s}}}(|\pi_m\widehat u_0 (t)_{m,k}|^2+|\widehat u_1 (t)_{m,k}|^2).
	\end{align}
	Summing over $m,k$, integrating against the Plancherel measure of $\widehat G$ and applying the Plancherel identity, inequality \eqref{lemma case 3} becomes
	\begin{align}\notag
	\|e^{B\mathcal R^{\frac{1}{2s}}}u\|_{L^2(G)}&+\|e^{B\mathcal R^{\frac{1}{2s}}}\partial_tu\|_{L^2(G)}\leq
	\|e^{B\mathcal R^{\frac{1}{2s}}}\mathcal R^{\frac{1}{2}}u\|_{L^2(G)}+\|e^{B\mathcal R^{\frac{1}{2s}}}\partial_tu\|_{L^2(G)}\lesssim\\\label{case 2 conclusion}
	&\lesssim \|e^{A\mathcal R^{\frac{1}{2s}}}\mathcal R^{\frac{1}{2}}u_0\|_{L^2(G)}+\|e^{A\mathcal R^{\frac{1}{2s}}}u_1\|_{L^2(G)}.
	\end{align}
	From the definition of $\mathcal{G}_{\mathcal{R}}^s(G)$, it follows that if a function $f$ belongs to $\mathcal G^s_{\mathcal R}(G)$, then $\mathcal R^\frac{1}{2}f$ is also in $\mathcal G_\mathcal R^s(G)$. Therefore, from \eqref{case 2 conclusion} we obtain the desired well-posedness result.
	
	\medskip
	\textbf{Case 3:} $\bm{a\in\mathcal C^l([0,\tau])}$\textbf{, with} $\bm{l\geq 2,~a(t)\geq 0.}$
	
	Similarly to the previous cases, the application of Proposition \ref{lemma} yields the existence of two positive constants $C,K>0$ such that
	\begin{align*}
	|\pi_m \widehat{u}(t,\pi)_{m,k}(t)|^2&+|\widehat{u}'(t,\pi)_{m,k}|^2\leq\\
	&\leq C\pi_m^{\frac{l}{\sigma}+2}e^{KT\pi_m^\frac{1}{\sigma}}| \widehat{u}_0(\pi)_{m,k}|^2+C\pi_m^\frac{l}{\sigma}e^{KT\pi_m^\frac{1}{\sigma}}|\widehat{u}_1(\pi)_{m,k}|^2\leq\\\notag
	&\leq Ce^{K'\pi_m^\frac{1}{s}}| \widehat{u}_0(\pi)_{m.k}|^2+Ce^{K'\pi_m^\frac{1}{s}}|\widehat{u}_1(\pi)_{m,k}|^2,
	\end{align*}
	with $1\leq s<\sigma=1+\frac{l}{2}$, for some $K'>0$ small enough. Proceeding as in Case 2, we obtain the desired inequality.
	
	\medskip
	\textbf{Case 4: $\mathbf{a\in\mathcal C^\alpha([0,\tau])}$, with $\mathbf{0<\alpha<2}$, $\mathbf{a(t)\geq0}$.}
	
	In this last case, applying Proposition \ref{lemma} we have that there exist two positive constants $C,K>0$ such that 
	\[
	\pi_m^2|\widehat{u}(t.\pi)_{m.k}|^2+|\widehat{u}'(t,\pi)_{m,k}|^2\leq C\pi_m^\frac{1}{\alpha+1}e^{KT\pi_m^\frac{1}{s}}(\pi_m^2|\widehat{u}_0(\pi)_{m,k}|^2+|\widehat{u}_1(\pi)_{m,k}|^2),
	\]
	with $1\leq s< 1+\frac\alpha{2}$. Arguing as above, the result follows and the theorem is fully demonstrated.
\end{proof}


\chapter{Sub-Laplacian Gevrey Spaces}\label{CP:4}
In Chapter \ref{CP:3}, we have presented an example of a Cauchy problem where, under suitable conditions, the Gevrey regularity appears naturally. The Gevrey classes of functions are intermediate spaces between smooth functions and real-analytic functions, and for this reason they are particularly important in applications, e.g., in Gevrey micro-local analysis, Gevrey solvability, in the study of hyperbolic equations, dynamical systems and evolution partial differential equations. See Chapter \ref{CP:1} for references. 

In this Chapter, by modifying the definition of Gevrey spaces we aim to study the Gevrey hypoellipticity of sub-Laplacians. In order to do this we introduce a variation of these spaces.\footnote{It might be interesting to investigate their relation with the `original' spaces.}

Firstly, we recall the Euclidean Gevrey spaces and present the main ideas of the fundamental characterization of Gevrey functions in terms of their Fourier transform, referring to the monograph by Rodino \cite{R1993}. Then, we move to the non-commutative setting, mentioning the case of compact groups studied by Dasgupta and Ruzhansky in \cite{DR2014}. Subsequently, we consider the more general setting of manifolds and introduce our proposed definition of \textit{sub-Laplacian Gevrey spaces}. Then, we study these classes of functions, embedding them in bigger spaces explicitly depending on the sub-Laplacian. We notice that our new classes of functions match the Gevrey type functions we encountered in the study of the well-posedness of the wave equation in Chapter \ref{CP:3}. Finally, we consider the case of the special unitary group and the Heisenberg group, where the explicit symbolic calculus allows us to prove a complete characterization for our sub-Laplacian Gevrey functions.

\section{Euclidean Gevrey spaces}
In 1918 the French mathematician Maurice Gevrey introduced in \cite{G1918} the `\textit{fonctions de classe donn\'ee}', later called \textit{Gevrey functions} in his honour:
\begin{definition}[Gevrey functions of order $s$ in $\Omega$]
	Let $\Omega$ be an open subset of $\Rn$ and let $s\geq 1$. A function $f$ is a Gevrey function of order $s$, written $f\in G^s(\Omega)$, if $f\in\mathcal{C}^\infty(\Omega)$ and for all compact subsets $K$ of $\Omega$ there exist two positive constants $A$ and $C$ such that for all $\alpha\in\N^n_0$ and for all $x\in K$ we have
	\[
	|\partial^\alpha f(x)|\leq A C^{|\alpha|}(\alpha!)^s.
	\]
\end{definition}

It follows immediately from the definition that for $s=1$ the corresponding Gevrey class of functions coincides with the space of real analytic functions, while in general they provide an intermediate scale of spaces between smooth functions $\mathcal{C}^\infty$ and real-analytic functions. This means that Gevrey classes are widely relevant in the analysis of operators with  some properties failing in $\mathcal{C}^\infty$ or in the analytic frameworks.

A simple but meaningful example is the homogeneous equation associated to the heat operator $L=\partial_{t}-\sum_{j=1}^{n}\partial_{x_j}^2$ in $\Rn$ with $n\geq 1$. Indeed, the solutions of the homogeneous equation $Lu=0$ are not analytic in general, though always $\mathcal C^\infty$, and by calculating derivatives of the fundamental solution of $L$ we can deduce they are Gevrey for $s\geq 2$. This provides more precise information on the regularity of the solutions of the heat equation. An example in the other direction is that the Cauchy problem for the wave equation is analytically well-posed but not well posed in $\mathcal C^\infty$ in the presence of multiple characteristics. Consequently, determining the sharp Gevrey order for the well-posedness is a challenging problem with several results, starting with the seminal work of Colombini, de Giorgi and Spagnolo \cite{CDS79}, and continuing with many others as discussed in Chapter \ref{CP:3}.

The Gevrey spaces can be effectively characterised on the Fourier transform side, which increases their applicability in many problems, most notably allowing on to obtain energy estimates for evolution partial differential equations; see, e.g., Rodino's book \cite{R1993}. In particular, the following theorem holds:
\begin{theorem}\label{Thm:Rodino}~\newline
	\begin{enumerate}[i.]
		\item Let $s>1$. If $\phi\in G^s(\Rn)\cap\mathcal{C}^\infty_0(\Rn)$, then there exist positive constants $C$ and $\epsilon$ such
		that
		\begin{equation}\label{Rodino}
		|\widehat\phi(\xi)|\leq C e^{-\epsilon |\xi|^{\frac{1}{s}}},\quad\text{for all }\xi\in\Rn;
		\end{equation}
		\item Let $s\geq 1$. If the Fourier transform of $\phi\in \mathcal S'(\Rn)$ satisfies \eqref{Rodino}, then $\phi\in G^s(\Rn)$.
	\end{enumerate}
\end{theorem}
This characterisation plays a fundamental r\^ole in applications, see for example \cite{DFT2009}, and allows one to deduce an important link between the Gevrey functions and the Laplacian operator, as the function $ |\xi|^{\frac{1}{s}}=(|\xi|^2)^{\frac{1}{2s}}$ appearing in \eqref{Rodino} contains $|\xi|^2$, the Fourier transform of the Laplacian. A proof of Theorem \ref{Thm:Rodino} can be found in \cite{R1993}.

\subsection{Gevrey spaces on compact Lie groups}
Recently, Dasgupta and Ruzhansky extended the study of Gevrey regularity to the non-commutative compact case, generalising work by Seeley \cite{S1969} in which he characterized analytic functions on compact manifolds in terms of their eigenfunction expansions. In \cite{DR2014} the former two showed that the Gevrey spaces defined in local coordinates on a compact Lie group allow similar global descriptions in terms of the Laplacian on the group. Such a characterisation was used in \cite{GR2015} to find energy estimates for the corresponding wave equations for the Laplacian and establish a well-posedness result in Gevrey classes.

Subsequently, the characterisation from \cite{DR2014} on compact Lie groups was extended further in \cite{DR2016} to the setting of elliptic operators on general compact manifolds. More precisely, they showed that if $E$ is a positive elliptic pseudo-differential operator of order $\nu>0$ on a compact manifold $M$ without boundary, and if  $G^s(M)$ denotes the Gevrey space on $M$ defined in local coordinates, then 
\begin{equation}\label{EQ:mfd-E}
f\in G^s(M)\;\textrm{ if and only if }\; |\widehat{f}(j)|\leq C e^{-\epsilon \lambda_j^{\frac{1}{s\nu}}},
\end{equation}
for some $\epsilon, C>0$ and for all $j\in\N$, where $\lambda_j$ are the eigenvalues of $E$ and 
$\widehat{f}(j)=(f,e_j)_{L^2(M)}$, and $e_j$ are the corresponding eigenfunctions of $E$. Consequently, using the spectral calculus of $E$, condition \eqref{EQ:mfd-E} can be rewritten as
\begin{equation}\label{EQ:mfd-E2}
f\in G^s(M)\;\textrm{ if and only if }\; \exists D>0\textrm{ such that }
\| e^{D E^{\frac{1}{s\nu}}} f\|_{L^2(M)}<\infty. 
\end{equation}
Moreover, the equivalence \eqref{EQ:mfd-E2} was established in \cite{DR2016} also for more general spaces, including abstract Komatsu classes of ultradifferentiable functions and the corresponding spaces of ultradistributions.

Our aim here is twofold. We want
\begin{itemize}
	\item to investigate what changes in equivalence \eqref{EQ:mfd-E2} if the operator $E$ is no longer elliptic but hypoelliptic. The model case for this will be when, for example, $\nu=2$ and we replace the second order elliptic pseudo-differential operator $E$ by H\"ormander's sum of squares, denoted $\slp$. The main question here is: what space, depending on $\slp$, say $G^s_\slp(M)$, should we put on the left hand side of \eqref{EQ:mfd-E2} for such a characterisation to continue to hold?
	\item to remove the compactness assumption on $M$ while also, for some results, allowing it to have more structure. For example, if $\slp$ is the sub-Laplacian, we are interested in examining the cases where $M$ is a stratified Lie group, for example the Heisenberg group.
\end{itemize}

We are able to give partial answers to these questions.  
Indeed, we introduce new classes of functions, adapting the definition to the fact that we consider hypoelliptic operators. Here, however, we are assisted by the choice of the operator being H\"ormander's sum of squares so that we have a fixed collection of  vector fields at our disposal to work with. In fact, we introduce our suggested definition for Gevrey-type spaces in the setting of general manifolds and then we restrict our considerations to Lie groups. We prove the inclusion of our spaces in spaces of functions defined in terms of the boundedness of a heat kernel operator.

In the other direction, the general characterisation remains open, but we show that the non-commutative Fourier analysis on groups allows us to obtain the desired characterisation on two important groups: the special unitary group $SU(2)$ and the Heisenberg group $\h_n$, which are respectively compact and non-compact.

As we observed for the Euclidean cases, characterisations such as \eqref{EQ:mfd-E2} are particularly relevant for applications, for example for the well-posedness questions for hyperbolic pdes such as in \cite{CDS79}. In fact, in Chapter \ref{CP:3}, we considered the Cauchy problem for the wave equation
\begin{equation}\label{EQ:we}
\partial_t^2 u+a(t)\slp u=0
\end{equation}
for $\slp$ being a positive sub-Laplacian\footnote{In Chapter \ref{CP:3} we consider a more general situation. In fact we study the well-posedness of the Cauchy problem for the wave equation $\partial_t^2 u+a(t)\mathcal{R} u=0$, for $\mathcal{R}$ being a positive Rockland operator on a graded Lie group. Hence above we refer to the special case when $\mathcal{R}=\slp$.} on a stratified Lie group $G$ of step $r$ and for $a\geq 0$ of H\"older regularity. We observed that already in local coordinates it is natural to expect the appearance of Gevrey spaces in such problems, and the class, of functions $f$ satisfying the condition $\| e^{D \slp^{\frac{1}{s\nu}}} f\|_{L^2(G)}<\infty$ for some $D>0$, appeared naturally in the energy estimates for \eqref{EQ:we}. This allowed sharp well-posedness results in sub-Laplacian Sobolev and Gevrey type spaces with the loss of regularity depending on the step $r$; see Corollary \ref{COR:main}. However, a description of such spaces as modifications of the classical Gevrey spaces is missing. 

In the remainder of this Chapter we provide such a description. This will allow us to have a better understanding of the Gevrey hypoellipticity of sub-Laplacians, even if there are still several open questions that we will mention in Chapter \ref{CP:5}. Indeed, the comprehension of these spaces turned to be more complicated than we expected. This prompts the question: why not use, for example, heat kernel techniques, a standard approach in global harmonic analysis? In fact, when first facing this problem one can quickly realise that estimates (independently of the order of differentiation, or at least polynomially or exponentially dependent on this order) of the higher order Riesz transforms or of the derivatives of the heat kernel  would easily answer our questions. At the time of writing these estimates are not available\footnote{In \cite{VSC1992} and in \cite{tER:97} we can find estimates for the derivatives of the heat kernel and of the higher order Riesz transforms, respectively. Unfortunately the control of the dependence of the constants on the order of differentiation seems a currently intractable problem.}, their attainment being a difficult problem in its own right. 

The organisation of this Chapter is as follows. In Section \ref{SEC:mfds}, we introduce our suggested definition for \textit{sub-Laplacian Gevrey spaces} on manifolds, and we prove the inclusion of these spaces in bigger classes of functions. Then in Section \ref{SEC:cpt-groups}, we consider the case of compact Lie groups, and in particular we present the characterisation of sub-Laplacian Gevrey spaces on the special unitary group $SU(2)$. Finally in Section \ref{SEC:heis}, we provide the more detailed description of sub-Laplacian Gevrey spaces in the setting of the Heisenberg group $\h_n$, after recalling the main aspects of its representation theory. In both cases, a complete characterisation is achieved because the symbolic calculus is completely explicit.

\section{Sub-Laplacian Gevrey spaces on manifolds}
\label{SEC:mfds}

Let $M$ be a smooth manifold of dimension $n$ endowed with a  measure $\mu$, on which we will impose assumptions later. Let $\mathbf{X} =\{X_1,\dots,X_r\}$ be a family of vector fields on $M$ satisfying the H\"{o}rmander condition, that is, the vector fields are real-valued and at every point $x$ of $M$ the (real) Lie algebra generated by $\mathbf {X}$ coincides with the tangent space $T_x(M)$ at $x$. We have observed that we can then define the corresponding (positive) sub-Laplacian operator to be
\[
\slp:=-(X_1^2+\dots+X_r^2).
\]
In view of the celebrated H\"ormander theorem \cite{H1967}, presented here in Theorem \ref{HormanderTheorem}, it is clear this is a positive hypoelliptic operator. Such families of vector fields and their properties have been extensively studied in the literature: see, e.g., \cite{B2014} and the references therein. In the context of Lie groups, we can refer to \cite{VSC1992} or \cite{tER2012} for the case of Lie groups with polynomial growth, or also to Chapter \ref{CP:2} for a quick introduction to the topic. 

Our first aim is to introduce Gevrey type classes of functions on 
the manifold $M$ with respect to the H\"ormander system $\mathbf X$ and to characterise these spaces globally, on the spectral side in terms of the sub-Laplacian operator associated with $\mathbf X$. In order to do this, let us start with our proposed definition of the spaces we are interested in.

\begin{definition}[$(\slp,L^\infty)$-Gevrey spaces]\label{DEF:gsl}
	Let $s>0$. The sub-Laplacian Gevrey space $\gamma^s_{\mathbf{X},L^\infty}(M)$ of order $s$ on $M$ is the space of all functions $\phi\in\mathcal{C}^\infty(M)$ satisfying that for every compact set $K\subset M$ there exist two constants $A,C>0$ such that for every $\alpha\in\N^r_0$  we have the inequality 
	\begin{equation}\label{gevrey}
	|\partial^\alpha \phi(x)|\leq C A^{|\alpha|}(\alpha!)^s,\;\textrm{ for all } x\in K.
	\end{equation}
	Here $\partial^\alpha=Y_1\dots Y_{|\alpha|}$, with $Y_j\in\{X_1,\dots,X_r\}$ for every $j=1,\dots,|\alpha|$, and $\sum_{Y_j=X_k}1=\alpha_k$ for every $k=1,\dots,r$.
\end{definition}
Since the differential operators are local, in the subsequent analysis it will be often enough to assume that the function $\phi$ is compactly supported. This restricts the consideration to the `interesting' range\footnote{The only real-analytic function with compact support is the function identically equivalent  to the zero function. This can be easily proved by showing that the support of a compactly supported real-analytic function is empty.} $s>1$. 

We would like to obtain global characterisations of these spaces using the sub-Laplacian operator. More precisely, we want to prove a characterisation of the following type for compactly supported functions $\phi\in\mathcal{C}_0^\infty(M)$:
\begin{equation}\label{EQ:char1}
\phi\in\gamma^s_{\mathbf{X},L^{\infty}}(M)\quad\textrm{ if and only if }\quad
\exists \,D>0\text{ such that }\|e^{D\slp^{\frac{1}{2s}}}\phi\|_{L^2(M)}<\infty.
\end{equation}
Here we assume that $M$ is with no boundary and is equipped with a measure $\mu$ such that the differential operator $\slp$ (defined on $\mathcal{C}^\infty(M)$) is a non-negative essentially self-adjoint operator on $L^2(M)$.
Note that the essential self-adjointness of $\slp$ provides the functional calculus for the sub-Laplacian,\footnote{The sub-Laplacian is an example of a Rockland operator, to which can we apply Theorem \ref{RockFuncCalc} and deduce that if it is essentially self-adjoint then there is a self-adjoint extension; we apply the spectral theorem to the latter.}
and therefore a meaning to the exponential of its fractional power. More precisely, we can apply the spectral theorem in its projection valued form, see \cite{RS1980}. This yields a one-to-one correspondence between our self-adjoint operators and projection-valued measures $E_\lambda$, $\lambda\in\R$, on $L^2(M,\mu)$, given by
\begin{align}
\slp=\int_{-\infty}^{+\infty}\lambda dE_\lambda.
\end{align}
Furthermore, given a real-valued Borel function $g(\cdot)$ on $\R$, the (possibly unbounded) operator
\begin{align}\label{EQ:FuncCalcSlp}
g(\slp)=\int_{-\infty}^{+\infty}g(\lambda) dE_\lambda
\end{align}
is self-adjoint, if 
\begin{align*}
\Dom \big(g(\slp)\big):=\Big{\{} f\in L^2(M,\mu)|\int_{-\infty}^{+\infty}|g(\lambda)|^2(dE_\lambda f,f)<\infty\Big{\}}
\end{align*}  
is dense in $L^2(M,\mu)$, for instance when $g\in L^\infty(\R)$. Hence a more precise version of \eqref{EQ:char1} is
\begin{align*}
\phi\in\gamma^s_{\mathbf{X},L^{\infty}}(M)\quad\textrm{ if and only if }\quad
\exists \,D>0\text{ such that }\phi\in\Dom\big(e^{D\slp^{1/(2s)}}\big).
\end{align*}
In \eqref{EQ:char1}, the $L^2$-norm on the right suggests that it may also be technically convenient to avoid the compact support assumption and to work with functions having all of their derivatives in $L^2(M)$. 

To this end we define the $L^2$-version of the sub-Laplacian Gevrey space.
\begin{definition}[$(\slp,L^{2})$-Gevrey spaces]\label{DEF:gslL2}
	Let $s\geq 1$. The \textit{sub-Laplacian Gevrey space with respect to the $L^2$-norm} $\gamma^s_{\mathbf{X},L^2}(M)$ of order $s$ on $M$ endowed with a measure $\mu$ is the space of all functions $\phi\in \mathcal{C}^\infty(M)$ for which there exist two constants $A,C>0$ such that for every $\alpha\in\N^r_0$ we have
	\begin{equation}\label{gevrey-L2}
	\|\partial^\alpha \phi\|_{L^2(M)}\leq C A^{|\alpha|}(\alpha!)^s.
	\end{equation}
	Here $\partial^\alpha=Y_1\dots Y_{|\alpha|}$, with $Y_j\in\{X_1,\dots,X_r\}$ for every $j=1,\dots,|\alpha|$ and $\sum_{Y_j=X_k}1=\alpha_k$ for every $k=1,\dots,r$.
\end{definition}

In order to analyse the desired characterisation \eqref{EQ:char1}, we will first prove two auxiliary statements of equivalence. We work in an abstract formulation: instead of specifying the precise form of the vector fields or the measure, we instead assume some properties that they should satisfy that allow us to derive the desired conclusions. The first property that we will use is as follows.

\begin{proposition}[Equivalence between $L^\infty$-norm and $L^2$-norm]
	\label{PROP:prop1}
	Let $M$ be a smooth manifold and let $\mathbf{X} =\{X_1,\dots,X_r\}$ be a H\"ormander system. 
	We have the inclusion
	\begin{equation}\label{EQ:inc1}
	\gamma^s_{\mathbf{X},L^{\infty}}(M)\cap \mathcal{C}_0^\infty(M)\subset \gamma^s_{\mathbf{X}, L^2}(M).
	\end{equation}
	Conversely, assume that, with respect to the given measure $\mu$ and the H\"ormander system $\mathbf{X}$, the \textit{Sobolev embedding} holds, namely, that there is $k$ such that $\|f\|_{\infty}\lesssim\sum_{|\alpha|\leq k}\|X^\alpha f\|_{L^2}$. Then we have
	\begin{equation}\label{EQ:inc2}
	\gamma^s_{\mathbf{X}, L^2}(M)\subset \gamma^s_{\mathbf{X},L^{\infty}}(M).
	\end{equation}
\end{proposition}

\begin{proof}
	The inclusion \eqref{EQ:inc1} follows from the embedding $L^\infty\hookrightarrow L^2$ for compactly supported functions, so it remains to show \eqref{EQ:inc2}.
	From our hypothesis, given a function $\phi\in\gamma^s_{\mathbf{X}, L^2}(M)$, we have
	\[
	\|\phi\|_{L^\infty}\leq c_k\sum_{|\beta|\leq k}\|\partial^\beta\phi\|_{L^2}.
	\] 
	Then, for every $\alpha$, we can apply the latter inequality to $\partial^{\alpha}\phi$ so that we have
	\[
	\|\partial^\alpha\phi\|_{L^\infty}\leq c_k\sum_{|\beta|\leq k}\|\partial^{\beta+\alpha}\phi\|_{L^2}\leq c_k\sum_{|\beta|\leq k}C A^{|\alpha+\beta|}\big( (\alpha+\beta)!\big)^s.
	\]
	We recall the following inequalities for the factorial:
	\begin{enumerate}[I.]
		\item $\alpha!\leq|\alpha|!\leq n^{|\alpha|}\alpha!$\,,
		\item $(|\alpha|+k)!\leq 2^{|\alpha|+k}k!|\alpha|!$.
	\end{enumerate}
	Hence
	\begin{align}\notag
	\|\partial^\alpha\phi\|_{L^\infty}&\leq c_kC\sum_{|\beta|\leq k}A^{|\alpha|+k}\big((\alpha+\beta)!\big)^s \leq c'_kA^{|\alpha|+k}\big((|\alpha|+k)!\big)^s\sum_{|\beta|\leq k}1\\\notag
	&\leq c''_kA^{|\alpha|}(2^s)^{|\alpha|}A^k2^{sk}(k!)^sn^{s|\alpha|}(\alpha!)^s\leq C' A'^{|\alpha|}(\alpha!)^s,
	\end{align}
	with $C'=c''_kA^k2^{sk}(k!)^s$ and $A'=A2^sn^s$. This shows that $\phi\in  \gamma^s_{\mathbf{X},L^{\infty}}(M)$.
\end{proof}

The second equivalence is given by:

\begin{proposition}\label{PROP:prop2}
	Let $M$ be a manifold with no boundary equipped with a measure $\mu$ and let $\mathbf{X} =\{X_1,\dots,X_r\}$ be a H\"ormander system with associated sub-Laplacian $\slp=-\sum_{j=1}^rX_j^2$. We assume that $\slp$ is a non-negative essentially self-adjoint operator on $L^2(M)$. Let $\phi$ be a smooth function on $M$. Then the following statements are equivalent:
	\begin{enumerate}[i.]
		\item  there exist $A,C>0$ such that for every integer $k\in\N_0$ we have $\|\slp^k\, \phi\|_{L^2}\leq C A^{2k}((2k)!)^s;$
		\item there exists a constant $D>0$ such that $\|e^{D\slp^{\frac{1}{2s}}}\phi\|_{L^2}<\infty$.
	\end{enumerate}
\end{proposition}

In the following proof we allow a slight abuse of notation: we will use the {factorial function} $(\cdot)!$ for positive real numbers to mean its extension, the \textit{Gamma function} $\Gamma(\cdot)$.

\begin{proof}[Proof of Proposition \ref{PROP:prop2}]
	($i \Rightarrow ii$)
	Assuming \textit{i}, let us show that condition \textit{i} holds not only for positive integers but also for any positive real number $b$. Indeed, given any even integer $a$, by hypothesis, we get
	\begin{align*}
	\|\slp^{\frac{a}{2}}\phi\|_{L^2}\leq CA^a(a!)^s.
	\end{align*}
	Now take any positive real number $b\in\R^+\setminus\N$ and choose  an even integer $a\in\Z_+$ such that $a<b<a+2$. Then there exists $\theta\in (0,1)$ such that $b:=a\theta+(a+2)(1-\theta)$. Therefore, from \eqref{EQ:FuncCalcSlp} we have
	\begin{align*}
	\slp^{\frac{b}{2}}=\int_{-\infty}^{+\infty}\lambda^{\frac{b}{2}}dE_{\lambda}=\int_{-\infty}^{+\infty}\lambda^{\frac{a}{2}\theta}\lambda^{\frac{a+2}{2}(\theta-1)}dE_{\lambda},
	\end{align*}
	which yields
	\begin{align}\label{EQ:sym1}
	\|\slp^{\frac{b}{2}}\phi\|^2_{L^2}=\int_{-\infty}^{+\infty}|\lambda|^{b}d\langle E_{\lambda}\phi,\phi\rangle
	=\int_{-\infty}^{+\infty}|\lambda|^{a\theta}\lambda^{(a+2)(\theta-1)}d\langle E_{\lambda}\phi,\phi\rangle.
	\end{align}
	Hence, we can apply H\"older's inequality to \eqref{EQ:sym1} with $p=\frac{1}{\theta}$ and $q=\frac{1}{1-\theta}$, so that 
	\begin{align}\notag
	\|\slp^{\frac{b}{2}}\phi\|_{L^2}&\leq\|\slp^{\frac{a}{2}}\phi\|^\theta_{L^2}\|\slp^{\frac{a+2}{2}}\phi\|^{(1-\theta)}_{L^2}\leq\big (CA^a(a!)^s\big )^\theta\big (CA^{a+2}((a+2)!)^s\big )^{(1-\theta)}=\\\notag
	&=CA^b(a!)^{s\theta}((a+2)!)^{s(1-\theta)}\leq 2^{3s}C(2^sA)^b(b!)^s=C'(A')^b(b!)^s.
	\end{align}
	Note that the latter inequality derives from the fact that $(\alpha+\beta)!\leq2^{\alpha+\beta}\alpha!\beta!$. Indeed, we have
	\begin{align*}
	\frac{(a!)^\theta((a+2)!)^{1-\theta}}{b!}\leq\frac{(a!)^\theta(a!)^{1-\theta}2^{1-\theta+(a+2)(1-\theta)}}{b!}=\frac{a!}{b!}2^{(1-\theta)3}2^{(1-\theta)a}\leq 2^32^b,
	\end{align*} 
	by observing that $a!/b!\leq 1$ whenever $a\leq b$, and that $0<1-\theta<1$. This means that condition \textit{i} holds for any positive real number. Then, recalling the Taylor expansion for the exponential, we get
	\begin{align}\notag
	\|e^{B\slp^{\frac{1}{2s}}}\phi\|_{L^2}&\leq\sum_{k=0}^\infty \frac{\|B^k\slp^{\frac{k}{2s}}\phi\|_{L^2}}{k!}{\leq}\sum_k\frac{B^k}{k!}CA^{\frac{k}{s}}\Big(\big(k/s\big)!\Big)^s=\\&=C\sum_k\frac{(BA^{\frac{1}{s}})^k}{k!}\Big(\big(k/s\big)!\Big)^s.\notag
	\end{align}
	Let us check that the ratio test allows us to conclude that this series is convergent. In fact we set $T:=BA^{\frac{1}{s}}$ and we calculate the ratio of two consecutive terms of the series whose general term is given by $a_k=T^k/k!\Big(\big(\frac{k}{s}\big)!\Big)^s$, so that we obtain 
	\begin{align*}
	\lim_{k\rightarrow\infty}\frac{a_{k+1}}{a_k}=\lim_{k\rightarrow\infty}\frac{T^{k+1}}{(k+1)!}\Big(\big(k+1/s\big)!\Big)^s\frac{k!}{T^k\big(\big(k/s\big)!\big)^s}=\lim_{k\rightarrow\infty}\frac{T}{k+1}\frac{\big(\big(k+1/s\big)!\big)^s}{\big(\big(k/s\big)!\big)^s}.
	\end{align*} 
	Now, Stirling's approximation $ n ! \sim \sqrt{2\pi n} (n/e)^n $ implies 
	\begin{align*}
	\lim_{k\rightarrow\infty}\frac{a_{k+1}}{a_k}=\lim_{k\rightarrow\infty}\frac{T}{k+1}\sqrt{\frac{2\pi}{s}}\frac{(k+1)^{\frac{s}{2}+k+1}}{(se)^{k+1}}\frac{(se)^k}{\sqrt{\frac{2\pi}{s}}k^{\frac{s}{2}+k}}=\frac{T}{se}\lim_{k\rightarrow\infty}\frac{(k+1)^{\frac{s}{2}+k}}{k^{\frac{s}{2}+k}}=\frac{T}{s}.
	\end{align*} 
	Choosing the constant $B<sA^{-\frac{1}{s}}$, we have that the limit as $k$ goes to infinity of $\frac{a_{k+1}}{a_k}$ is strictly less than $1$. Thus by the ratio convergence test we deduce that
	\begin{align*}
	\sum_{k=1}^\infty\frac{(BA^{\frac{1}{s}})^k}{k!}\Big(\big(k/s\big)!\Big)^s<\infty,
	\end{align*}
	and this yields condition \textit{ii.}\\
	
	($ii \Rightarrow i$) 
	We now assume that the $L^2$-norm of $e^{D\slp^{\frac{1}{2s}}}\phi$ is finite. Then, taking into account norm properties, we get for any integer $k\in\N_0$ and $\phi\in\mathcal C^\infty_0$ that
	\begin{align}\notag
	\|\slp^k \phi\|_{L^2}&=\|\slp^k e^{-D\slp^{\frac{1}{2s}}}e^{D\slp^{\frac{1}{2s}}}\phi\|_{L^2}\leq\|\slp^k e^{-D\slp^{\frac{1}{2s}}}\|_{L^2\rightarrow L^2}\|e^{D\slp^{\frac{1}{2s}}}\phi\|_{L^2}\leq\\\notag
	&\leq C\|\slp^k e^{-D\slp^{\frac{1}{2s}}}\|_{L^2\rightarrow L^2}\leq C \sup_{\lambda>0} \lambda^k e^{-D\lambda^{\frac{1}{2s}}}.
	\end{align}
	If we set the function $f(\lambda):=\lambda^k e^{-D\lambda^{\frac{1}{2s}}}$, then its maximum is achieved at $\lambda_D=\big (\frac{2ks}{D}\big)^{2s}$. Therefore, we deduce that
	\[
	\|\slp^k \phi\|_{L^2}\leq C e^{-2ks}\Big( \frac{2ks}{D}\Big)^{2ks}\leq C\Big(\frac{s^s}{D^s}\Big)^{2k}\big((2k)!\big)^s,
	\]
	where we use the estimate $N^N e^{-N}\leq N!$, see, e.g., \cite{R1993}. Defining a constant $A:=s^s/D^s$, the claim follows.
\end{proof}

Thus, considering the two equivalences given in Propositions \ref{PROP:prop1} and \ref{PROP:prop2}, to get the global implication, we want to prove a further equivalence linking `general' differentiation to powers of the sub-Laplacian in the following way:
\begin{align}\notag
&\exists A,C >0 \text{ s.t. }\forall\alpha\in\N^r\text{ we have }\|\partial^\alpha\phi\|_{L^2}\leq C A^{|\alpha|}(\alpha!)^s\iff\,\\\notag&\exists B,D>0\text{ s.t. }\forall k\in\N_0\text{ we have } \|\slp^{k}\phi\|_{L^2}\leq CA^{2k}((2k)!)^s.
\end{align}
This equivalence remains a conjecture, but we are able to show one implication in full generality:

\begin{proposition}[$\partial^\alpha\rightarrow\slp^k$ implication]\label{PROP:prop3}
	Let $M$ be a manifold equipped with a measure $\mu$. Let $\mathbf{X} =\{X_1,\dots,X_r\}$ be a H\"ormander system and $\slp$ the associated sub-Laplacian $\slp=-\sum_{j=1}^rX_j^2$. Then
	\begin{align}\notag
	\phi\in\gamma^s_{\mathbf{X},L^2}(M)\Longrightarrow\;&\text{there exist }A,C>0\text{ such that for every }k\in\N_0\text{ we have }\\\notag
	&\|\slp^{k}\phi\|_{L^2}\leq CA^{2k}((2k)!)^s.
	\end{align}
\end{proposition}

\begin{proof}
	Recall the multinomial theorem adapted to non-commutative elements of an algebra of vector fields, according to which we have
	\[
	(Y_1+\dots+Y_m)^h=\frac{1}{m!}\sum_{k_1+\dots+k_m=h}\frac{h!}{k_1!\dots k_m!}\sum_{\sigma\in \text{sym}(m)}\prod_{1\leq t \leq m}Y_{\sigma(t)}^{k_t}.
	\]
	By using this theorem, and using $\slp=-(X_1^2+\dots+X_r^2)$ and the factorial inequalities mentioned in the proof of Proposition \ref{PROP:prop1}, we obtain
	\begin{align*}
	\|\slp^k\phi\|_{L^2}&\leq\frac{1}{r!}\sum_{|\alpha|=k}\frac{k!}{\alpha!}\sum_{\sigma\in\text{sym}(r)}\big{\|}\underbrace{ (X_{\sigma(1)}^2)^{\alpha_1}\dots(X_{\sigma(r)}^2)^{\alpha_r} }_{\partial^{2\alpha}} \phi\big{\|}_{L^2} \\
	&{\leq}C\sum_{|\alpha|=k}\frac{k!}{\alpha!}A^{2|\alpha|}((2\alpha)!)^s\leq CA^{2k}((2k)!)^s\sum_{|\alpha|=k}\frac{k!}{\alpha!} \\
	&\leq CA^{2k}((2k)!)^s\sum_{|\alpha|=k}\frac{k!}{|\alpha|!}r^{|\alpha|}\leq CA^{2k}((2k)!)^sr^kk^{r-1}\leq CA'^{2k}((2k)!)^s,
	\end{align*}  
	with $A'=Ar$, and where we used that $\frac{k^{r-1}}{r^k}\leq 1$.
\end{proof}
Thus, the `only if' part of the characterisation \eqref{EQ:char1} follows:

\begin{theorem}\label{THM:implication}
	Let $M$ be a manifold with no boundary equipped with a measure $\mu$ and let $\mathbf{X} =\{X_1,\dots,X_r\}$ be a H\"ormander system with associated sub-Laplacian $\slp=-\sum_{j=1}^rX_j^2$. Assume that the \textit{Sobolev embedding} holds with respect to the measure $\mu$ and the H\"ormander system $\mathbf{X}$, i.e. there is some $m$ such that $\|f\|_{L^\infty}\lesssim\sum_{|\alpha|\leq m}\|X^\alpha f\|_{L^2}$. We also assume that $\slp$ is a non-negative essentially self-adjoint operator on $L^2(M)$. Then we have
	\begin{align*}
	\gamma^s_{\mathbf{X},L^2}(M)\subset\big\{\phi\in\mathcal C^\infty(M)\,:\,\exists D>0\text{ such that }\|e^{D\slp^{\frac{1}{2s}}}\phi\|_{L^2}<\infty\big\}.
	\end{align*}
\end{theorem}

\begin{proof}
	The inclusion follows after applying Propositions \ref{PROP:prop2} and \ref{PROP:prop3}.
\end{proof}

Combining Proposition \ref{PROP:prop1} and Theorem \ref{THM:implication} we see that if $M$ is a smooth manifold endowed with a Borel measure and $\slp$ is a sub-Laplacian associated to a H\"ormander system of vector fields then we have an inclusion involving our \textit{original} sub-Laplacian Gevrey space $\gamma^s_{\mathbf{X},L^{\infty}}(M)$:
\begin{corollary}\label{COR:implication}
	Under the same hypotheses as Theorem \ref{THM:implication}, we have the inclusion
	\begin{equation*}
	\gamma^s_{\mathbf{X},L^{\infty}}(M)\cap \mathcal C_0^\infty(M)\subset\big\{\phi\in\mathcal C_0^\infty(M)\,:\,\exists D>0\text{ such that }\|e^{D\slp^{\frac{1}{2s}}}\phi\|_{L^2}<\infty\big\}.
	\end{equation*}
\end{corollary}
\begin{remark}\label{REM:oppImplication}
	We would like to prove the reverse inclusion to that in Theorem \ref{THM:implication} or, alternatively, the reverse implication to that in Proposition \ref{PROP:prop3}, i.e.,
	\begin{align}\notag
	\big(\text{there exist }&A,C>0\text{ such that }\forall k\in\N_0\text{ we have } \|\slp^{k}\phi\|_{L^2}\leq CA^{2k}((2k)!)^s \big)\\\notag
	& \implies\big(\text{for every }\alpha\text{ we have }\|\partial^\alpha\phi\|_{L^2}\leq C A^{|\alpha|}(\alpha!)^s\big).
	\end{align}
	In order to obtain this implication, we may look at any derivatives of $\phi$ in the following way:
	\begin{align}\label{EQ:insertingIdentity}
	\|\partial^\alpha\phi\|_{L^2}=\|\partial^\alpha(\slp)^{-\frac{|\alpha|}{2}}(\slp)^{\frac{|\alpha|}{2}}\phi\|_{L^2}\leq \|\partial^\alpha(\slp)^{-\frac{|\alpha|}{2}}\|_{\OP_{L^2}}\|(\slp)^{\frac{|\alpha|}{2}}\phi\|_{L^2},
	\end{align}
	where in the latter inequality we have used norm properties.
	Then, taking into account the hypothesis on the $L^2$-norm of powers of the sub-Laplacian, to prove the conclusion we need to show the boundedness---independently of $\alpha$, or at least polynomially or exponentially dependent on $\alpha$---of the $\OP_{L^2}$-norm of the higher order Riesz transform
	\begin{align}\label{EQ:RieszTransform}
	R_\alpha:=\partial^\alpha(\slp)^{-\frac{|\alpha|}{2}}.
	\end{align}
	In the context of Lie groups, the natural choice of measure is a Haar measure $\mu$, outlined in Definition \ref{DEF:haar}. Recall that if $\mu$ is a left-invariant Haar measure and $X$ is a left-invariant vector field then $X^*=-\overline{X}$, since, representing $X$ according to \eqref{vectorfield}, for every $f_1,f_2\in\mathcal{C}^\infty_0(G)$ we have
	\begin{align*}
	\int_G (Xf_1) \overline{f_2} d\mu=\frac{d}{dt}\Bigr{|}_{t=0}\int_Gf_1\big(xe^{tX}\big)\overline{f_2}(x)d\mu(x).
	\end{align*}
	Performing the change of variable $y=xe^{tX}$, the latter inequality becomes
	\begin{align*}
	\int_G (Xf_1) \overline{f_2} d\mu&=\frac{d}{dt}\Bigr{|}_{t=0}\int_Gf_1(y)\overline{f_2}(ye^{-tX})d\mu(y)=-\int_G f_1(X\overline{f_2})d\mu=\\
	&=-\int_G f_1\overline{\overline{X}f_2}d\mu.
	\end{align*}
	Note that a similar but more detailed argument can be found in the proof of part $2$ of Proposition \ref{Prop:groupFouriertransform}. 
	
	Therefore given $k$ real-valued, left-invariant vector fields $X_1,\dots,X_k$, setting $\slp=-\sum_{j=1}^kX_j^2$, for any $f\in\mathcal{C}_0^\infty(G)$ we have
	\begin{align*}
	\big(\slp f_1,f_2\big)_{L^2}=-\sum_{j=1}^k\int_G\big(X_j^2f\big)\overline{f}d\mu=\sum_{j=1}^k\int_GX_jfX_j\overline{f}d\mu=\sum_{j=1}^k\|X_j\|_{L_2}^2,
	\end{align*}
	where in the second-to-last equality we have used the result proved above $X_j^*=-X_j$, and in the last equality we used $X_j\overline{f}=\overline{X_jf}$ since $X_j$ is real-valued. Hence, for every $f\in\mathcal{C}_0^\infty(G)$ we have
	\begin{align}\label{EQ:boundXf}
	\|X_jf\|_{L^2}^2\leq  \big(\slp f,f\big)_{L^2}.   
	\end{align}
	When $\slp$ is essentially self-adjoint, inequality \eqref{EQ:boundXf} yields
	\begin{align*}
	\|X_jf\|_{L^2}^2\leq\|\slp^{1/2}f\|_{L^2},
	\end{align*} 
	for all $f\in\mathcal{C}_0^\infty(G)\cap\Dom(\slp^{1/2})$. Therefore, inequality \eqref{EQ:insertingIdentity} is bounded for $|\alpha|=1$. But for $|\alpha|\geq 2$, $G$ with polynomial growth and $\slp$ a sub-Laplacian (i.e. $\mathbf{X}$ satisfies the H\"ormander condition), results in \cite{ER1999} yield that \eqref{EQ:insertingIdentity} is bounded only when $G$ is a direct product of a compact and a nilpotent Lie groups.  
	
	We highlight that several other authors have been interested in proving inequalities for the Riesz transform. For example, in \cite{CMZ1996, LP2004} heat kernel techniques have been used to determine estimates for the \textit{first order} Riesz transform $X\slp^{\frac{1}{2}}$ on a generalisation of the Heisenberg group. The lack of estimates for derivatives of the heat kernel, independent of the order of derivation, prevents us from extending those inequalities to higher order Riesz transforms. Nevertheless, we develop a different argument to achieve the characterisation on the Heisenberg group, as we will show in Subsection \ref{Hn}.  
\end{remark}

In the next sections we add further structure to our manifolds, in that we consider the case of (compact and non-compact) Lie groups. In doing so we obtain the desired control on the constants in the settings of the special unitary group $SU(2)$ and of the Heisenberg group.

\section{Sub-Laplacian Gevrey spaces on compact Lie Groups}
\label{SEC:cpt-groups}

In this Section we discuss the reverse inclusion to the one in Theorem \ref{THM:implication} in the setting of \textit{compact} Lie groups. While we are still unable to prove it for general (compact and non-compact) Lie groups, we will use the well-known non-commutative Fourier analysis on the compact group $SU(2)$ to show the converse inclusion in the latter case. We start by setting up the framework for the pseudo-differential analysis on compact Lie groups.

\medskip
Assume now that $G$ is a \textit{compact} Lie group. We equip $G$ with the bi-invariant Haar measure. 
Let $\Gh$ be the unitary dual of $G$, that is, the set of equivalence classes of continuous irreducible unitary representations of $G$. As discussed in Chapter \ref{CP:2}, to simplify the notation we will not distinguish between representations and their equivalence classes. Since $G$ is compact, $\Gh$ is discrete and all the representations are finite-dimensional. Therefore, given $\xi\in\Gh$ and a basis in the representation space of $\xi$, we can view $\xi$ as a matrix-valued function $\xi:G\rightarrow \C^{d_\xi\times d_\xi}$ where $d_\xi$ is the dimension of this representation space.

For a function $f\in L^1(G)$, the group Fourier transform at $\xi\in\Gh$ is defined as 
\begin{align*}
\widehat{f}(\xi)=\int_G f(x)\xi(x)^* dx,
\end{align*}
where $dx$ is the Haar measure on $G$. Applying the Peter--Weyl theorem (see \cite{RT2009}), we obtain the Fourier inversion formula (for instance for $f\in\mathcal{C}^\infty(G)$)
$$
f(x)=\sum_{\xi\in\Gh} d_\xi {\rm Tr}(\xi(x) \widehat{f}(\xi)).
$$
Moreover, the Plancherel identity holds and we have
\begin{align*}
\|f\|_{L^2(G)}=\Big(\sum_{\xi\in\Gh}d_{\xi}\|\widehat{f}(\xi)\|^2_{HS}\Big)^{1/2}=:\|\widehat{f}\|_{l^2(\Gh)}.
\end{align*}
Here, since $\widehat {f}(\xi)\in\C^{d_\xi\times d_\xi}$ is a matrix, $\|\widehat {f}(\xi)\|_{\HS}$ stands for its Hilbert--Schmidt norm. We recall that for any matrix $A\in\C^{d\times d}$ it is defined by 
\[
\|A\|_{\HS}:=\langle A,A\rangle_{\HS}^{\frac{1}{2}}=\sqrt{\sum_{i,j=1}^d \overline{A_{ij}}A_{ij}}.
\]
Given a left-invariant operator $T$ on $G$ (more precisely $T:\mathcal{D}(G)\rightarrow\mathcal{D}'(G)$ with $T\big(f(x_0\cdot)\big)x=\big(Tf\big)(x_0x)$), its matrix-valued symbol is $\sigma_T(\xi)=\xi(x)^* T\xi(x)\in \C^{d_\xi\times d_\xi}$ for each representation $\xi\in \Gh$. Therefore, formally (or for all $f$ such that $\widehat{f}(\pi)=0$ for all but a finite number of $\pi\in\Gh$) we have
$$
Tf(x)=\sum_{\xi\in\Gh} d_\xi {\rm Tr}(\xi(x)\sigma_T(\xi) \widehat{f}(\xi)).
$$
In other words $T$ is a Fourier multiplier with symbol $\sigma_T$.
For the details of these constructions we refer the reader to \cite{RT2009,RT2013, T1986}.
To simplify the notation, we can also denote $\sigma_T(\xi)$ by $\widehat{T}(\xi)$ or simply by $\widehat{T}$.

\subsection{Fourier descriptions of sub-Laplacian spaces}

Let $\mathbf{X}=\{X_1,\dots,X_r\}$ be a H\"ormander system of left-invariant vector fields. The associated sub-Laplacian
\begin{align*}
\slp=-\sum_{j=1}^rX_j^2 
\end{align*}
is non-negative and essentially self-adjoint on $L^2(G)$. Its kernel in $L^2(G)$ comprises the constant functions.
If $G$ is a compact Lie group, functions satisfying $\|e^{D\slp^{\frac{1}{2s}}}\phi\|_{L^2(G)}<\infty$ can be described in terms of the behaviour of their Fourier coefficients, as a consequence of the Plancherel theorem on $G$. Here, we denote by $\widehat\slp=\widehat\slp(\xi)$ the matrix symbol of $\slp$ at $\xi\in\Gh$. Since $\slp$ is a non-negative operator, it follows that $\widehat{\slp}(\xi)$ is a positive matrix and we can always choose a basis in representation spaces such that $\widehat\slp=\widehat\slp(\xi)$ is a positive diagonal matrix.

\begin{proposition}\label{PROP:ft-equiv}
	Let $G$ be a compact Lie group and $\slp$ a sub-Laplacian on $G$. Let $\phi\in C^{\infty}(G)$. Then the following statements are equivalent:\\
	$(i)$ There exists a constant $B>0$ such that for every $\xi\in\widehat G$ we have
	\begin{equation}\label{HS}
	\|e^{B\widehat\slp(\xi)^{\frac{1}{2s}}}\widehat \phi(\xi)\|_{\HS}<\infty.
	\end{equation}
	$(ii)$ There exists a constant $D>0$ such that  
	\begin{equation}\label{l2}
	\|e^{D\slp^{\frac{1}{2s}}}\phi\|_{L^2(G)}<\infty.
	\end{equation}
\end{proposition} 
In order to prove Proposition \ref{PROP:ft-equiv}, we need the following:
\begin{lemma}\label{Lem:Bessel}
	Let $2s>\frac{n}{2}$. Then the $l^2(\Gh)$-norm of $(1+\widehat{\slp})^{-s}$ is finite, meaning
	\begin{align*}
	\sum_{[\xi]\in\Gh}d_\xi\|\big(I+\widehat{\slp}(\xi)\big)^{-s}\|_{HS}^2<\infty.
	\end{align*}
\end{lemma}
\begin{proof}
	We recall that for every $s>0$ the \textit{gamma function} at $s$ is defined to be
	\begin{align*}
	\Gamma(s):=\int_{0}^\infty t^{s-1}e^{-t}dt.
	\end{align*}	
	From this, the change of variable $t=\lambda t'$ with $\lambda>0$ yields 
	\begin{align*}
	\lambda^{-s}=\frac{1}{\Gamma(s)}\int_{0}^\infty t^{s-1}e^{-\lambda t}dt.
	\end{align*}
	The functional calculus of the sub-Laplacian $\slp$ formally allows us to consider $\lambda=I+\slp$ in the latter identity, and this implies 
	\begin{align*}
	(I+\slp)^{-s}=\frac{1}{\Gamma(s)}\int_0^\infty t^{s-1}e^{-(I+\slp)t}dt.
	\end{align*}
	Therefore the kernel of the convolution operator $(I+\slp)^{-s}$ is
	\begin{align*}
	(I+\slp)^{-s}\delta_0=\frac{1}{\Gamma(s)}\int_0^\infty t^{s-1}e^{-t}h_t dt,
	\end{align*}
	where $h_t=e^{-t\slp}\delta_0$ is the heat kernel of the sub-Laplacian $\slp$, see \cite{VSC1992} for more details. Now we recall that $h_t(x)\geq 0$, symmetric in the sense that $h_t(x^{-1})=-h_t(x)$  and \cite[VIII 2.2 Proposition]{VSC1992} yields that 
	\begin{align*}
	h_t(e)\lesssim t^{-n/2}, 
	\end{align*}
	where $e$ is the identity of the group. Consequently, $(I+\slp)^{-s}\delta_0$ is a real-valued symmetric function and
	\begin{align*}
	\|(I+\slp)^{-s}\delta_0\|_{L^2}^2&=(I+\slp)^{-s}\delta_0\star (I+\slp)^{-s}\delta_0(e)=\\
	&=\frac{1}{\Gamma(s)^2}\int_0^\infty \int_0^\infty t_1^{s-1}e^{-t_1}t_2^{s-1}e^{-t_2}h_{t_1}\star h_{t_2}(e)dt_1dt_2= \\
	&=  \frac{1}{\Gamma(s)^2}\int_0^\infty \int_0^\infty (t_1t_2)^{s-1}e^{-(t_1+t_2)}h_{t_1+t_2}(e)dt_1dt_2.
	\end{align*}
	Considering the change of variable $u=t_1+t_2$, i.e. $t_2=u-t_1$, we obtain 
	\begin{align*}
	\|(I+\slp)^{-s}\delta_0\|_{L^2}^2&=\frac{1}{\Gamma(s)^2}\int_0^\infty \int_{t_1=0}^u \big(t_1(u-t_1)\big)^{s-1}e^{-u}h_u(e)dt_1du=\\
	&=\frac{1}{\Gamma(s)^2}\int_0^1\big(t(1-t)\big)^{s-1}dt\int_0^\infty u^{2(s-1)}e^{-u}h_u(e)u~du,
	\end{align*}
	where in the last equality we have used the change of variable $t_1=tu$. We observe that the first integral is bounded, so we only need to understand the behaviour of the second integral. Recalling that $h_u(e)\lesssim u^{-n/2}$, we have
	\begin{align*}
	\int_0^\infty u^{2(s-1)}e^{-u}h_u(e)u~du\lesssim \int_0^\infty u^{2(s-1)-\frac{n}{2}+1}e^{-u}du.
	\end{align*}
	The last integral is bounded if and only if 
	\begin{align*}
	2(s-1)-\frac{n}{2}+1>-1\quad\iff\quad 2s>\frac{n}{2}.
	\end{align*}  
	(For $u$ that goes to infinity the integral is bounded, but when $u$ approaches $0$ we need to impose the above condition). Hence for $2s>n/2$ we have
	\begin{align*}
	\|(I+\slp)^{-s}\delta_0\|_{L^2}^2<\infty.
	\end{align*}
	As the $L^2(G)$-norm of $(I+\slp)^{-s}\delta_0$ is finite when $2s>\frac{n}{2}$, so is the $l^2(\Gh)$-norm of $(I+\widehat{\slp})^{-s}$. This completes the proof.
\end{proof}
\begin{proof}[Proof of Proposition \ref{PROP:ft-equiv}]
	$(i)\Rightarrow (ii)$ \\
	We assume that there exists $B>0$ such that for every $\xi\in\widehat G$ the estimate \eqref{HS} holds. Take an arbitrary constant $D$ (we will choose it at the end). Then applying the Plancherel theorem\footnote{In Chapter \ref{CP:2} we have stated the Plancherel theorem for the more general case of simply connected, nilpotent Lie group, see Theorem \ref{Thm:Plancherelformula}. In \cite[Corollary 7.6.7]{RT2009}, we can find the equivalent statement for compact groups.} and the definition of the $l^2(\widehat G)$-norm we have
	\begin{align}\notag
	\|e^{D\slp^{\frac{1}{2s}}}\|^2_{L^2(G)}=\|e^{D\widehat\slp^{\frac{1}{2s}}}\widehat\phi\|^2_{l^2(\widehat G)}=\sum_{[\xi]\in\widehat G} d_\xi\|e^{D\widehat\slp^{\frac{1}{2s}}} \widehat{ \phi}(\xi)\|^2_{\HS}.
	\end{align} 
	Introducing $(I+\widehat\slp)^N(I+\widehat\slp)^{-N}$ where we consider $N\gg 1$ and splitting the exponential we get
	\begin{align}\notag
	\|e^{D\slp^{\frac{1}{2s}}}\|^2_{L^2(G)}=\sum_{[\xi]\in\widehat G} d_\xi\|e^{(D-B)\widehat\slp^{\frac{1}{2s}}}(I+\widehat\slp)^N(I+\widehat\slp)^{-N}e^{B\widehat\slp^{\frac{1}{2s}}}\widehat \phi(\xi)\|^2_{\HS}.
	\end{align}
	Now, choose the constant $D$ such that the new constant $D':=D-B$ is strictly less than 0 and recall that, given bounded $T$ and Hilbert--Schmidt $H$, we have
	\[
	\|TH\|_{\HS}\leq \|T\|_{\OP}\|H\|_{\HS}.
	\]
	Then
	\begin{align}\notag
	\|e^{D\slp^{\frac{1}{2s}}}\|^2_{L^2(G)}\leq\sum_{[\xi]\in\widehat G}d_\xi\|(I+\widehat\slp)^{-N}\|^2_{\HS}\|e^{D'\widehat\slp^{\frac{1}{2s}}}(I+\widehat\slp)^N\|^2_{\OP}\|e^{B\widehat\slp^{\frac{1}{2s}}}\widehat\phi(\xi)\|^2_{\HS}.\end{align}
	By hypothesis $\|e^{B\widehat\slp^{\frac{1}{2s}}}\widehat\phi(\xi)\|^2_{\HS}:= K$ is finite. Furthermore, we can define the multiplier
	\[
	m(\lambda):=e^{D'\lambda^{\frac{1}{2s}}}(I+\lambda)^N,\quad\text{with }D'<0.
	\]
	Formally evaluating this multiplier in $\widehat\slp$ we get exactly the operator which we are interested in, that is, $m(\widehat\slp)=e^{D'\widehat\slp^{\frac{1}{2s}}}(I+\widehat\slp)^N$. Thus, we can bound by a constant $K'$ another term in the argument of the previous sum, observing that
	\[
	\|m(\widehat\slp)\|_{\OP_\xi}=\sup_{\lambda\in\sigma(\widehat\slp)}|m(\lambda)|<\infty.
	\]
	Therefore, we obtain
	\[
	\|e^{D\slp^{\frac{1}{2s}}}\|^2_{L^2(G)}\leq K\,K'\sum_{[\xi]\in\widehat G}d_\xi\|(I+\widehat\slp(\xi))^{-N}\|^2_{\HS}<\infty,
	\]
	where the final inequality follows from the boundedness for sufficiently large $N$ (more precisely, $2N>n/2$) of the $l^2(\Gh)$-norm of $(1+\widehat{\slp})^{-N}$, as proved in Lemma \ref{Lem:Bessel}.
	
	(For more precise estimates involving the step in H\"ormander's condition the reader can see \cite[Proposition 3.1]{GR2015}). 
\end{proof}

\subsection{Laplacian Gevrey spaces}\label{SSEC:laplacian}

In this Subsection we briefly recall the characterisation of Gevrey spaces corresponding to the Laplace operator that was obtained in \cite{DR2014}. We will also present an alternative, shorter proof for such a characterisation, independently of the symbolic calculus.

Let $G$ be a compact Lie group. For a basis $X_1,\dots,X_n$ of the Lie algebra $\mathfrak{g}$ of $G$ and any multi-index $\alpha\in\N^n_0$, we define the left-invariant differential operator of order $\alpha$, $|\alpha|=\alpha_1+\dots+\alpha_n$, as a composition of $X_j$'s such that each $X_k$ enters our operator exactly $\alpha_k$ times, that is
\[
\partial^\alpha:= Y_1\dots Y_{|\alpha|}, 
\]
with $Y_j\in\{X_1,\dots, X_n\}$ for all $1\leq j\leq |\alpha|$ and $\sum_{j\,\text{s.t.}\,Y_j=X_k}1=\alpha_k$, $1\leq k \leq n$. 
We underline that here we consider \textit{all} the elements of the basis of the Lie algebra, without restricting ourselves to a H\"ormander system. Therefore we can consider the positive elliptic {Laplace--Beltrami operator} on $G$, defined by
\[
\Delta:=-(X_1^2+\dots+X_n^2).
\]
For all the elements of the unitary dual space of our compact Lie group, $[\xi]=(\xi_{ij})_{1\leq i,j\leq d_{\xi}}\in\widehat G$, we denote by $\lambda^2_{[\xi]}$ the associated eigenvalue for the Laplace--Beltrami operator $\Delta$. Then the eigenvalue corresponding to the representation $[\xi]$ for the operator $(1+\Delta)^{\frac{1}{2}}$ is given by  
\[
\langle\xi\rangle:=(1+\lambda^2_{[\xi]})^{\frac{1}{2}}.
\] 
For each $s\geq 1$, the Gevrey space $\gamma^s(G)$ defined in local coordinates of $G$ can be described as the space of functions $\phi\in\mathcal{C}^\infty(G)$ for which there exist two positive constants $A$ and $C$ such that for all $\alpha\in\N^n_0$, we have
\[
\|\partial^\alpha \phi\|_{L^\infty}=\sup_{x\in G}|\partial^\alpha \phi(x)|\leq CA^{|\alpha|}(\alpha !)^s.
\]
Then we have the following characterisations of $\gamma^s(G)$ in terms of Fourier coefficients of functions, and also in terms of the space $\gamma^s_{\mathbf{X},L^{\infty}}(G)$ from Definition \ref{DEF:gsl} with $\slp=\Delta$ and $\mathbf{X}=\{X_1,\dots,X_n\}$ being the basis of $\mathfrak{g}$. In the latter case we will also write $\gamma^s_\Delta(G)$.

\begin{theorem}\label{THM:DR2014}
	Let $0<s<\infty$. The following statements are equivalent:
	\begin{enumerate}[i.]
		\item $\phi\in \gamma^s(G)$; 
		\item $\phi\in \gamma^s_\Delta(G)$; 
		\item there exist $B>0$ and $K>0$ such that
		\[
		\|\widehat {\phi}(\xi)\|_{\HS}\leq K e^{-B\langle\xi\rangle^{\frac{1}{s}}}
		\]
		holds for all $\xi\in\widehat G$.
	\end{enumerate}
\end{theorem}
\begin{proof}[Symbolic-calculus-independent proof of Theorem \ref{THM:DR2014}]~
	
	The arguments that we have developed so far for sub-Laplacian Gevrey spaces work perfectly in the case of the (Laplace--Beltrami)--Gevrey spaces on compact groups. In fact the Laplacian can be seen as a particular example of sub-Laplacian. 
	
	$i.\iff ii.$
	
	The equivalence between $i.$ and $ii.$ follows straightforwardly from the definition of $\gamma^s(G)$ and $\gamma_{\Delta}^s(G)$, since we are working in a compact setting.
	
	$ii.\implies iii.$
	
	Applying the same arguments as in Proposition \ref{PROP:ft-equiv} to the Laplace--Beltrami operator, we deduce that for every function $\phi\in\mathcal{C}^{\infty}(G)$ the existence of a positive constant $D>0$ such that
	\begin{align*}
	\|e^{D\Delta^{1/2s}}\phi\|_{L^2(G)}<\infty
	\end{align*}
	is equivalent to the existence of a positive constant $B>0$ satisfying that for every $\xi\in\Gh$ we have
	\begin{align*}
	\|e^{B\widehat{\Delta}^{1/2s}(\xi)}\widehat{\phi}(\xi)\|_{HS}= \|e^{B\langle \xi\rangle^{1/s}(\xi)}\widehat{\phi}(\xi)\|_{HS}=e^{B\langle \xi\rangle^{1/s}(\xi)}\|\widehat{\phi}(\xi)\|_{HS}<\infty,
	\end{align*}
	which is equivalent to $iii$. Therefore Corollary \ref{COR:implication} yields $ii\implies iii$.
	
	$iii.\implies ii.$
	
	Combining together Proposition \ref{PROP:ft-equiv} and Proposition \ref{PROP:prop2}, the proof of this implication is equivalent to the proof of the boundedness of the higher order Riesz transform $\partial^{\alpha}\Delta^{-\frac{|\alpha|}{2}}$, following the reasoning of Remark \ref{REM:oppImplication}, and in particular looking at the inequality \eqref{EQ:insertingIdentity}. 
	
	It follows from Remark \ref{REM:oppImplication} that for every $j\in\{1,\dots,n\}$ we have 
	\begin{align}\label{EQ:order1}
	\|X_j\slp^{-\frac{1}{2}}\|_{\Op_{L^2_0}}\leq 1,
	\end{align}
	where $L^2_0$ is the orthogonal complement in $L^2(G)$ of the space of constant functions on $G$.
	The commutativity of the Laplace--Beltrami operator plays a fundamental r\^ole to show the boundedness of the Riesz transform for any $\alpha\in\N_0$. In fact, assuming $\slp =\Delta$, we have
	\begin{align}\notag
	\|\partial^\alpha\Delta^{-\frac{|\alpha|}{2}}\|_{L^2_0\rightarrow L^2}=\|X_1\Delta^{-\frac{1}{2}}X_2\dots X_{|\alpha|}\Delta^{-\frac{1}{2}}\|_{L^2_0\rightarrow L^2}\leq 1,
	\end{align}
	obtained applying the inequality \eqref{EQ:order1} repeatedly $|\alpha|$ times. Then we immediately obtain the implication $iii\implies ii$, and, therefore, the characterisation is fully achieved.
\end{proof}

\subsection{Sub-Laplacian Gevrey spaces on $SU(2)$}
\label{SEC:su2}

In this Subsection we show that in the case of the canonical sub-Laplacian on the special unitary group $SU(2)$, we have the converse inclusion to that in Theorem \ref{THM:implication}. First, we recall the notation and the necessary tools to develop our argument, presented already in Chapter \ref{CP:2}.

Let $G= {\rm SU(2)}$ and let $X,Y,Z$ be a basis for its Lie algebra $\mathfrak{su}(2)$ such that $[X,Y]=Z$. According to our notation, we can define the positive sub-Laplacian as 
\begin{equation}\label{slpSU2}
\slp:=-(X^2+Y^2).
\end{equation}
Given a function $f\in\mathcal{C}^\infty(G)$ we denote for any multi-index $\alpha=(\alpha_1,\alpha_2)\in\N^2_0$ the differential operator
\[
\partial^\alpha f:=X_1\dots X_{|\alpha|}f,
\]
with $X_j\in\{X,Y\}$, $\sum_{X_j=X}1=\alpha_1$ and $\sum_{X_j=Y}1=\alpha_2$.

The symbols of left-invariant vector fields on $SU(2)$ have been explicitly calculated in 
\cite[Theorem 12.2.1]{RT2009}, and we have mentioned them in Subsection \ref{SEC:SU2}. Nevertheless we recall them below, since they will be essential in the proof of the upcoming proposition. Thus, according to \eqref{EQ:x} and \eqref{EQ:y} the symbols of $X$ and $Y$ are given by
\begin{align}\label{x}
&\sigma_X(l)_{m,n}=-\sqrt{(l-n)(l+n+1)}\delta_{m,n+1}=-\sqrt{(l-m+1)(l+m)}\delta_{m-1,n};\\\label{y}
&\sigma_Y(l)_{m,n}=-\sqrt{(l+n)(l-n+1)}\delta_{m,n-1}=-\sqrt{(l+m+1)(l-m)}\delta_{m+1,n}.
\end{align}
Here we use the customary notation for $SU(2)$, coming from the spin structure, to work with representations $t^l\in\C^{(2l+1)\times (2l+1)}$, $l\in \frac12\N_0$ being half-integers, with components $t^l_{m,n}$, with indices $-l\leq m,n\leq l$ running from $-l$ to $l$ spaced by an integer. Here $\delta_{m,n}$ denotes Kronecker's delta.

Furthermore, we recall also the symbol of the sub-Laplacian \eqref{EQ:slp} given by the diagonal matrix whose general entry is
\begin{equation}\label{symbolslp}
\sigma_\slp(l)_{m,n}=(l(l+1)-m^2)\delta_{m,n}.
\end{equation}
In this setting, we have the following implication:

\begin{proposition}\label{PROP:SU2}
	Let $G=SU(2)$ and let $\Phi\in\mathcal{C}^{\infty}(G)$. Suppose there exist positive constants $C,A$ such that for any integer $k\in\N$ we have the inequality
	\begin{equation}\notag
	\|\slp^k\Phi\|_{L^2}\leq CA^{2k}((2k)!)^s.
	\end{equation} 
	Then for any multi-index $\alpha$ it follows that
	\begin{equation}\notag
	\|\partial^\alpha\Phi\|_{L^2}\leq CA^{|\alpha|}(\alpha!)^s.
	\end{equation}
\end{proposition}

\begin{remark}\label{REM:matrices}
	In order to prove Proposition \ref{PROP:SU2}, we exploit special behaviour of norms of linear operators which have zeroes everywhere except for the `first' upper or lower diagonal. 
	Indeed, given a linear operator $A:V\rightarrow W$ between two finite-dimensional normed vector spaces, the \emph{operator norm}, $\|A\|_{\OP}:=\sup_{\|v\|\leq1}\|Av\|$, and the \emph{maximum norm}, $\|A\|_{\infty}:=\max_{i,j}|A_{i,j}|$, are equivalent, in the sense there exist two positive constants $C_1$ and $C_2$, {depending on the dimension of $V$}, such that
	\begin{align}\label{norm}
	C_1\|A\|_\infty\leq \|A\|_{\OP}\leq C_2 \|A\|_\infty.
	\end{align}
	In our case, we deal with `special' matrices whose entries $a_{i,j}$ are not zero only when $i-j=-1$ or $i-j=1$. We can show in this case, the constant $C_2$ in \eqref{norm} can be taken to be 1, so that it is independent of the dimension of the matrix, and thus  the argument in the following proof will be justified. Indeed, if we consider a matrix whose entries $a_{i,j}$ are not zero only when $i-j=1$, i.e., only the first lower diagonal is not null, we have
	\begin{align*}
	\|A\|^2_{\OP}&=\sup_{\|v\|\leq1}\sum_i |A^{(i)}v|^2=\sup_{\|v\|\leq1}\sum_i\Big{|}\sum_j a_{i,j}v_j\Big{|}^2=\sup_{\|v\|\leq1}\sum_i |a_{i,i-1}v_{i-1}|^2\leq\\
	&\leq \max_{i} |a_{i,i-1}|^2\sup_{\|v\|\leq1}\sum_i |v_{i-1}|^2\leq \|A\|^2_{\infty}\sup_{\|v\|\leq1}\|v\|^2\leq \|A\|^2_{\infty},
	\end{align*}
	where $A^{(i)}$ denotes the $i^{th}$ row of $A$.
	Therefore, $\|A\|_{\OP}\leq \|A\|_{\infty}.$ Such features of `sparse' symbols appearing in the analysis on $SU(2)$ have been also discussed in detail in \cite[Section 12.6]{RT2009}, allowing 
	one to work with matrices of size going to infinity in the same way one does with finite dimension.
\end{remark}

\begin{proof}[Proof of Proposition \ref{PROP:SU2}]
	Given any multi-index $\alpha$, we want to estimate the $L^2$-norm of $\partial^\alpha f$. In order to use our hypothesis we consider $\partial^\alpha=\partial^\alpha\slp^{-\frac{|\alpha|}{2}}\slp^{\frac{|\alpha|}{2}}$. Then, norm properties and hypotheses yield
	\begin{align*}
	\|\partial^\alpha f\|_{L^2}&\leq\|\partial^\alpha\slp^{-\frac{|\alpha|}{2}}\|_{\OP}\|\slp^{\frac{|\alpha|}{2}}f\|_{L^2}\leq\|\partial^\alpha\slp^{-\frac{|\alpha|}{2}}\|_{\OP}CA^{|\alpha|}(|\alpha|!)^s\\
	&\leq\|\partial^\alpha\slp^{-\frac{|\alpha|}{2}}\|_{\OP}C(2^sA)^{|\alpha|}(\alpha!)^s,
	\end{align*}
	as $|t|!\leq n^{|t|}t!$, for $t=(t_1,\dots, t_n)$. Here $\|\cdot\|_{\OP}=\|\cdot\|_{L^2_0\rightarrow L^2}$.
	
	This means that the claim can be reformulated\footnote{We are following the same argument as in Remark \ref{REM:oppImplication}.} as claiming the boundedness of the operator $\partial^\alpha\slp^{-\frac{|\alpha|}{2}}$ for any $\alpha\in\N^r$. Recalling that $\partial^\alpha=X_1\dots X_{|\alpha|}$, with the same trick used before, we get
	\begin{align}\notag
	\|\partial^\alpha\slp^{-\frac{|\alpha|}{2}}\|_{\OP}&=\|X_1\slp^{-\frac{1}{2}}\slp^\frac{1}{2}X_2\slp^{-\frac{1}{2}}\dots\slp^\frac{1}{2}X_{|\alpha|}\slp^\frac{|\alpha|}{2}\|_{\OP}\\\notag
	&\leq\|X_1\slp^{-\frac{1}{2}}\|_{\OP}\|\slp^\frac{1}{2}X_2\slp^{-\frac{1}{2}}\|_{\OP}\dots\|\slp^\frac{1}{2}X_{|\alpha|-1}\slp^{-\frac{1}{2}}\|_{\OP}\|\slp^\frac{1}{2}X_{|\alpha|}\slp^{-\frac{|\alpha|}{2}}\|_{\OP}.
	\end{align} 
	Thus, the study of the boundedness can be split into the analysis of three types of operators:
	\begin{enumerate}[1.]
		\item $X_j\slp^{-\frac{1}{2}}$,
		\item $\slp^{\frac{1}{2}}X_j\slp^{-\frac{1}{2}}$,
		\item $\slp^{\frac{1}{2}}X_j\slp^{-\frac{|\alpha|}{2}}$,
	\end{enumerate} 
	where $X_j\in\{X,Y\}$. Without loss of generality, we may assume that $X_j=X$ (the argument can be repeated analogously replacing $X$ by $Y$). Let us proceed estimating the operator norm of each operator individually.
	
	1. Considering the left-invariance of the vector fields we are dealing with, we have 
	\[
	\|X\slp^{-\frac{1}{2}}\|_{\OP}=\sup_{l\in\frac{1}{2}\N_0}\|\sigma_{X}(l)\sigma_{\slp^{-\frac{1}{2}}}(l)\|_{\infty}.
	\]  
	The explicit expressions \eqref{x}, \eqref{y} and \eqref{symbolslp} allow us to determine and estimate from above the general element of the product matrix $\sigma_{X}(l)\sigma_{\slp^{-\frac{1}{2}}}(l)$, that is
	\begin{align}\notag
	\big|\big(\sigma_{X}(l)\sigma_{\slp^{-\frac{1}{2}}}(l)\big)_{m,n}\big|&=\sum_{k=-l}^l \frac{\sqrt{(l-k)(l+k+1)}\delta_{m,k+1}}{\sqrt{l(l+1)-k^2}\delta_{k,n}}\leq\\\notag
	&\leq\frac{\sqrt{(l-n)(l+n+1)}}{\sqrt{(l-n)(l+n)}}\delta_{m,n+1}=\sqrt{\frac{l+n+1}{l+n}}\delta_{m,n+1}.
	\end{align} 
	Combining this inequality with Remark \ref{REM:matrices}, we get
	\[
	\|\sigma_X(l)\sigma_{\slp^{-\frac{1}{2}}}(l)\|_{\infty}\leq \max_{-l+1\leq k\leq l}\sqrt{\frac{l+k+1}{l+k}}=2.
	\]
	It follows that we can estimate uniformly $\|\sigma_{X}(l)\sigma_{\slp^{-\frac{1}{2}}}(l)\|_{\infty}\leq C_1$.
	
	2. Using similar considerations, we can now estimate the general element of the symbol matrix associated to the operator $\slp^\frac{1}{2}X\slp^{-\frac{1}{2}}$. Hence, we calculate
	\begin{align}\notag 
	\big|\big(\sigma_{\slp^\frac{1}{2}}(l)\sigma_X(l)\sigma_{\slp^{-\frac{1}{2}}}(l)\big)_{m,n}\big|&=\sum_{k=-l}^l \frac{\sqrt{l(l+1)-(k+1)^2}\sqrt{(l-k)(l+k+1)}\delta_{m,k+1}}{\sqrt{l(l+1)-k^2}\delta_{k,n}}\\\notag
	&\leq\frac{\sqrt{l(l+1)-n^2}\sqrt{(l-n)(l+n+1)}}{\sqrt{l(l+1)-n^2}}\delta_{m,n+1}.
	\end{align} 
	Then we obtain in this case that 
	\[
	\|\sigma_{\slp^\frac{1}{2}}(l)\sigma_X(l)\sigma_{\slp^{-\frac{1}{2}}}(l)\|_{\infty}\leq\max_{-l\leq k\leq l-1}\sqrt{(l-k)(l+k+1)}\leq C_2 l.
	\]
	
	3. Finally, we focus on the last type of operator. We estimate
	\begin{align}\notag
	\big|\big(\sigma_{\slp^\frac{1}{2}}(l)\sigma_{X}(l)\sigma_{\slp^{-\frac{|\alpha|}{2}}}(l) \big)_{m,n}\big|&=\frac{\sqrt{l(l+1)-(n+1)^2}\sqrt{(l-n)(l+n+1)}}{\sqrt{(l(l+1)-n^2)^{|\alpha-1|}}}\delta_{m,n+1}\\\notag
	&\leq\frac{\sqrt{(l-n)(l+n+1)}}{\sqrt{(l(l+1)-n^2)^{|\alpha|-1}}}\delta_{m,n+1}.
	\end{align}
	Passing to the norm, we get
	\begin{align}\notag
	\|\sigma_{\slp^\frac{1}{2}}(l)\sigma_{X}(l)\sigma_{\slp^{-\frac{|\alpha|}{2}}}(l)\|_{\infty}&\leq\max_{-l\leq k\leq l-1}\frac{\sqrt{(l-k)(l+k+1)}}{\sqrt{(l(l+1)-k^2)^{|\alpha|-1}}}\\\notag
	&\leq \max_{-l\leq k\leq l-1}\frac{\sqrt{l+k+1}}{\sqrt{(l-k)^{|\alpha|-2}(l+k)^{|\alpha|-1}}}\sim l^{2-|\alpha|},
	\end{align}
	which gives
	\[
	\|\sigma_{\slp^\frac{1}{2}}(l)\sigma_{X}(l)\sigma_{\slp^{-\frac{|\alpha|}{2}}}(l)\|_{\infty}\leq C_3 l^{2-|\alpha|}.
	\]
	
	Combining together all the estimates above, we deduce the boundedness of the operator. Indeed, we have
	\begin{align*}
	&\hspace{2.5cm}\|\partial^{\alpha}\slp^{-\frac{|\alpha|}{2}}\|_{\OP}=\sup_{l\in\frac{1}{2}\N_0}\|\sigma_{\partial^{\alpha}\slp^{-\frac{|\alpha|}{2}}}(l)\|_{\infty}\leq\\
	\leq\sup_{l\in\frac{1}{2}\N_0}&\Big(\|\sigma_{X_1}(l)\sigma_{\slp^{-\frac{1}{2}}}(l)\|_{\infty}
	\underbrace{\|\sigma_{\slp^\frac{1}{2}}(l)\sigma_{X_2}(l)\sigma_{\slp^{-\frac{1}{2}}}(l)\|_{\infty}\dots\|\sigma_{\slp^\frac{1}{2}}(l)\sigma_{X_{|\alpha|-1}}(l)\sigma_{\slp^{-\frac{1}{2}}}(l)\|_{\infty}}_{|\alpha|-2\text{ terms}}\times\\
	&\times \|\sigma_{\slp^\frac{1}{2}}(l)\sigma_{X_{|\alpha|}}(l)\sigma_{\slp^{-\frac{|\alpha|}{2}}}(l)\|_{\infty}\Big )\leq C \sup_{l\in\frac{1}{2}\N_0}l^{|\alpha|-2}l^{2-|\alpha|}<C<\infty,
	\end{align*}  
	using in the last inequality results from points $1$, $2$ and $3$ above.
	This completes the proof.
\end{proof}

The result obtained in Proposition \ref{PROP:SU2} allows one to characterise globally the sub-Laplacian Gevrey spaces on $SU(2)$ on the Fourier transform side as desired. 
Also, since $SU(2)$ is compact, we have the equality of spaces $\gamma^s_{\mathbf{X},L^{\infty}}$ and $\gamma^s_{\mathbf{X},L^2}$, with $\mathbf{X}=\{X,Y\}$.
Summarising, we can state the following equivalence:
\begin{cor} \label{COR:su2}
	We have $$\gamma^s_{\mathbf{X},L^{\infty}}(\SU)=\gamma^s_{\mathbf{X},L^2}(\SU),$$ and, moreover,
	\begin{align*}
	\gamma^s_{\mathbf{X},L^{\infty}}(\SU)=\big\{\phi\in\mathcal C^\infty(\SU)\,|\, \exists \,D>0\text{ such that }\|e^{D\slp^{\frac{1}{2s}}}\phi\|_{L^2(\SU)}<\infty\big\}.
	\end{align*}
\end{cor}

\section{Sub-Laplacian Gevrey spaces on the Heisenberg Group}
\label{SEC:heis}

In Chapter \ref{CP:2} we have seen that the Heisenberg group $\mathbb H_n$ is the first example of a non-abelian, non-compact, locally compact, nilpotent, unimodular, stratified Lie group. There is a substantial amount of literature about it and we recall here few titles, such as \cite{FH1987}, \cite{FR2016}, \cite{Fol} and \cite{T1998}.

Throughout this Section we will look at the Heisenberg group as the manifold $\R^{2n+1}$ endowed with the group law
\[
(x,y,t)(x',y',t'):=(x+x',y+y',t+t'+\frac{1}{2}(xy'-x'y)),
\]
where $(x,y,t),\, (x',y',t')\in\R^n\times\R^n\times\R\sim\h_n$. We consider the canonical basis for the {Heisenberg Lie algebra} $\mathfrak{h_n}$ associated with the Heisenberg group, given by
\begin{align*}
&X_j=\partial_{x_j}-\frac{y_j}{2}\partial_t\quad\text{and}\quad Y_j=\partial_{y_j}+\frac{x_j}{2}\partial_t,\quad\text{for } j\in\{1,\dots,n\},\\
&T=\partial_t.
\end{align*}
These vector fields satisfy the canonical commutation relations 
\[
[X_j,Y_j]=T \quad\text{for every } j\in\{1,\dots,n\},
\] 
with all other possible combinations being zero. This also implies that the set of vector fields given by $\{X_j,Y_j\}_{j=1,\dots,n}$ is a H\"ormander system, so that the canonical sub-Laplacian defined by
\begin{equation*}
\slp:=-\sum_{j=1}^n(X^2_j+Y_j^2)
\end{equation*}
is hypoelliptic.
For each $\lambda\in\R\setminus\{0\}$, the corresponding Schr\"odinger representation 
\[
\pi_\lambda:\mathbb{H}_n\rightarrow\mathcal{U}(L^2(\Rn)),
\]
is a unitary irreducible representation given by 
\begin{align}\notag
\pi_\lambda(x,y,t)\phi(u)=[\pi_1(\sqrt\lambda x,\sqrt\lambda y,\lambda t)](u)=e^{i\lambda(t+\frac{1}{2}xy)}e^{i\sqrt\lambda yu}\phi(u+\sqrt{|\lambda|}x).
\end{align}
In the above definition we use the following convention from \cite{FR2016}:
\begin{align*}
\sqrt\lambda:={\rm sgn}(\lambda)\sqrt{|\lambda|}=
\begin{cases}
\sqrt\lambda&\text{if }\lambda>0,\\
-\sqrt{|\lambda|}&\text{if }\lambda<0.
\end{cases}
\end{align*}

We move now to the infinitesimal representations associated to the Schr\"odinger representations. We have seen in Section \ref{SEC:quantization} that they play a crucial r\^ole in determining the symbols of left-invariant differential operators. Considering the aforementioned canonical basis of $\mathfrak h_n$, for every $\lambda \in\R\setminus\{0\}$ the corresponding infinitesimal representations of the elements of the basis are given by
\begin{subequations}
	\label{infrep}
	\begin{align}
	&\pi_\lambda(X_j)=\sqrt{|\lambda|}\partial_{x_j}\quad&\text{ for }j=1,\dots,n;\label{eq1}\\
	&\pi_\lambda(Y_j)=i\sqrt\lambda x_j \quad&\text{ for }j=1,\dots,n;\label{eq2}\\
	&\pi_\lambda(T)=i\lambda I.\label{eq3}
	\end{align}
\end{subequations}
We recall that for every $\lambda\in\Rn\setminus\{0\}$ the space of all smooth vectors $\mathcal H^\infty_{\pi_\lambda}$ is the Schwartz space $\mathcal{S}(\Rn)$. An easy calculation yields that the infinitesimal representation of the sub-Laplacian $\slp$ is given by 
\begin{align}\label{irslp}
\pi_\lambda(\slp)=|\lambda|\sum_{j=1}^n(x_j^2-\partial_{x_j}^2),
\end{align}
which is clearly related to the harmonic oscillator 
\begin{align*}
H=-\Delta + |x|^2. 
\end{align*}
\subsection{Hermite polynomials and matrix representation}

The aim of this Subsection is to obtain a matrix representation of the operators 
\eqref{infrep} and \eqref{irslp}. This will be useful when we will consider Gevrey spaces on the Heisenberg group, as we will see in the next Subsection.

To simplify the notation, we will work with the three-dimensional Heisenberg group $\mathbb{H}_1$, i.e. $n=1$. The extension to any $n$ is straightforward. It is well known that the \textit{Hermite polynomials}, once normalised, form an orthonormal basis of $L^2(\R)$ consisting of eigenfunctions of $\pi_\lambda(\slp)$. Here we will fix the notation and recall some properties of these polynomials, see e.g. \cite{S1975} for more details. 

For every $k\in\N$ and $x\in\R$ the ($k$-th)-Hermite polynomial is given by
\[
H_k(x):=(-1)^ke^{x^2}\Big(\frac{d^k}{dx^k}e^{-x^2}\Big).
\]
We see that
\[
\int_\R{e^{-x^2}H_k(x)H_m(x)}dx=\pi^{\frac{1}{2}}2^{k}k!\delta_{km},
\]
thus the {normalised Hermite functions} are defined by
\begin{align}\label{h}
h_k(x):=\frac{1}{\sqrt{\sqrt\pi 2^kk!}}e^{-\frac{x^2}{2}}H_k(x)=c_ke^{-\frac{x^2}{2}}H_k(x),
\end{align}
where $c_k:=\frac{1}{\sqrt{\sqrt\pi 2^kk!}}$. As mentioned above, the normalised Hermite functions $\{h_k(\cdot)\}_{k\in\N_0}$ form a basis of $L^2(\R)$.\footnote{In \cite{Fol} and \cite{J2014} two different proofs can be found. The former relies on complex function theory and the latter on real function descriptions of the fundamental theorem of calculus and the Schwarz inequality.}
\begin{remark}
	For the sake of completeness we can observe that the higher dimensional Hermite functions are products of one-dimensional Hermite functions, namely for every multi-index $\alpha=(\alpha_1,\dots,\alpha_n)$ and $x\in\Rn$ we have $h_\alpha(x)=h_{\alpha_1}(x_1)\dots h_{\alpha_n}(x_n)$. The family $\{h_{\alpha}\}_{\alpha\in\N^n_0}$ provides an orthonormal basis for $L^2(\Rn)$.
\end{remark}

We now calculate the matrices corresponding to the infinitesimal representations of the elements of the fixed canonical basis of $\mathbb H_1$ and of the sub-Laplacian. In order to do this, we recall from \cite{S1975} useful properties of the Hermite functions: for all $k\in\N, \,k\geq2$ and $x\in\R$ we have
\begin{subequations}
	\label{prop}
	\begin{align}
	&H_k(x)=2 x H_{k-1}(x)-2(k-1)H_{k-2}(x),\label{prop1}\\
	&H'_k(x)=2kH_{k-1}(x)\label{prop2}.
	\end{align}
\end{subequations}
Equations \eqref{eq2} and \eqref{prop1} imply the following equality
\[
\pi_\lambda (Y)H_k(x)=i\sqrt\lambda x H_k(x)=i\sqrt\lambda\Big(\frac{1}{2}H_{k+1}(x)+kH_{k-1}(x)\Big).
\]
Multiplying both sides by $c_ke^{-\frac{x^2}{2}}$ and looking at the definition of \eqref{h}, we immediately obtain
\begin{equation}\label{mol}
\pi_\lambda(Y)h_k(x)=i\sqrt{\lambda}\Bigg(\sqrt{\frac{k+1}{2}}h_{k+1}(x)+\sqrt\frac{k}{2}h_{k-1}(x)\Bigg).
\end{equation}
Now we evaluate \eqref{eq1} at $h_k(x)$. First of all we observe that 
\begin{align*}
\pi_{\lambda}(X)h_k(x)&=\sqrt{|\lambda|}h_k'(x)=\sqrt{|\lambda|}\big(-xh_k(x)+c_ke^{-\frac{x^2}{2}}H'_k(x)\big)=\\
&=-\pi(Y)h_k(x)+\sqrt{|\lambda|}\big(c_ke^{-\frac{x^2}{2}}H'_k(x)\big).
\end{align*}
Thus, the property of the first derivative of the Hermite polynomials  \eqref{prop2} yields
\begin{align}\notag
\pi_{\lambda}(X)h_k(x)&=\sqrt{|\lambda|}\Bigg(-\sqrt{\frac{k+1}{2}}h_{k+1}(x)-\sqrt\frac{k}{2}h_{k-1}(x)+c_ke^{-\frac{x^2}{2}}2kH_{k-1}(x)\Bigg)=\\
&=-\sqrt{|\lambda|}\sqrt\frac{k+1}{2}h_{k+1}(x)+\sqrt{|\lambda|}\sqrt{\frac{k}{2}}h_{k-1}(x).\label{EQ:der}
\end{align}
Therefore, we deduce that the matrices associated to the infinitesimal representations of the left-invariant vector fields $X$ and $Y$ have all entries null except for the elements in the first upper and lower diagonal.

Now, we calculate the matrix associated to the infinitesimal representations of the sub-Laplacian. From \eqref{prop1} and \eqref{EQ:der} it follows that
\begin{align}\notag
h''_k(x)&=(h_k'(x))'=\Big(-\sqrt\frac{k+1}{2}h_{k+1}(x)+\sqrt{\frac{k}{2}}h_{k-1}(x)\Big)'=\\
\notag
&=\frac{\sqrt{(k+1)(k+2)}}{2}h_{k+2}(x)-\frac{2k+1}{2}h_k(x)=\frac{\sqrt{k(k-1)}}{2}h_{k-2}(x),\end{align}
and from \eqref{prop2} and \eqref{mol} we obtain
\begin{align}\notag
x^2h_k(x)&=x\Bigg(\sqrt{\frac{k+1}{2}}h_{k+1}(x)+\sqrt\frac{k}{2}h_{k-1}(x)\Bigg)=\\\notag
&=\frac{\sqrt{(k+1)(k+2)}}{2}h_{k+2}(x)+\frac{2k+1}{2}h_k(x)+\frac{\sqrt{k(k-1)}}{2}h_{k-2}(x).
\end{align}
Combining together the calculations above and the expression of the infinitesimal representation of the sub-Laplacian given by \eqref{irslp}, we finally obtain
\begin{align}
\pi_{\lambda}(\slp)h_k(x)=|\lambda|(2k+1)h_k(x).
\end{align}
We use the same notation $\pi_\lambda(X)$, $\pi_\lambda(Y)$ and $\pi_\lambda(\slp)$ to denote both the operators and the infinite matrices associated to our vector fields with respect to the orthonormal basis comprising the Hermite functions $\{h_k\}_{k\in\N}$. Then for all $k,l\in\N$ the $(k,l)$-entries of these matrices are given by
\begin{align}\label{matrix1}
&\big(\pi_\lambda(\slp)\big)_{k,l}=|\lambda|(2k+1)\delta_{k,l},\\\label{matrix2}
&\big(\pi_\lambda(X)\big)_{k,l}=
\begin{cases}
\sqrt{|\lambda|}\sqrt{\frac{k+1}{2}}&\text{if }k=l-1\\
-\sqrt{|\lambda|}\sqrt{\frac{k}{2}}&\text{if }k=l+1\\
0&\text{otherwise}
\end{cases},\\\label{matrix3}
&\big(\pi_\lambda(Y)\big)_{k,l}=
\begin{cases}
i\sqrt{\lambda}\sqrt{\frac{k+1}{2}}&\text{if }k=l-1\\
i\sqrt{\lambda}\sqrt{\frac{k}{2}}&\text{if }k=l+1\\
0&\text{otherwise}
\end{cases}.
\end{align}

\subsection{Characterisation of sub-Laplacian  Gevrey spaces on $\h_n$}\label{Hn}

Theorem \ref{THM:implication} shows that considered on a general manifold $M$, the sub-Laplacian Gevrey spaces $\gamma^s_{\mathbf{X},L^2}(M)$ are included in bigger classes of functions:
\begin{align*}
\gamma^s_{\mathbf{X},L^2}(M)\subset\big\{\phi\in\mathcal C^\infty(M)\,:\,\exists D>0\text{ such that }\|e^{D\slp^{\frac{1}{2s}}}\phi\|_{L^2}<\infty\big\}.
\end{align*}
In this Subsection, we consider the manifold to be the Heisenberg group $\mathbb H_n$ so that, because of the explicit symbolic calculus recalled in the previous Subsection, we can prove also the opposite implication. Therefore, we obtain a \textit{global characterisation} of the sub-Laplacian Gevrey spaces on $\mathbb H_n$ on the Fourier transform side.

\begin{theorem}[Characterisation of sub-Laplacian Gevrey spaces on $\mathbb H_n$]
	\label{THM:Heis}
	The following statements are equivalent:
	\begin{enumerate}[i.]
		\item $f\in \gamma^s_{\mathbf{X},L^2}(\mathbb H_n);$
		\item $f\in\mathcal C^\infty (\mathbb H_n)$ and there exists $D>0$ such that $\|e^{D\slp^{\frac{1}{2s}}} f\|_{L^2(\mathbb H_n)}<\infty$.
	\end{enumerate}
\end{theorem}

In order to prove this characterisation, we will need the following result:

\begin{proposition}\label{PROP:Heis}
	Let $f \in\mathcal{C}^{\infty}(\h_n)$. Suppose that there exist positive constants $C,A>0$ such that for any integer $k\in\N$, we have the inequality 
	\begin{equation}\notag
	\|\slp^k f\|_{L^2}\leq CA^{k}((2k)!)^s.
	\end{equation} 
	Then for any multi-index $\alpha$, it follows that
	\begin{equation}\notag
	\|\partial^\alpha f\|_{L^2}\leq CA^{|\alpha|}(\alpha!)^s.
	\end{equation}
\end{proposition}
\begin{remark}
	Since the Heisenberg group is non-compact, we are allowed to work with the inverse of the sub-Laplacian on $L^2$. This is due to the fact that constant functions are not square integrable on a non-compact group. Therefore, the sub-Laplacian $\slp$ is injective on $L^2(\h_n)$, in the sense that if $f\in\Dom_{L^2}(\slp)$ and $\slp f=0$, then $f\equiv 0$.
\end{remark}
\begin{proof}[Proof of Proposition \ref{PROP:Heis}]
	This argument mimics that for $SU(2)$. Given any multi-index $\alpha$, to estimate the $L^2$-norm of $\partial^\alpha f$, we write the differentiation as $\partial^\alpha=\partial^\alpha\slp^{-\frac{|\alpha|}{2}}\slp^{\frac{|\alpha|}{2}}$. Then norm properties and hypotheses yield
	\begin{align}\notag
	\|\partial^\alpha f\|_{L^2}&\leq\|\partial^\alpha\slp^{-\frac{|\alpha|}{2}}\|_{\OP}\|\slp^{\frac{|\alpha|}{2}}f\|_{L^2}\leq\|\partial^\alpha\slp^{-\frac{|\alpha|}{2}}\|_{\OP}CA^{|\alpha|}(|\alpha|!)^s\\
	&\leq\|\partial^\alpha\slp^{-\frac{|\alpha|}{2}}\|_{\OP}C(2^sA)^{|\alpha|}(\alpha!)^s,\label{alpha}
	\end{align}
	as $|t|!\leq n^{|t|}t!$, for $t=(t_1,\dots, t_n)$.
	
	This means that the conclusion can be reformulated as proving the boundedness, independently of or ``nicely'' dependent on $\alpha$ as before, of the operator $\partial^\alpha\slp^{-\frac{|\alpha|}{2}}$ for any $\alpha\in\N^{2n}$. Since $\partial^\alpha=X_1\dots X_{|\alpha|}$, to prove this boundedness we will make use of the explicit symbolic calculus of $\mathbb{H}_n$. In particular we look at the matrices associated with the infinitesimal representations of the operators involved. 
	
	To simplify the notation, once again, we will work with the three-dimensional case, i.e., $n=1$.  Formulae \eqref{matrix2} and \eqref{matrix3} provide explicit expressions for the entries of the matrices associated with the vector fields of the elements of the canonical basis of $\mathfrak h_1$. A quick glance at these matrices suggests we should consider an `equivalent' basis for $\mathfrak h_1$ whose associated matrix representations have all null entries except on one (upper or lower) diagonal. Thus we define
	\begin{align*}
	&Z:=X+iY,\\
	&\overline Z:=X-iY,\\
	&T=\frac{i}{2}[Z,\overline Z],
	\end{align*}
	and we can check that the space $\gamma^s_{\mathbf{X},L^\infty}(\h_1)$ of functions obtained considering the initial vector fields $\{X,Y,T\}$ is the same as the one obtained taking into account the elements of the complex basis $\{Z,\overline Z,T\}$. More precisely,
	\begin{multline*}
	\Big\{ f\in \mathcal C^\infty(\h_1)\,|\, \forall \alpha\in\N^2,\|\partial^\alpha f\|_{L^\infty}\leq CA^{|\alpha|}(\alpha!)^s,\partial^\alpha=X_1\dots X_{|\alpha|}, X_j\in\{X,Y\}\Big\}
	\\ =
	\Big\{f\in \mathcal C^\infty(\h_1)\,|\, \forall \alpha\in\N^2,\|\partial^\alpha f\|_{L^\infty}\leq CA^{|\alpha|}(\alpha!)^s,\partial^\alpha=Z_1\dots Z_{|\alpha|}, Z_j\in \big\{ Z,\overline Z \big\}\Big\}.
	\end{multline*}
	Therefore, we can again reformulate the conclusion as proving the boundedness of the operator $\partial^\alpha\slp^{-\frac{|\alpha|}{2}}$ where $\partial^\alpha=Z_1\dots Z_{|\alpha|}$ with every $Z_j\in\{Z,\overline Z\}$.
	Without loss of generality we can restrict to the case $\lambda>0$. We calculate the entries of the matrices associated with the infinitesimal representations of $Z$ and $\overline Z$, obtaining
	\begin{align}\label{mat1}
	&\big(\pi_\lambda(Z)\big)_{k,l}=
	\begin{cases}
	-\sqrt{\lambda}\sqrt{\frac{k}{2}}&\text{if }k=l+1\\
	0&\text{otherwise}
	\end{cases},\\
	\label{mat2}&\big(\pi_\lambda(\overline Z)\big)_{k,l}=
	\begin{cases}
	2\sqrt{\lambda}\sqrt{\frac{k+1}{2}}&\text{if }k=l-1\\
	0&\text{otherwise}
	\end{cases}.
	\end{align}
	Note that the sub-Laplacian is now given by $$\slp=\frac{Z^2 +\overline Z^2}{2}.$$
	We recall that for a left-invariant operator $A$ on a Lie group $G$ we have formally
	\[
	\|A\|_{\OP}=\sup_{\xi\in\widehat G}\|\sigma_A(\xi)\|_{\OP},
	\]
	where $\sigma_A(\xi)$ is the symbol associated with $A$ and $\|A\|_{\OP}$ is the operator norm of $A$ in $L^2(G)$.
	
	We observe that since all the vector fields are left-invariant, for all $\lambda\in\R\setminus \{0\}$ we have
	\[
	\pi_{\partial^\alpha\slp^{-\frac{|\alpha|}{2}}}(\lambda)=\pi_{Z_1\dots Z_{|\alpha|}\slp^{-\frac{|\alpha|}{2}}}(\lambda)=\pi_{Z_1}(\lambda)\dots\pi_{Z_{|\alpha|}}(\lambda)\pi_{\slp^{-\frac{|\alpha|}{2}}}(\lambda).
	\]
	Therefore, we can evaluate the non-zero entries of the matrix product above in order to estimate the norm of our operator.
	
	Summarising, according to \eqref{mat1} and \eqref{mat2}, we have the following:
	\begin{enumerate}[1.]
		\item $\pi_{Z_j}(\lambda)$ (or $\pi_{\overline Z_j}(\lambda)$) is a matrix whose entries equal zero except on the first lower (or upper) diagonal. Precisely, they are of the form $c{\big(\lambda m\big)^\frac{1}{2}}$, where $c$ is a fixed constant;
		\item $\pi_{\slp^{-\frac{|\alpha|}{2}}}(\lambda)$ is a diagonal matrix whose diagonal entries are of the form ${c' \big(\lambda m\big)^\frac{-|\alpha|}{2}}$, where $c'$ is a fixed constant.
	\end{enumerate}
	If we look at the product matrix
	\[
	\underbrace{\pi_{Z_1}(\lambda)\dots\pi_{Z_{|\alpha|}}(\lambda)}_{|\alpha|\text{ times}}
	\]
	we can observe that each product will produce a matrix with a `travelling' lower (or upper) non-zero diagonal. Once all the products have been accomplished, the non-zero entries, placed all on one lower (or upper) diagonal,  will be of the form $ c^{|\alpha|}\big(\lambda m\big)^{\frac{|\alpha|}{2}}$. Therefore, the non-zero entries placed all on one lower (or upper) diagonal of the final matrix
	\[
	\pi_{Z_1}(\lambda)\dots\pi_{Z_{|\alpha|}}(\lambda)\pi_{\slp^{-\frac{|\alpha|}{2}}}(\lambda)
	\]
	will be of the form $c^{|\alpha|}\big(\lambda m\big)^\frac{|\alpha|}{2}\big(\lambda m\big)^\frac{-|\alpha|}{2}\sim c^{|\alpha|}$. This yields
	\begin{align*}
	\|\partial^\alpha\slp^{-\frac{|\alpha|}{2}}\|_{\OP}&=\sup_{\lambda\in\R\setminus \{0\}}\|\pi_{\partial^\alpha\slp^{-\frac{|\alpha|}{2}}}(\lambda)\|_{\OP_{\lambda}}=\sup_{\lambda\in\R\setminus \{0\}}\max_{m,k\geq 1}\Big|\big(\pi_{\partial^\alpha\slp^{-\frac{|\alpha|}{2}}}(\lambda)_{m,k}\big)\Big|\\
	&\leq c^{|\alpha|}.
	\end{align*} 
	Combining this result with \eqref{alpha}, we finally obtain
	\begin{align*}
	\|\partial^\alpha f\|_{L^2}\leq \|\partial^\alpha\slp^{-\frac{|\alpha|}{2}}\|_{\OP}C(2^sA)^{|\alpha|}(\alpha!)^s\leq C(2^scA)^{|\alpha|}(\alpha!)^s,
	\end{align*}
	completing the proof of Proposition \ref{PROP:Heis}.
\end{proof}

\begin{proof}[Proof of Theorem \ref{THM:Heis}]
	The first inclusion follows from Theorem \ref{THM:implication}.
	To prove the opposite inclusion we recall that by Proposition \ref{PROP:prop2} the following statements are equivalent:
	\begin{enumerate}[i.]
		\item there exist $A,C>0$ such that for every integer $k\in\N_0$ it holds $\|\slp^k \phi\|_{L^2}\leq C A^{2k}((2k)!)^s$;
		\item there exists a constant $D>0$ such that $\|e^{D\slp^{\frac{1}{2s}}}\phi\|_{L^2(\mathbb H_n)}<\infty$.
	\end{enumerate}
	Moreover, Proposition \ref{PROP:prop1} states that in the definition of sub-Laplacian Gevrey spaces we can replace the $L^\infty$-norm with the $L^2$-norm. The application of Proposition \ref{PROP:Heis} completes the proof.
\end{proof}

\section{Further remarks and conjectures}

In this Section we collect some thoughts about directions this work could be continued. We start by mentioning a more restricted class than that considered until now in this Chapter.

The Gevrey spaces we have considered so far are also known as {\em Gevrey--Roumieu} spaces. However, there is another class known as the {\em Gevrey--Beurling} spaces, which impose stricter conditions on the constants and are subspaces of the Gevrey--Roumieu spaces. We can state the adaptation,  to the sub-Laplacians, of the definition of this second class.

\begin{definition}[$\slp$-Gevrey--Beurling spaces, $L^\infty$]\label{DEF:gsl-b}
	Let $s>0$. Let $M$ be a manifold and $\mathbf{X}=\{X_1,\dots,X_r\}$ be a H\"ormander system of vector fields. The sub-Laplacian Gevrey--Beurling space $\gamma^{(s)}_{\mathbf{X},L^{\infty}}(M)$ of order $s$ on a manifold $M$ is the space of all functions $\phi\in\mathcal{C}^\infty(M)$ for which for every compact set $K\subset M$ and for any constant $A>0$ there exists $C_A>0$ such that for every $\alpha\in\N^r_0$ we have
	\begin{equation}\label{gevrey-b}
	|\partial^\alpha \phi(x)|\leq C_A A^{|\alpha|}(\alpha!)^s,\;\textrm{ for all } x\in K,
	\end{equation}
	where $\partial^\alpha=Y_1\dots Y_{|\alpha|}$, with $Y_j\in\{X_1,\dots,X_r\}$ for every $j=1,\dots,|\alpha|$ and $\sum_{Y_j=X_k}1=\alpha_k$ for every $k=1,\dots,r$.
\end{definition}	
It follows from the definition that $\gamma^{(s)}_{\mathbf{X},L^{\infty}}(M)\subset\gamma^{s}_{\mathbf{X},L^{\infty}}(M)\subset\mathcal{C}^{\infty}(M)$. As with the Gevrey--Roumieu spaces, we can define the $L^2$-version as follows:
\begin{definition}[$\slp$-Gevrey--Beurling spaces, $L^2$]
	The $L^2$-version of the sub-Laplacian Gevrey--Beurling space, $\gamma^{(s)}_{\mathbf{X},L^2}(M)$, of order $s$ on $M$ endowed with a measure $\mu$, is the space of all functions $\phi\in \mathcal{C}^\infty(M)$ such that for any $A>0$ there exists $C_A>0$ such that   for every $\alpha\in\N^r_0$ we have
	\begin{equation}\label{gevrey-b-L2}
	\|\partial^\alpha \phi\|_{L^2(M)}\leq C_A A^{|\alpha|}(\alpha!)^s,
	\end{equation}
	where $\partial^\alpha$ are as for \eqref{gevrey-b}.
\end{definition}
Similarly to Proposition \ref{PROP:prop1}, under its assumptions, we have the inclusions
\begin{equation}\label{EQ:inc-b}
\gamma^{(s)}_{\mathbf{X},L^{\infty}}(M)\cap \mathcal{C}_0^\infty(M)\subset \gamma^{(s)}_{\mathbf{X},L^2}(M) \subset \gamma^{(s)}_{\mathbf{X},L^{\infty}}(M).
\end{equation}
\begin{remark}
	These inclusions mean that every result we have proved in this Chapter holds if one replaces the Gevrey--Roumieu space with its Gevrey--Beurling equivalent, so the latter may not seem particularly interesting. However their dual spaces, i.e., the spaces of continuous linear functionals on the respective original Gevrey spaces, satisfy the reverse inclusion. In light of the conjectures we will shortly formulate concerning these duals, this dual inclusion makes the Gevrey--Beurling spaces worthy of investigation in this topic.
\end{remark}

\subsection{Hints on ultradistributions}

As just mentioned, after introducing the sub-Laplacian Gevrey functions, we can also consider the spaces of linear functionals defined on the Gevrey classes. These spaces provide a bigger environment, comparable to that offered by the `classical' distributions, where we might, for example, look for solutions of Cauchy problems. In this Subsection we conjecture theorems basing our claims on previous work on the compact case, e.g., \cite{DR2014}. It is our intention to verify thoroughly the details of the proofs in the near future. 

\begin{definition}[Ultradistributions]
	Let $s\geq 1$ and let $M$ be a manifold. The spaces of continuous linear functionals on $\gamma^{s}_{\mathbf{X}, L^2}(M)$ and $\gamma^{(s)}_{\mathbf{X}, L^2}(M)$ are called the {\em spaces of ultradistributions} and are denoted respectively by
	\begin{align*}
	(\gamma^{s}_{\mathbf{X}, L^2})'(M)\quad\text{and}\quad (\gamma^{(s)}_{\mathbf{X}, L^2})'(M).
	\end{align*}
	Clearly $\mathcal{D}'(M)\subset(\gamma^{s}_{\mathbf{X}, L^2})'(M)\subset(\gamma^{(s)}_{\mathbf{X}, L^2})'(M)$.
\end{definition}

We refer to \cite{DR2014} or \cite{DR2016} for a more detailed discussion of such spaces.
Combining our characterization for Gevrey spaces on $SU(2)$ and on the Heisenberg group with this work from Dasgupta and Ruzhansky on ultradistributions on compact groups and manifolds, we expect to be able to characterise the corresponding spaces of (sub-Laplacian) ultradistributions in the settings of $SU(2)$ and of the Heisenberg group. 

Since we can not talk about pointwise estimates for derivatives in the setting of distributions, it is most effective to aim at characterisations of these spaces on their Fourier transform side. By duality, the Fourier transform has an extension that acts on ultradistributions. 

We follow the notation of Subsection \ref{SEC:su2} and Section \ref{SEC:heis}, where sub-Laplacian Gevrey classes have been described on the Fourier transform side in the case of $SU(2)$ and the Heisenberg group, respectively. 
In particular, looking at the case of $SU(2)$, we recall the symbol of the sub-Laplacian from \eqref{slpSU2}, so that in the theorem below, for every $l\in \frac12\N_0$ being half-integer, $\widehat{\slp}(l)$ is the diagonal matrix in $\C^{(2l+1)\times (2l+1)}$ with entries
\begin{equation}\label{slp-d}
\widehat{\slp}(l)_{m,n}=(l(l+1)-m^2)\delta_{m,n},\; -l\leq m,n\leq l.
\end{equation}

\begin{conj}\label{THM:ultra-su2}
	Let $1\leq s<\infty$. We have $u\in (\gamma^{s}_{\mathbf{X}, L^2}(\SU))'$ if and only if for every $B>0$ there exists $K_B>0$ such that we have
	\begin{equation}\label{EQ:u-su2}
	\|e^{-B \widehat{\slp}(l)^{\frac{1}{2s}}}\widehat{u}(l)\|_{\HS}\leq K_B  \textrm{ for all } l\in\frac12\N_0.
	\end{equation}
	We also have $v\in (\gamma^{(s)}_{\mathbf{X}, L^2}(\SU))'$ if and only if there exist constants $B>0$ and $K>0$ such that
	\begin{equation}\label{EQ:u-su2-b}
	\|e^{B\widehat{\slp}(l)^{\frac{1}{2s}}} \widehat{v}(l)\|_{\HS}\leq K  \textrm{ for all } l\in\frac12\N_0.
	\end{equation}
\end{conj}
We expect the proof of Conjecture \ref{THM:ultra-su2} to follow analogously to the proof of \cite[Theorem 2.5]{DR2014} based on the characterisation from Corollary \ref{COR:su2}.

We recall that on the Heisenberg group the symbol of the sub-Laplacian $\slp$ coincides with its infinitesimal representation $\pi_\lambda(\slp)$ given by \eqref{matrix1}. We also recall the Fourier transform $$\widehat{f}(\lambda)=\int_{\mathbb H_n} f(x)\pi_\lambda(x)^* dx,$$ which can be extended to ultradistributions. 
The Plancherel formula takes the form
$$
\int_{\mathbb H_n} |f(x)|^2 dx=c_n\int_{\R\backslash\{0\}} \|\widehat{f}(\lambda)\|_{\HS}^2 |\lambda|^n d\lambda,
$$
with an appropriate constant $c_n$.
Then similarly based on Theorem \ref{THM:Heis} on the Heisenberg group, we have 

\begin{conj}\label{THM:ultra-Hn}
	Let $1\leq s<\infty$. We have $u\in (\gamma^{s}_{\mathbf{X}, L^2}(\mathbb H_n))'$ if and only if for every $B>0$ we have
	\begin{equation}\label{EQ:u-Hn}
	\int_{\R\backslash\{0\}}\|e^{-B \pi_\lambda(\slp)^{\frac{1}{2s}}}\widehat{u}(\lambda)\|_{\HS}|\lambda|^n d\lambda <\infty.
	\end{equation}
	We also have $v\in (\gamma^{(s)}_{\mathbf{X}, L^2}(\mathbb H_n))'$ if and only if there exists $B>0$  such that
	\begin{equation}\label{EQ:u-Hn-b}
	\int_{\R\backslash\{0\}}\|e^{-B \pi_\lambda(\slp)^{\frac{1}{2s}}}\widehat{v}(\lambda)\|_{\HS}|\lambda|^n d\lambda <\infty.
	\end{equation}
\end{conj}


\chapter{Conclusion and future work}\label{CP:5}

The presentation of this thesis does not reflect the chronological developments of my Ph.D. In fact, I first started by studying the sub-Laplacians and trying to understand their Gevrey hypoellipticity, mainly under the invaluable supervision of Dr. V. Fischer. Once I became arguably confident\footnote{During one of my latest discussions with Dr. Fischer, while I was complaining, as usual, about the weakness of my knowledge in this field, she emphasised how `becoming an expert' has the drawback of becoming aware of the uncountable things that one does not know or, even worse, understand. } with the non-abelian environment and I started to appreciate the variety of observations and the pedantic attention needed to perform any calculation in the non-Euclidean environment, she argued that in a few months we would have been able to answer the Gevrey hypoellipticity questions. This has not been the case, since our questions turned out to be deeper and more complicated (and general) than we expected, even when restricted to the `simple' case $\Rn$. Nevertheless, the lack of straightforward answers made the topic itself even more interesting. After proving one implication, we decided to freeze the problem temporarily, due to time restrictions.

Instead we investigated `applications' of hypoelliptic operators, and Prof. Ruzhansky gave the suggestion of wave operators on Lie groups. In this case, everything worked smoothly and we are positive about the possibility to obtain further extensions and results.

From this I have concluded that negative results as well as positive ones offer ways to deepen one's knowledge. The contribution of this document in terms of mathematical achievements might not be pioneering in the general picture, but it has been fundamental for my personal growth and mathematical understanding.

\textbf{Future work}\\
There are still several points to investigate and the \textit{to-do-list} can be easily drawn. Regarding our work on weakly hyperbolic equations it would be interesting to:
\begin{itemize}
	\item Consider \textit{very weak solutions}  for the Cauchy problem for the wave equation. This means one imposes less regularity on the time-dependent coefficient $a(t)$. For example we can consider $a$ to be a distribution, or a measure. In this setting, the first necessary thing to do is to overcome the impossibility of multiplying distributions by smooth functions. This is achieved by introducing a new definition of solutions; of course it must be consistent with the existing ones. Such a definition has been introduced by Garetto and Ruzhansky in \cite{Garetto-Ruzhansky:ARMA}, where they considered wave-type equations in $\Rn$. Recently, Ruzhansky and Tokmagambetov \cite{RTok2017} reviewed this definition, adapting it to the abstract setting provided by operators with discrete spectrum.
	
	\item Consider the Cauchy problem for the wave equation with a coefficient dependent both on time and space variables, i.e. $a\equiv a(t,x)$.
	
	\item Recalling that we started from a partial differential equation and, after applying the group Fourier transform, we moved to an ordinary differential equation (ode), one can look for different versions of Proposition \ref{lemma} for odes. This could help to improve understanding of the well-posedness of Cauchy problems for different pdes.
	
	\item Investigate which conditions on lower-order terms might be added to the operator not to lose the well-posedness of the Cauchy problem in Theorem \ref{THM:main}.
	
	\item Understand whether the well-posedness results for the Cauchy problem in Theorem \ref{THM:main} are sharp, as they are in the Euclidean case, see \cite{CDS79}.
	
\end{itemize}

Regarding the study of Sub-Laplacian Gevrey spaces, the priority is a better investigation of the spaces of ultradistributions, and provision of rigorous proofs of Conjecture \ref{THM:ultra-su2} and \ref{THM:ultra-Hn}. Thinking in wider terms we would like to:
\begin{itemize}
	\item Study other groups. For example, following the case of the Heisenberg group, it may be possible to extend the characterisation to any stratified Lie group of step $2$. Indeed for these groups we have an  explicit formula for the symbol of sub-Laplacians, \cite{CRS2005,BFG2016}. If the characterisation of sub-Laplacian Gevrey spaces holds for these classes of groups, the natural conjecture will then be that the characterisation holds for stratified Lie groups of \textit{any} steps.
	
	\item Understand the cases of a sub-Laplacian on a manifold. Assuming we eventually prove our conjecture for stratified nilpotent Lie groups, we can then try to use the Rothschild--Stein machinery (introduced in Chapter \ref{CP:1}) to lift the problem from a manifold to a stratified nilpotent Lie group. 
	
	The first step in this direction would be to study the toy models given by the Heisenberg group and the complex sphere. In the case of the $3$-dimensional Heisenberg group $\h_1$ and complex sphere $\mathbb{S}^3$ (or equivalently $SU(2)$), an almost global lifting exists: it is given by the theory of contractions \cite{RR1986}. Therefore the beginning of the study of such a question would be to see whether the contraction of the sub-Laplacian Gevrey spaces on $SU(2)$ yields the ones on $\h_1$. Note that we have already characterised these spaces in this thesis, see Corollary \ref{COR:su2} for $SU(2)$ and Theorem \ref{THM:Heis} for $\h_n$.
	
	\item Investigate the possible local embeddings between the Euclidean and sub-Laplacian Gevrey spaces, most likely with different exponents $s$. This would allow us to provide satisfactory answers about the Gevrey-hypoellipticity of operators in the `canonical' sense. Indeed, given a manifold $M$ equipped with a measure $\mu$ and a H\"ormander system $\mathbf{X} =\{X_1,\dots,X_r\}$ with associated sub-Laplacian $\slp$, that we assume\footnote{In this way, the hypothesis of Theorem \ref{THM:implication} are fulfilled, and therefore $\gamma^s_{\mathbf{X},L^2}(M)\subset\big\{\phi\in\mathcal C^\infty(M)\,:\,\exists D>0\text{ such that }\|e^{D\slp^{\frac{1}{2s}}}\phi\|_{L^2}<\infty\big\}$.
	} to be non-negative and essentially self-adjoint on $L^2(M)$, we can argue that $\slp$ is hypoelliptic for $\gamma_{\mathbf{X},L^2}^s (M)$. This means that for every $f \in \mathcal{C}_0^\infty (M)$
	we have 
	\begin{align*}
	\slp f \in \gamma_{\mathbf{X},L^2}^s \implies f \in \gamma_{\mathbf{X},L^2}^s.
	\end{align*}
	This would follow easily from the characterisation of these spaces and the functional calculus of $\slp$. In fact if
	\begin{align*}
	\slp f \in \gamma_{\mathbf{X},L^2}^s
	\end{align*}
	according to Theorem \ref{THM:implication}, there exists a constant $D>0$ such that
	$$
	\| e^{D\slp^{\frac{1}{2s}}} \slp f\|_{L^2}<\infty.
	$$
	Observing that by functional calculus the operator $e^{(D-D’)\slp^{\frac{1}{2s}}}\slp$ is bounded on $L^2$, we can conclude that there exists a constant $D'>0$ such that 
	\begin{align*}
	\| e^{D'\slp^\frac{1}{2s}} f\|_{L^2}<\infty.
	\end{align*}
	Therefore, the idea is that a possible way to study the usual (Euclidean) Gevrey hypoellipticity of a sub-Laplacian is to study the various properties of the sub-Laplacian Gevrey spaces, especially their relations with the Euclidean ones.
	
\end{itemize}

In conclusion, starting from an abstract scenario, the work we carried out produced a pleasing piece of research with numerous possible further developments and a great potential of relevance to some difficult and important applications, especially Gevrey hypoellipticity of sub-elliptic operators (as observed in the Introduction).


\setstretch{1.2}


\clearpage

\begin{thebibliography}{99}
	
	\bibitem{BFG2016}
	H. Bahouri, C. Fermanian Kammerer, I. Gallagher.
	\emph{Dispersive estimates for the Schr\"odinger operator on step $2$ stratified Lie groups},
	Anal. PDE, $\mathbf{9}$ (2016), 545-574.
	
	\bibitem{Beals-Rockland}
	R.~Beals.
	\emph{Op\'erateurs invariants hypoelliptiques sur un groupe de Lie nilpotent},
	S\'eminaire \'Equations aux d\'eriv\'ees partielles (Polytechnique), (1976), 1-8.
	
	\bibitem{BLU} 
	A. Bonfiglioli, E. Lanconelli, F. Uguzzoni. 
	\emph{Stratified Lie groups and potential theory for their sub-Laplacians}.
	Springer Science \& Business Media, 2007.
	
	\bibitem{BT1997}
	A. Bove, D. Tartakoff. 
	\emph{Optimal non-isotropic Gevrey exponents for sums of squares of vector fields},
	Commun. Partial Diff. Eq., $\mathbf{22}$ (1997), 1263–1282.
	
	\bibitem{B2014}
	M. Bramanti.
	\emph{An Invitation to Hypoelliptic Operators and H\"ormander's Vector Fields},
	Springer, 2014.
	
	\bibitem{Bony-Schapira:analytic-IM-1972}
	J.-M. Bony, P.~Schapira.
	\emph{Existence et prolongement des solutions holomorphes des {\'e}quations
		aux d{\'e}riv{\'e}es partielles},
	\newblock {\em Invent. Math.}, $\mathbf{17}$ (1972), 95-105.
	
	\bibitem{Bronshtein:TMMO-1980}
	M.~D. Bron{\v{s}}te{\u\i}n.
	\emph{The {C}auchy problem for hyperbolic operators with characteristics of
		variable multiplicity},
	\newblock {\em Trudy Moskov. Mat. Obshch.}, $\mathbf{41}$ (1980), 83-99.
	
	\bibitem{CFS2010}
	J. Chen, D. Fan, L. Sun.
	\emph{Hardy space estimates for the wave equation on compact Lie groups},
	J. Funct. Anal., $\mathbf{259}$ (2010), 3230-3264. 
	
	\bibitem{CRS2005}
	P. Ciatti, F. Ricci, M. Sundari.
	\emph{Uncertainty inequalities on stratified nilpotent groups},
	Bull. Kerala Math. Assoc., Special issue (2005), 53-72.
	
	\bibitem{CHR08a}
	M. Cicognani, F. Hirosawa, M. Reissig.
	\emph{The Log-effect for p-evolution type models}, 
	{J. Math. Soc. Japan}, $\mathbf{60}$ (2008), 819-863.
	
	\bibitem{CHR08b}
	M. Cicognani, F. Hirosawa, M. Reissig. 
	\emph{Loss of regularity for p-evolution type models}, 
	J. Math. Anal. Appl.,$\mathbf{347}$ (2008), 35-58.
	
	\bibitem{CC10}
	M. Cicognani, F. Colombini.
	\emph{The Cauchy problem for p-evolution equations},
	Trans. Amer. Math. Soc., $\mathbf{362}$ (2010), 4853-4869.
	
	\bibitem{CZ1993}
	M. Cicognani, L. Zanghirati.
	\emph{On a class of unsolvable operators}, 
	Ann. Scuola Norm. Sup. Pisa, $\mathbf{20}$ (1993), 357–369.
	
	\bibitem{CS2014}
	G. Citti, A. Sarti.
	\emph{The constitution of visual perceptual units in the functional architecture of $V1$},
	arXiv:1406.0289, 2014.
	
	\bibitem{CDS79}
	F. Colombini, E. De Giorgi, S. Spagnolo.
	\emph{Sur les \'equations hyperboliques avec des coefficients qui ne d\'ependent que du temps},
	(French) Ann. Scuola Norm. Sup. Pisa Cl. Sci., $\mathbf 6$ (1979), 511-559. 
	
	\bibitem{Colombini-Jannelli-Spagnolo:Annals-low-reg}
	F.~Colombini, E.~Jannelli, S.~Spagnolo.
	\emph{Nonuniqueness in hyperbolic {C}auchy problems},
	Ann. of Math. (2), $\mathbf{126}$ (1987), 495-524. 
	
	\bibitem{CK2002} 
	F. Colombini, T. Kinoshita.
	\emph{On the Gevrey well-posedness of the Cauchy problem for weakly hyperbolic equations of higher order}, Journal of Differential Equations, $\mathbf{186}$ (2002), 394-419.
	
	\bibitem{CL}
	F. Colombini, N. Lerner.
	\emph{Hyperbolic operators with non-Lipschitz coefficients}, 
	Duke Math. J., $\mathbf{77}$ (1995), 657-698.
	
	\bibitem{CM}
	F. Colombini, G. M\'etivier.
	\emph{The Cauchy problem for wave equations with non Lipschitz coefficients; application to continuation of solutions of some nonlinear wave equations}, 
	Ann. Sci. \'Ec. Norm. Sup\'er., $\mathbf{41}$ (2008), 177-220.
	
	
	\bibitem{CS1982} 
	F. Colombini, S. Spagnolo.
	\emph{An example of a weakly hyperbolic Cauchy problem not well posed in $\mathcal{C}^\infty$}, 
	Acta Mathematica, $\mathbf{148}$ (1982), 243-253.
	
	
	\bibitem{CG90} 
	L. J. Corwin, F. P. Greenleaf.
	\emph{Representations of nilpotent Lie groups and their applications: Volume 1, Part 1, Basic theory and examples}, 
	Cambridge University Press, Cambridge, $\mathbf{18}$ (1990).
	
	\bibitem{CMZ1996}
	T. Coulhon, D. M\"uller, J. Zienkiewicz. 
	\emph{About Riesz transforms on the Heisenberg groups},
	Mathematische Annalen, $\mathbf{305}$ (1996), 369-79.
	
	\bibitem{DS1998} 
	P. D'Ancona, S. Spagnolo.
	\emph{Quasi-symmetrization of hyperbolic systems and propagation of the analytic regularity}, 
	Bollettino della Unione Matematica Italiana, $\mathbf{1 B}$ (1998), 169-186.
	
	\bibitem{DR2014}
	A. Dasgupta, M. Ruzhansky.
	\emph{Gevrey functions and ultradistributions on compact Lie groups and homogeneous spaces},
	Bulletin des Sciences Math\'ematiques, $\mathbf{138}$ (2014), 756-782.
	
	\bibitem{DR2016}
	A. Dasgupta, M. Ruzhansky.
	\emph{Eigenfunction expansions of ultradifferentiable functions and ultradistributions}, 
	Trans. Amer. Math. Soc., {\bf 368} (2016), 8481-8498.
	
	\bibitem{DE2009}
	A. Deitmar, S. Echterhoff.
	\emph{Principles of Harmonic Analysis},
	Springer, 2014.
	
	\bibitem{DFT2009}
	L. Desvillettes, G. Furioli, E. Terraneo.
	\emph{Propagation of Gevrey regularity for solutions of the Boltzmann equation for Maxwellian molecules},
	Trans. Amer. Math. Soc., $\mathbf{361}$ (2009), 1731-1747.
	
	\bibitem{DM78}
	J. Diximier, P. Malliavin.
	\emph{Factorisations de fonctions et de vecteurs ind\'efiniment diff\'erentiables},
	Bull. Sci. Math., $\mathbf{102}$ (1978), 307-330. 
	
	\bibitem{D1996}
	J.J. Duistermaat.
	\emph{Fourier Integral Operators},
	Birka\"auser Boston, 1996.
	
	\bibitem{FH1987}
	J. Faraut, K. Harzallah.
	\emph{Deux Cours d'Analyse Harmonique},
	Progr. Math., 69, Birkh\"auser Boston, Boston, MA, 
	1987.
	
	\bibitem{FR:Sobolev}
	V.~Fischer, M.~Ruzhansky. 
	\emph{Sobolev spaces on graded groups},
	to appear in Ann. Inst. Fourier, arXiv:1311.0192, 2013.
	
	\bibitem{FR2016} 
	V. Fisher, M. Ruzhansky.
	\emph{Quantization on Nilpotent Lie Groups}, 
	Progress in Mathematics,  Birkhauser, 2016.
	
	\bibitem{FRT20..} 
	V. Fischer, M. Ruzhansky, C. Taranto.
	\emph{On the sub-Laplacian Gevrey spaces}, preprint. 
	
	\bibitem{F1989}
	C. Foias, R. Temam. 
	\emph{Gevrey class regularity for the solutions of the Navier-Stokes equations},
	J. Funct. Anal., $\mathbf{87}$ (1989), 359–369.
	
	\bibitem{F1973}
	G. B. Folland.
	\emph{A fundamental solution for a subelliptic operator}, 
	Bull. Amer. Math. Soc., $\mathbf{79}$ (1973), 373-376.
	
	\bibitem{F75}
	G.~B. Folland.
	\emph{Subelliptic estimates and function spaces on nilpotent {L}ie groups},
	Ark. Mat., $\mathbf{13}$ (1975), 161-207.
	
	\bibitem{F1977}
	G. B. Folland.
	\emph{Applications of analysis on nilpotent groups to partial differential equations},
	Bull. Amer. Math. Soc., $\mathbf{83}$ (1977), 912-930.
	
	\bibitem{Fol} 
	G. B. Folland.
	\emph{Harmonic Analysis in Phase Space},
	Princeton, Princeton University Press, 2nd ed., 1989. 
	
	\bibitem{F1992}
	G. F. Folland.
	\emph{Fourier analysis and its applications. Vol. 4},
	AMS, 1992.
	
	\bibitem{F1994}	
	G. B. Folland.
	\emph{A Course in Abstract Harmonic Analysis},
	CRC Press, 1994.
	
	\bibitem{F2013}
	G. B. Folland.
	\emph{Real analysis: modern techniques and their applications},
	Harvard, John Wiley \& Sons, 2013.
	
	\bibitem{FS-CPAM} 
	G.~B.~Folland, E.~M.~Stein.
	\emph{Estimates for the $\overline{\partial_{b}}$ complex and analysis on the Heisenberg group},
	Comm. Pure Appl. Math., $\mathbf{27}$ (1974), 429-522.
	
	\bibitem{FS}
	G.~B. Folland, E.~M. Stein.
	\emph{Hardy spaces on homogeneous groups}, 
	Princeton University Press, 1982.
	
	\bibitem{GR2012}
	C. Garetto, M. Ruzhansky.
	\emph{On the well-posedness of weakly hyperbolic equations with time-dependent coefficients},
	J. Differential Equations, $\mathbf{253}$ (2012), 1317-1340.
	
	\bibitem{GR2013}
	C. Garetto, M. Ruzhansky.
	\emph{Weakly hyperbolic equations with non-analytic coefficients and lower order terms}, 
	Math. Ann., $\mathbf{357}$ (2013), 401-440.
	
	\bibitem{GR2015}
	C. Garetto, M. Ruzhansky.
	\emph{Wave equation for sum of squares on compact Lie groups},
	J. Differential Equations, $\mathbf{258}$ (2015), 4324-4347.
	
	\bibitem{Garetto-Ruzhansky:ARMA}
	C.~Garetto, M.~Ruzhansky.
	\emph{Hyperbolic second order equations with non-regular time dependent coefficients},
	Arch. Rational Mech. Anal., $\mathbf{217}$ (2015), 113-154.
	
	
	\bibitem{G1918}
	M. Gevrey.
	\emph{Sur la nature analytique des solutions des \'equations aux d\'eriv\'ees partielles. Premier m\'emoire}, Ann. Sci. Ecole Norm. Sup., $\mathbf{35}$ (1918), 129-190.
	
	\bibitem{G76}
	R. W. Goodman.
	\emph{Nilpotent Lie Groups: structure and applications to analysis},
	Springer-Verlag, Berlin Heidelberg, 1976.
	
	\bibitem{G1995}
	T. Gramchev, G. Popov.
	\emph{Nekhoroshev type estimates for billiard ball maps},
	Ann. Inst. Fourier, Grenoble, $\mathbf{45}$ (1995), 859-895. 
	
	
	\bibitem{H2003}
	B. Hall.
	\emph{Lie Groups, Lie Algebras, and Representations: An Elementary Introduction},
	Springer International Publishing,
	2003.
	
	
	\bibitem{HN-79}
	B.~Helffer, J.~Nourrigat.
	\emph{Caracterisation des op\'erateurs hypoelliptiques homog\`enes invariants \`a gauche sur un groupe de {L}ie nilpotent gradu\'e},
	Comm. Partial Differential Equations, $\mathbf{4}$ (1979), 899-958. 
	
	\bibitem{Helgason:wave-eqns-hom-spaces-1984}
	S.~Helgason.
	\emph{Wave equations on homogeneous spaces},
	Lie Group Representations III, (1984), 254-87.
	
	\bibitem{H1963}
	L. H\"ormander.
	\emph{Linear Partial Differential Operators},
	Springer-Verlag, Berlin Heidelberg, 1963.
	
	\bibitem{H1967} 
	L. H\"ormander.
	\emph{Hypoelliptic second order differential equations},
	Acta Mathematica, $\mathbf{119}$ (1967), Issue 1, 147-171.
	
	\bibitem{H1977}
	L. Hörmander. 
	\emph{The Cauchy problem for differential equations with double characteristics}, 
	Journal d’Analyse Mathématique, $\mathbf{32}$ (1977), 118-96.
	
	\bibitem{H1990}
	L. H\"ormander.
	\emph{The Analysis of Linear Partial Differential Operators I-IV},
	Springer-Verlag, Berlin, 1985. Vol. I: second rdition, 1990.
	
	\bibitem{HJL1985} A. Hulanicki, J. W. Jenkins, J. Ludwig.
	\emph{Minimum eigenvalues for positive, Rockland operators},
	Proc. Amer. Math. Soc., $\mathbf{94}$ (1985), 718-720.
	
	\bibitem{IP1974}
	V. Ivrii, V. Petkov. 
	\emph{Necessary conditions for the Cauchy problem for non-strictly hyperbolic equations to be well-posed}, Russian Mathematical Surveys, $\mathbf{29}$ (1974).
	
	\bibitem{J2014} 
	W. Johnston.
	\emph{The Weighted Hermite Polynomials Form a Basis for $L^2(\R)$},
	The American Mathematical Monthly, $\mathbf{121.3}$ (2014), 249-253.
	
	\bibitem{KS2006} 
	T. Kinoshita, S. Spagnolo.
	\emph{Hyperbolic equations with non-analytic coefficients},
	Math. Ann., $\mathbf{336}$ (2006), 551-569.
	
	\bibitem{K1934}
	A. Kolmogorov.
	\emph{Zufallige Bewegungen (Zur Theorie der Brownschen Bewegung)}, 
	Annals of Mathematics, $\mathbf{35}$ {(1934)}, No. 1, 116-117.
	
	\bibitem{K1931}
	A. Kolmogoroff.
	\emph{\"Uber die analytischen Methoden in der Wahrscheinlichkeitsrechnung},
	Mathematische Annalen, $\mathbf{104}$ {(1931)}, Issue 1, 415-458.
	
	\bibitem{KV2011}
	I. Kukavica, V. Vicol.
	\emph{On the analyticity and Gevrey-class regularity up to the boundary for the Euler equations},
	Nonlinearity, $\mathbf{24}$ (2011), 765-796.
	
	\bibitem{LMR2000}
	B. Laroche, P. Martin, P. Rouchon.
	\emph{Motion planning for the heat equation},
	Int. J. Robust Nonlinear Control, $\mathbf{102}$ (2000), 629-643.
	
	\bibitem{L1956}
	A. Lax.
	\emph{On Cauchy's problem for partial differential equations with multiple characteristics},
	Communications on Pure and Applied Mathematics, $\mathbf{9}$ (1956), number 2, 135-169.
	
	\bibitem{LP2004}
	F. Lust-Piquard. 
	\emph{Riesz transforms on generalized Heisenberg groups and Riesz transforms associated to the CCR heat flow}, 
	Publicacions Matemàtiques, $\mathbf{48}$ (2004), 309-33.
	
	\bibitem{MR1997}
	M. Mascarello, L. Rodino.
	\emph{Partial differential equations with multiple characteristics},
	Math. Topics, Berlin: Akademie Verlag, 1997.
	
	\bibitem{Melrose:wave-subelliptic-1986}
	R.~Melrose.
	\emph{Propagation for the wave group of a positive subelliptic second-order differential operator},
	Hyperbolic equations and related topics (Katata/Kyoto, 1984), (1986), 181-92.
	
	\bibitem{Miller:80}
	K.~G. Miller.
	\emph{Parametrices for hypoelliptic operators on step two nilpotent Lie groups},
	Comm. Partial Differential Equations, $\mathbf{5}$ (1980), 1153-1184.
	
	\bibitem{M1985}
	S. Mizohata.
	\emph{On the Cauchy problem}, 
	Acad. Press \& Sci. Press Beijing, 1985. 
	
	\bibitem{Muller-Stein:Lp-wave-Heis}
	D.~M{{\"u}}ller, E.~M. Stein.
	\emph{{$L^p$}-estimates for the wave equation on the {H}eisenberg group},
	Rev. Mat. Iberoamericana, $\mathbf{15}$ (1999), 297-334.
	
	\bibitem{M1993}
	D. Mumford.
	\emph{Elastica and computer vision},
	Algebraic Geometry and its Applications, 1993, 507-518.
	
	\bibitem{Nachman:wave-Heisenberg-CPDE-1982}
	A.~I. Nachman.
	\emph{The wave equation on the {H}eisenberg group},
	Comm. Partial Differential Equations, $\mathbf{7}$ (1982), 675-714.
	
	\bibitem{Nishitani:BSM-1983}
	T.~Nishitani.
	\emph{Sur les {\'e}quations hyperboliques {\`a} coefficients h{\"o}ld{\'e}riens en {$t$} et de classe de {G}evrey en {$x$}},
	Bull. Sci. Math. (2), $\mathbf{107}$ (1983), 113-138.
	
	\bibitem{O1970}
	O. A. Oleinik.
	\emph{On the Cauchy Problem for Weakly Hyperbolic Equations},
	Comm. on pure and app. Math., $\mathbf{XXIII}$ (1970), 569-586.
	
	\bibitem{PP2004}
	C. Parenti, A. Parmeggiani. 
	\emph{The Cauchy problem for certain systems with double characteristics},
	Osaka J. Math., $\mathbf{41}$ (2004), 659-680.
	
	\bibitem{PP2009}
	C. Parenti, A. Parmeggiani. 
	\emph{On the Cauchy problem for hyperbolic operators with double characteristics}, 
	Communications in Partial Differential Equations, $\mathbf{347}$ (2009), 837-88.
	
	\bibitem{Perelman1}
	G. Perelman.
	\emph{The entropy formula for the Ricci flow and its geometric applications},
	arXiv: math.DG/0211159v1, 11 Nov 2002.
	
	\bibitem{Perelman2}
	G. Perelman.
	\emph{Ricci flow with surgery on three-manifolds},
	arXiv: math.DG/0303109, 10 Mar 2003.
	
	\bibitem{Perelman3}
	G. Perelman.
	\emph{Finite extinction time for the solutions to the Ricci flow on certain three-
		manifolds},
	arXiv: math.DG/0307245, 17 Jul 2003.
	
	\bibitem{Poincare}
	H. Poincar\'e.
	\emph{\OE uvres}, Tome VI, Gauthier--Villars, Paris, 1953.
	
	\bibitem{P2008}
	R. Ponge.
	\emph{Heisenberg calculus and spectral theory of hypoelliptic operators on Heisenberg manifolds},
	AMS, 2008.
	
	\bibitem{RS1996}
	J.P. Ramis, R. Schäfke.
	\emph{Gevrey separation of fast and slow variables},
	Nonlinearity , $\mathbf{9}$ (1996), 353–384.
	
	\bibitem{Ric}
	F. Ricci.
	\emph{Sub-Laplacians on nilpotent Lie groups},
	unpublished lecture notes accessible on webpage \url{http://homepage.sns.it/fricci/corsi.html}.
	
	\bibitem{R2001}
	F. Ricci.
	\emph{Analisi armonica non commutativa},
	unpublished lecture notes accessible on webpage \url{http://homepage.sns.it/fricci/papers/annoncomm.pdf}.
	
	\bibitem{R2014}
	F. Ricci.
	\newblock \emph{Appunti per il corso: Analisi Armonica su gruppi di Lie: nozioni fondamentali ed esempi},
	unpublished lecture notes accessible on webpage \url{https://serse.sns.it/sns/_DOCs/CRS_MATERIALE_DIDATTICO/f7680534c90356a455dfcc659dcd053b.pdf}.
	
	\bibitem{RR1986}
	F. Ricci, R. L. Rubin.
	\emph{Transferring Fourier multipliers from $SU(2)$ to the Heisenberg group},
	American Journal of Mathematics, $\mathbf{108}$ (1986), 571-88.
	
	\bibitem{R1993}
	L. Rodino.
	\emph{Linear partial differential operators in Gevrey spaces},
	World Scientific, Singapore, 1993.
	
	\bibitem{RS1980} 
	M. Reed, B. Simon.
	\emph{Methods of Modern Mathematical Physics, Vol. $1$: Functional Analysis, revised and enlarged edition},
	Academic Press, 1980.
	
	\bibitem{Rockland}
	C.~Rockland.
	\emph{Hypoellipticity on the {H}eisenberg group-representation-theoretic criteria},
	Trans. Amer. Math. Soc., $\mathbf{240}$ (1978), 1-52.
	
	\bibitem{Rothschild-Stein:AM-1976}
	L.~P. Rothschild, E.~M. Stein.
	\emph{Hypoelliptic differential operators and nilpotent groups},
	Acta Math., $\mathbf{137}$ (1976), 247-320.
	
	\bibitem{RT2017}
	M. Ruzhansky, C. Taranto.
	\emph{Time-dependent wave equations on graded groups},
	arXiv:1705.03047, 2017.
	
	\bibitem{RTok2017}
	M. Ruzhansky, N. Tokmagambetov.
	\emph{Wave equation for operators with discrete spectrum and irregular propagation speed},
	Arch. Ration. Mech. Anal., to appear. arXiv:1705.01418, 2017.
	
	
	\bibitem{RT2009}  
	M. Ruzhansky, V. Turunen.
	\emph{Pseudo-Differential Operators and Symmetries: Background Analysis and Advanced Topics},
	Basel, Birkhauser,
	2009.
	
	\bibitem{RT2013}
	M. Ruzhansky, V. Turunen.
	\emph{Global quantization of pseudo-differential operators on compact Lie groups, SU(2) and 3-sphere}, Int. Math. Res. Not. IMRN 2013, 2439-2496.
	
	\bibitem{SSS1991}
	A. Seeger, C. D. Sogge, E. M. Stein.
	\emph{Regularity properties of Fourier integral operators},
	Ann. of Math., $\mathbf{134}$ (1991), 231-251.
	
	\bibitem{S1969}
	R.~T. Seeley.
	\emph{Eigenfunction expansions of analytic functions.}
	Proc. Amer. Math. Soc., $\mathbf{21}$ (1969), 734-738.
	
	\bibitem{S1973}
	L.M. Stein.
	\emph{Some problems in harmonic analysis suggested by symmetric spaces and semi-simple group},
	Proc. Internat. Congress Math., Nice, $\mathbf{1}$ (1970), Gauthier-Villars, Paris, (1971), 173-189.
	
	\bibitem{SS2003}
	L. M. Stein, R. Shakarchi.
	\emph{Fourier Analysis - An Introduction},
	Princeton University Press, 2003.
	
	\bibitem{S1975} 
	G. Szeg\"o.
	\emph{Orthogonal polynomials}, 
	AMS, fourth edition, 1975. 
	
	\bibitem{T2008}
	L. A. Takhtadzhian.
	\emph{Quantum mechanics for mathematicians},
	AMS, 2008.
	
	
	\bibitem{T1986} 
	M. E. Taylor.
	\emph{Noncommutative harmonic analysis} (No. $22$),
	AMS, 1986.
	
	\bibitem{tER:97}
	A.~F.~M. ter Elst, D.~W. Robinson.
	\emph{Spectral estimates for positive Rockland operators},
	Algebraic groups and Lie groups, $\mathbf{9}$ (1997), 195-213.
	
	\bibitem{ER1999}
	A. F. M. ter Elst, D. W. Robinson, A. Sikora.
	\emph{Riesz transforms and Lie groups of polynomial growth},
	J. Funct. Anal., $\mathbf{162}$ (1999), 14-51.
	
	\bibitem{tER2012}
	A.~F.~M. ter Elst, D.~W. Robinson.
	\emph{Analysis on Lie groups with polynomial growth},
	Springer Science \& Business Media, 2012.
	
	\bibitem{T1998}
	S. Thangavelu.
	\emph{Harmonic analysis on the Heisenberg group},
	Boston, Birkh\"auser, 1998.
	
	\bibitem{VSC1992}
	N. T. Varopoulos, L. Saloff-Coste, T. Coulhon.
	\emph{Analysis and geometry on groups}, Cambridge University Press, 1992. 
	
\end{thebibliography}
 \end{document}